\documentclass[11pt,reqno]{amsart}

\usepackage{amsmath, amsthm, amssymb}
\usepackage{amsfonts}

\usepackage{hyperref}

\usepackage[ansinew]{inputenc}
\usepackage[dvips]{epsfig}
\usepackage{graphicx}
\usepackage[english]{babel}

\usepackage{thmtools}
\theoremstyle{plain}
\declaretheorem[title=Theorem, parent=section]{theorem}
\declaretheorem[title=Lemma,sibling=theorem]{lemma}
\declaretheorem[title=Proposition,sibling=theorem]{proposition}
\declaretheorem[title=Corollary,sibling=theorem]{corollary}
\declaretheorem[title=Question,sibling=theorem]{question}

\theoremstyle{definition}
\declaretheorem[title=Definition,sibling=theorem]{definition}
\declaretheorem[title=Remark,sibling=theorem]{remark}
\declaretheorem[title=Remark, numbered=no]{remark*}

\declaretheorem[title=Assumption, numbered=no]{assumption*}

\numberwithin{equation}{section}

\usepackage[backgroundcolor=white, bordercolor=blue,
linecolor=blue]{todonotes}

\usepackage{dsfont}
\usepackage{bbm}

\newcommand{\N}{\mathbb{N}}
\newcommand{\R}{\mathbb{R}}

\newcommand{\cE}{\mathcal{E}}

\newcommand{\eps}{\varepsilon}

\newcommand{\loc}{\mathrm{loc}}
\newcommand{\1}{\mathbbm{1}}

\DeclareMathOperator{\dist}{dist}

\DeclareMathOperator{\diam}{diam}

\DeclareMathOperator{\dvg}{div}

\DeclareMathOperator{\tail}{Tail}

\renewcommand{\d}{\textnormal{\,d}}

\newcommand{\average}{{\mathchoice {\kern1ex\vcenter{\hrule height.4pt
width 6pt depth0pt} \kern-9.7pt} {\kern1ex\vcenter{\hrule
height.4pt width 4.3pt depth0pt} \kern-7pt} {} {} }}
\newcommand{\dashint}{\average\int}

\begin{document}
\allowdisplaybreaks
\title[Optimal boundary regularity and Green function estimates]{Optimal boundary regularity and Green function estimates for nonlocal equations in divergence form}

\author{Minhyun Kim}
\address{Department of Mathematics \& Research Institute for Natural Sciences, Hanyang University, 04763 Seoul, Republic of Korea}
\email{minhyun@hanyang.ac.kr}
\urladdr{https://sites.google.com/view/minhyunkim/}

\author{Marvin Weidner}
\address{Institute for Applied Mathematics, University of Bonn, Endenicher Allee 60, 53115, Bonn, Germany}
\email{mweidner@uni-bonn.de}
\urladdr{https://sites.google.com/view/marvinweidner/}



%\allowdisplaybreaks

\allowdisplaybreaks

\begin{abstract}
In this article we prove for the first time the $C^s$ boundary regularity for solutions to nonlocal elliptic equations with H\"older continuous coefficients in divergence form in $C^{1,\alpha}$ domains. So far, it was only known that solutions are H\"older continuous up to the boundary, and establishing their optimal regularity has remained an open problem in the field. Our proof is based on a delicate higher order Campanato-type iteration at the boundary, which we develop in the context of nonlocal equations and which is quite different from the local theory.

As an application of our results, we establish sharp two-sided Green function estimates in $C^{1,\alpha}$ domains for the same class of operators. Previously, this was only known under additional structural assumptions on the coefficients and in more regular domains.
\end{abstract}
\maketitle
%\tableofcontents

\section{Introduction}

The aim of this work is to establish sharp two-sided Green function estimates and the optimal $C^s$ boundary regularity for solutions to nonlocal elliptic equations posed in $C^{1,\alpha}$ domains. We are interested in integro-differential operators
\begin{align}\label{eq-op}
Lu(x) = 2\, \text{p.v.} \int_{\R^n} (u(x) - u(y)) K(x,y) \d y
\end{align}
with kernels $K : \R^n \times \R^n \to [0,\infty]$ satisfying a natural symmetry and uniform ellipticity condition of the form
\begin{align}
\label{eq:Kcomp}
\lambda |x-y|^{-n-2s} \le K(x,y) = K(y,x) \le \Lambda |x-y|^{-n-2s} 
\end{align}
for some $s \in (0,1)$ with $2s < n$ and $\Lambda \geq \lambda > 0$. We assume throughout this paper that the kernels $K$ are locally H\"older continuous of order $\sigma \in (0,s)$ in the following sense: given a set $\mathcal{A} \subset \R^n$, (which will typically be the domain where the equation is posed)
\begin{align}\label{eq-K-cont} 
|K(x+h, y+h) - K(x, y)| \leq \Lambda \frac{|h|^{\sigma}}{|x-y|^{n+2s}} \quad\text{for all } x,y \in \mathcal{A}\text{ and } h \in B_1.
\end{align}
Due to the symmetry of $K$, such operators $L$ are naturally associated to variational functionals and often referred to as \textit{nonlocal operators in divergence form}. They can be considered as nonlocal counterparts of second-order operators in divergence form $-\dvg(A(x)\nabla u(x))$ with $A \in C^{\sigma}$.

\subsection{Two-sided Green function estimates}

Solutions to equations of the form
\begin{align}
\label{eq:dirichlet-intro}
\left\{
\begin{aligned}
    L u &= f &&\text{in } \Omega,\\
    u &= 0 &&\text{in } \R^n \setminus \Omega,
\end{aligned}
\right.
\end{align}
for some bounded open set $\Omega \subset \R^n$, enjoy the following representation formula:
\begin{align*}
    u(x) = \int_{\Omega} G(x,y) f(y) \d y,
\end{align*}
where $G : \R^n \times \R^n \to [0,\infty]$ denotes the Green function associated to $L$ and $\Omega$. The Green function is an important object in partial differential equations, potential analysis, and probability theory. It is a crucial task in all of these fields to gain a good understanding of fine properties of $G$. 

The study of the Green function for nonlocal operators $L$ was initiated in the works \cite{Rie38,BGR61,Lan72} (see also \cite{Buc16}), where an explicit formula was derived for the Green function of the fractional Laplacian in the unit ball. An important extension of these results was carried out in the influential papers \cite{ChSo98,Kul97} (see also \cite{CKS10}), where the following sharp two-sided estimates for the Green function of the fractional Laplacian were established in general bounded $C^{1,1}$ domains:
    \begin{align}
    \label{eq:green-est}
        c^{-1}\left( \frac{d_{\Omega}(x)}{|x-y|} \wedge 1 \right)^s\left( \frac{d_{\Omega}(y)}{|x-y|} \wedge 1 \right)^s \le \frac{G(x,y)}{|x-y|^{2s-n}} \le c \left( \frac{d_{\Omega}(x)}{|x-y|} \wedge 1 \right)^s\left( \frac{d_{\Omega}(y)}{|x-y|}  \wedge 1 \right)^s,
    \end{align}
where $c=c(n, s, \Omega)\geq1$ and $d_{\Omega}(x) = \dist(x,\R^n \setminus \Omega)$.

There are two natural ways in which the results of \cite{ChSo98,Kul97,CKS10} can be generalized:
\begin{itemize}
    \item[(i)] Establish the bounds \eqref{eq:green-est} for more general nonlocal operators than $L = (-\Delta)^s$.
    \item[(ii)] Establish the bounds \eqref{eq:green-est} in more general domains than $\partial \Omega \in C^{1,1}$.
\end{itemize}

In the last 15 years, the problems (i) and (ii) have attracted a lot of attention and a huge amount of research has been devoted to generalizing the results in \cite{ChSo98,Kul97,CKS10}. Significant contributions to this area of research include \cite{CKS12,CKS14,BGR14,KiMi18}, \cite{KiKi14,GKK20,CKW22}, \cite{KSV25,CKSV24,KSV24}, and \cite{CHZ25}, which mainly use tools from probability theory and potential analysis. However, despite significant progress, previous results leave large parts of (i) and (ii) unanswered, as we explain in more detail below.

In this article, we fully resolve both questions (i) and (ii) at the same time, using an approach which is entirely based on analytic techniques. In fact, we establish two-sided global estimates of the Green function associated with a general nonlocal operator $L$ in divergence form \eqref{eq-op}, \eqref{eq:Kcomp}, \eqref{eq-K-cont} in $C^{1,\alpha}$ domains. 

Our first main result reads as follows:

\begin{theorem}
\label{thm:main-1}
    Let $s \in (0,1)$ and $\alpha,\sigma \in (0,s)$. Let $\Omega \subset \R^n$ be a bounded $C^{1,\alpha}$ domain and let $L,K, \lambda, \Lambda$ be as in \eqref{eq-op}, \eqref{eq:Kcomp}. Assume that \eqref{eq-K-cont} holds with $\mathcal{A} \Supset \Omega$. Let $G$ be the Green function associated with $L$ and $\Omega$. Then, the estimate \eqref{eq:green-est} holds true with a constant $c > 0$ depending only on $n,s,\lambda,\Lambda,\alpha,\sigma$, and $\Omega$.
\end{theorem}

Let us comment on the novelty of \autoref{thm:main-1}.

\begin{itemize}
    \item We establish \eqref{eq:green-est} for the first time in $C^{1,\alpha}$ domains with $\alpha \in (0,s)$. This generality appears to be new even for the fractional Laplacian. Previous results consider $C^{1,1}$ domains (and their proofs rely on the domain satisfying the interior and exterior ball condition), except for \cite{KiKi14}, where $C^{1,s+\eps}$ domains are treated.
    \item Our result seems to be new already for translation invariant kernels, i.e.\ when $K$ only depends on $x-y$. Previous articles either assume rotational symmetry of the kernels (see \cite{CKS14,BGR14,KiMi18}) or work under additional structural assumptions on the coefficients, which allow to rewrite $L$ as a homogeneous non-divergence form operator with a lower order term\footnote{They assume $K(x,y) := \kappa(x,y)|x-y|^{-n-2s}$, where $|\kappa(x,y) - \kappa(x,x)| \le C |x-y|^{\theta}$ for some $\theta \in (s,1)$. In this case, one can write $K = K_1 + K_2$, where $K_1(x,y) = \kappa(x,x) |x-y|^{-n-2s}$ is homogeneous and $K_2$ has a lower order singularity.} (see \cite{KiKi14,Kim15,GKK20,KSV25,CKSV24,KSV24}). This is not possible in our setup.
    \item Our assumptions on $\partial \Omega$ and $K$ are optimal in the sense that \eqref{eq:green-est} fails to hold in $C^1$ or Lipschitz domains and also when $K$ is not H\"older continuous. In these cases, solutions are in general not $C^s$ up to the boundary.
\end{itemize}

The results in \cite{CKS14,BGR14,KiMi18,CHZ25} include L\'evy operators which might be of variable order. Such operators have kernels which satisfy two-sided bounds as \eqref{eq:Kcomp} but with another asymptotic behavior near the diagonal or the boundary. The articles \cite{KSV25,CKSV24,KSV24} focus on kernels that might decay or explode at the boundary of $\Omega$. Such properties imply Green function estimates of a different shape than \eqref{eq:green-est}. The study of such kernels goes beyond the scope of this article. Moreover, several of the aforementioned articles also establish estimates for the Dirichlet heat kernel, i.e.\ for parabolic problems. It is an intriguing problem for future work to establish Dirichlet heat kernel estimates in our general setup \eqref{eq-op}, \eqref{eq:Kcomp}, and \eqref{eq-K-cont}.

Moreover, as an application of \autoref{thm:main-1}, we prove gradient estimates for the Green function.

\begin{corollary}
\label{cor:gradient-bounds}
Let $s \in (\frac{1}{2},1)$ and $\alpha,\sigma \in (0,s)$. Let $\Omega \subset \R^n$ be a bounded $C^{1,\alpha}$ domain and let $L,K, \lambda, \Lambda$ be as in \eqref{eq-op}, \eqref{eq:Kcomp}. Assume that \eqref{eq-K-cont} holds with $\mathcal{A} \Supset \Omega$. Let $G$ be the Green function associated with $L$ and $\Omega$. Then
\begin{align*}
|\nabla_x G(x,y)| 
\le c \max\left\{ |x-y|^{-1} , d_{\Omega}(x)^{-1} \right\} G(x,y)  ~~ \forall x,y \in \Omega,
\end{align*}
where $c > 0$ depends only on $n,s,\lambda,\Lambda,\alpha,\sigma,\gamma$, and $\Omega$.
\end{corollary}

Gradient estimates for the Green function were established in \cite{ BKN02,KuRy16,GJZ22} for a class of nonlocal translation invariant operators. Interior gradient estimates for the Green function associated to a nonlocal operator \eqref{eq-op} with \eqref{eq:Kcomp}, \eqref{eq-K-cont} follow for instance from the gradient potential estimates in \cite{KNS22}. However, in this paper we provide an approach that is completely different from \cite{KNS22} and our result also holds true up to the boundary.
%Moreover, \autoref{cor:gradient-bounds} seems to be entirely new, when $s \le \frac{1}{2}$. In this case, gradient potential estimates are not expected to hold.

\begin{remark}
    In principle, our proof of \autoref{cor:gradient-bounds} remains valid for $s \le \frac{1}{2}$. The only reason why we exclude this case is due to the lack of an appropriate higher order interior regularity estimate in the literature. We refer to \autoref{remark:Green-gradient} for a more detailed discussion on this matter.
\end{remark}

As another application of \autoref{thm:main-1}, we establish Poisson kernel estimates for the same class of operators in $C^{1,\alpha}$ domains (see \autoref{cor:Poisson}). Moreover, in Section \ref{sec:Green} we prove upper Green function estimates in Lipschitz (and more general) domains (see \autoref{remark:Green-function-general-domains}). Note that similar results, in slightly different setups, have already been proved in \cite{Bog00,Jak02} (see also \cite{BGR10,ABR25}).

\begin{remark}
    Our main results \autoref{thm:main-1} and \autoref{cor:gradient-bounds} have several direct applications. 
    \hspace{-0.8cm}
   \begin{itemize}
       \item Pointwise global estimates for the Green function and its gradient are the main tool in \cite{AFLY24}, where Calder\'on--Zygmund boundary regularity of solutions to \eqref{eq:dirichlet-intro} is established in case $L = (-\Delta)^s$ in $C^2$ domains. Using our estimates, one can extend the results in \cite{AFLY24} to more general operators and rougher domains.
       \item Boundary Green function estimates are a common tool to establish the sharp boundary behavior of solutions to nonlocal porous medium or fast diffusion equations in domains (see for instance \cite{BoVa15,BFV18,BII23}). \autoref{thm:main-1} directly verifies one of the main assumptions in these articles and extends the class of admissible operators in their results.
   \end{itemize} 
\end{remark}

The Green function was formally defined as an analytic object in \cite{KLL23}, where also existence and uniqueness were established. In this paper, we use the same definition of $G$ (see also \autoref{def:Green}). 
The estimate \eqref{eq:green-est} in our main result \autoref{thm:main-1} reflects both, the behavior of the Green function at the boundary of $\Omega$, and also near the diagonal $\{ x= y \}$.  It was already shown in \cite{KKL23} (see also \cite{CaSi18}) that $G$ enjoys the following bounds near the diagonal
\begin{align}
\label{eq:int-upper}
    G(x,y) &\le c |x-y|^{-n+2s} ~~\quad \forall x,y \in \Omega,\\
    \label{eq:int-lower}
    G(x,y) &\ge c^{-1} |x-y|^{-n+2s} ~~ \forall x,y \in \Omega ~\text{ with } ~ |x-y| \le \min\{ d_{\Omega}(x) , d_{\Omega}(y) \}
\end{align}
for some constant $c \geq 1$, depending only on $n,s,\lambda$, and $\Lambda$.
These estimates are interior in nature, since they do not contain any information on the growth of $G$ near the boundary $\partial \Omega$. However, note that the constant $c \geq 1$ is also independent of $\Omega$, and in fact, no assumption needs to be made on the regularity of $\partial\Omega$ and $K$, apart from ellipticity \eqref{eq:Kcomp}.

Our proof of \eqref{eq:green-est} makes use of the near diagonal bounds \eqref{eq:int-upper} and \eqref{eq:int-lower}. The main ingredients in order to obtain the sharp boundary behavior of $G$ are the $C^s$ regularity and the Hopf lemma for solutions to \eqref{eq:dirichlet-intro}, which we establish here for the first time. We explain these results in the sequel.

\subsection{Optimal boundary regularity}

The regularity theory for nonlocal equations in divergence form \eqref{eq:dirichlet-intro} has been extensively studied in recent years and in particular the interior regularity of solutions is by now well understood. A nonlocal counterpart of the De Giorgi--Nash--Moser theory has been developed in \cite{BaLe02,Kas09,DyKa20,KaWe24} and \cite{DKP14,DKP16,Coz17}, where interior $C^{\gamma}$ regularity and a Harnack inequality are established.
Moreover, under additional regularity assumptions on $K$, higher order interior regularity of Schauder-, Cordes--Nirenberg-, and Calder\'on--Zygmund-type has been proved for instance in \cite{Coz17b}, \cite{Brli17,BLS18}, \cite{Now20,Now21,Now23}, \cite{Fal20}, \cite{MSY21,FMSY22}, \cite{FeRo24b}, \cite{KNS22}, \cite{DiNo25}, and \cite{Loh23}.

In contrast to the interior regularity, the boundary behavior of solutions to \eqref{eq:dirichlet-intro} is much less understood. In fact, until now it was only known that solutions to \eqref{eq:dirichlet-intro} with H\"older continuous kernels $K$ \eqref{eq:Kcomp} are $C^{\gamma}$ up to the boundary for some small $\gamma$, even when $\Omega = B_1$ (see \cite{KKP16, KiLe23}).

For the fractional Laplacian $L = (-\Delta)^s$, the study of the boundary behavior was initiated in the works \cite{ChSo98,CKS10,RoSe14}. More general translation invariant kernels were considered in \cite{Gru15,RoSe16,RoSe16b,RoSe17,AbRo20}, and finally in \cite{RoWe24},  where it was shown that solutions to \eqref{eq:dirichlet-intro} are $C^s(\overline{\Omega})$ when $\partial \Omega \in C^{1,\alpha}$ and $K$ is a translation invariant kernel satisfying \eqref{eq:Kcomp}. The $C^s(\overline{\Omega})$ regularity is optimal, even for the fractional Laplacian $(-\Delta)^s$ in smooth domains. This phenomenon is purely nonlocal and distinguishes the boundary behavior for nonlocal equations from that for local equations.

The following natural question has remained open since the pioneering works \cite{RoSe14,Gru15,RoSe16b} on the boundary regularity for (translation invariant) nonlocal equations:
\begin{question}\label{question}
Does the $C^s(\overline{\Omega})$ boundary regularity for \eqref{eq:dirichlet-intro} remain true when $L$ has (H\"older continuous\footnote{Note that without any regularity assumption on $K$, the $C^{\gamma}$ boundary regularity for some small $\gamma \in (0,s)$ is optimal.}) coefficients?
\end{question}
One central purpose of this paper is to answer this question affirmatively for nonlocal operators $L$ in divergence form \eqref{eq-op} and \eqref{eq:Kcomp}.

Note that for operators $L$ in non-divergence form, i.e.\ when $K(x,x+h) = K(x,x-h)$, there exist several works investigating the fine boundary regularity, see \cite{Gru15,AbGr23} and the references therein. Moreover, in this case one can still evaluate the operator in the classical sense for smooth functions. Since the proofs in most of the previously mentioned articles (see for instance \cite{RoSe16,RoSe16b,RoSe17,RoWe24}) addressing the translation invariant case heavily rely on the construction of explicit barrier functions, one could resolve \autoref{question} for non-divergence form operators by similar techniques. 

However, for operators in divergence form \eqref{eq:Kcomp}, it is not possible to compute $Lv(x)$ in the classical sense, even if $v \in C_c^{\infty}(\R^n)$. Consequently, proving fine boundary regularity results becomes significantly more challenging. 

\autoref{question} has remained \textit{wide open for divergence form operators}. The only existing result in the literature is \cite{Fal19}, which holds true under the restrictive assumption that $L d_{\Omega}^s$ can be evaluated in the classical sense and has nice enough regularity properties. Moreover, we mention \cite{IMS16,IMS20,IaMo24}, where fine boundary regularity results for the fractional $p$-Laplacian were established. All of these articles heavily rely on the construction of barrier functions, too. Hence, significant new ideas would be required to adapt their techniques to operators with coefficients.

In this article, we finally answer \autoref{question} affirmatively. Our main result establishes the optimal $C^s$ boundary regularity and a Hopf lemma for nonlocal operators \eqref{eq-op}, \eqref{eq:Kcomp}, \eqref{eq-K-cont}, and it reads as follows. 

\begin{theorem}
\label{thm:main-2}
    Let $s \in (0,1)$ and $\alpha,\sigma \in (0,s)$. Let $\Omega \subset \R^n$ be a $C^{1,\alpha}$ domain and let $L,K, \lambda, \Lambda$ be as in \eqref{eq-op}, \eqref{eq:Kcomp}. Assume that \eqref{eq-K-cont} holds with $\mathcal{A} = B_1$. Let $f$ be such that $d_{\Omega}^{s-p} f \in L^q(\Omega \cap B_1)$ for some $p \in (0,s]$ and $q \in (\frac{n}{p},\infty]$. Let $u$ be a weak solution to
    \begin{align*}
    \left\{
    \begin{aligned}
    Lu&=f &&\text{in }\Omega \cap B_1, \\
    u&=0 &&\text{in } B_1 \setminus \Omega.
    \end{aligned}
    \right.
    \end{align*}
    Then, $u \in C^s_{\mathrm{loc}}(\overline{\Omega} \cap B_{1})$ and
    \begin{align*}
        \Vert u \Vert_{C^s(\overline{\Omega \cap B_{1/2}})} \le c \left( \Vert u \Vert_{L^1_{2s}(\R^n)} + \Vert d^{s-p}_{\Omega} f \Vert_{L^q(\Omega \cap B_1)} \right)
    \end{align*}
    for some constant $c > 0$, depending only on $n,s,\lambda,\Lambda,\alpha,\sigma,p,q$, and $\Omega$.

    Moreover, if $u \ge 0$ in $\R^n$ and $f \ge 0$ in $\Omega \cap B_1$ for $f \in L^\infty(\Omega \cap B_1)$, then either $u \equiv 0$ or 
    \begin{align}
    \label{eq:Hopf-intro}
        u \ge c d^s_{\Omega} ~~ \text{ in } \Omega \cap B_{1/2}
    \end{align}
    for some $c > 0$.
\end{theorem}

Note that \autoref{thm:main-2} is a localized result in the sense that it only requires information about $\Omega,f$, and the solution in $B_1$. Moreover, we only require H\"older continuity of $K$ in \eqref{eq-K-cont} with $\mathcal{A} = B_1$. Considering localized problems instead of \eqref{eq:dirichlet-intro} is an important feature of this theorem and we use it heavily in the proof of the Green function estimates in \autoref{thm:main-1}.

As we explained before, the $C^s$ regularity up to the boundary and the boundary growth \eqref{eq:Hopf-intro} are optimal, even for the fractional Laplacian. \autoref{thm:main-2} is the first higher order regularity result for the general class of nonlocal operators in divergence form \eqref{eq-op}, \eqref{eq:Kcomp}, and \eqref{eq-K-cont}. In particular, we do not assume any structural assumptions on the coefficients, such as homogeneity or additional (non-divergence type) symmetry of the kernels. Recall that prior to our result it was only known that solutions are $C^{\gamma}$ for some $\gamma \in (0,s)$ (see \cite{KKP16,KiLe23}). 

The Hopf lemma \eqref{eq:Hopf-intro} seems to be entirely new for nonlocal operators in divergence form. Note that even in the local case, Hopf lemmas  have only been established quite recently for operators in divergence form with H\"older continuous coefficients (at least for general dimension $n > 2$, see \cite{Sab15,Ros19}). 

Let us point out that also the assumptions on $f$ in \autoref{thm:main-2} are quite general and take into account different effects. On the one hand, we allow for $f \in L^q(\Omega)$ for some $q \in (\frac{n}{s}, \infty]$ (when $p = s$), which seems to be new even for translation invariant operators. On the other hand, we also cover source terms $f$ that can explode at the boundary like $d_{\Omega}^{s-p}$ for some $p \in (0,s]$ (when $q = \infty$) as in \cite{FeRo24}. We refer to \autoref{remark:f-assumptions} for a possible extension of our main result to even more general source terms.

We also establish $C^{s-\eps}$ regularity in flat Lipschitz domains, as well as optimal regularity estimates of lower order in case $\frac{n}{s+p}<q \leq \frac{n}{p}$ (see \autoref{thm:Cs-eps}). Moreover, under additional structural assumptions on the kernels, we establish fine regularity results for $u/d_{\Omega}^s$ (see \autoref{thm-Campanato-Holder}). Note that without any additional assumption, $u/d_{\Omega}^s$ does not even need to be continuous in \autoref{thm:main-2} (see \cite{RoWe24}).

To prove \autoref{thm:main-2} we develop a new approach to the boundary regularity for nonlocal equations with coefficients. Indeed, we establish a nonlocal version of the celebrated Morrey--Campanato theory at the boundary. As we will explain below in detail, since the boundary behavior of nonlocal and local equations differs quite substantially, several new ideas are required in order to make this approach work for nonlocal problems.

\subsection{Strategy of the proof}

Our proof of the $C^s$ boundary regularity in \autoref{thm:main-2} is based on a nonlocal version of the Morrey--Campanato theory at the boundary. Such technique has already been established for nonlocal equations in the context of interior regularity (see for instance \cite{KMS15,KNS22}) and in this paper we prove  similar estimates at the boundary. 

Obtaining boundary regularity estimates for nonlocal problems is much more difficult than for local equations due to the following reasons:
\begin{itemize}
    \item For local equations it is sufficient to establish the boundary regularity in flat domains and to deduce it in general domains with a suitable (local) change of variables. Such an argument does not work for nonlocal equations.
    \item Solutions to nonlocal equations behave like $d_{\Omega}^s$ near $\partial \Omega$. Thus, derivatives of solutions explode at the boundary, and hence it seems hopeless to obtain higher order regularity estimates (of order $s$) by differentiation of the equation. Due to the same reason, one cannot approximate solutions by linear functions at the boundary.
\end{itemize}

Our proof of \autoref{thm:main-2} consists of several steps:

\subsubsection{Almost optimal regularity}

The first step in the proof of the optimal $C^s$ boundary regularity in \autoref{thm:main-2} is to establish an almost optimal regularity result of order $s-\eps$. We prove this result in the more general setting of flat Lipschitz domains. Moreover, our technique allows to prove optimal regularity estimates of lower order for less regular source terms. Both of these results are new in our setting and we believe them to be of independent interest.

\begin{theorem}
\label{thm:Cs-eps}
Let $s \in (0,1)$, $\sigma \in (0,s)$, and $\varepsilon \in (0, \sigma)$. Let $\Omega \subset \R^n$ be a Lipschitz domain with Lipschitz constant $\delta > 0$ and let $L,K,\lambda,\Lambda$ be as in \eqref{eq-op}, \eqref{eq:Kcomp}. Assume that \eqref{eq-K-cont} holds with $\mathcal{A} = B_1$. Let $f$ be such that $d_{\Omega}^{s-p}f \in L^q(\Omega \cap B_1)$ for some $p \in (0, s]$ and $q \in (\frac{n}{s+p}, \infty]$.
Let $u$ be a weak solution to
\begin{align*}
\left\{
\begin{aligned}
Lu&=f &&\text{in }\Omega \cap B_1, \\
u&=0 &&\text{in } B_1 \setminus \Omega.
\end{aligned}
\right.
\end{align*}
Then, there exists $\delta_0 > 0$, depending only on $n,s,\sigma,\lambda,\Lambda,p,q,\eps$, and $\Omega$, such that if $\delta \le \delta_0$, then the following holds true:
\begin{itemize}
\item If $q\geq \frac{n}{p}$, then $u \in C^{s-\varepsilon}_{\mathrm{loc}}(\overline{\Omega} \cap B_1)$, and
\begin{align*}
\|u\|_{C^{s-\varepsilon}(\overline{\Omega \cap B_{1/2}})} \le c \left( \Vert u \Vert_{L^{1}_{2s}(\R^n)}+ \Vert d_{\Omega}^{s-p}f \Vert_{L^q(\Omega \cap B_1)} \right)
\end{align*}
for some constant $c > 0$, depending only on $n, s, \sigma, \lambda, \Lambda, p, q, \varepsilon$, and $\Omega$.
\item
If $q< \frac{n}{p}$, then $u \in C^{s+p-\frac{n}{q}}_{\mathrm{loc}}(\overline{\Omega} \cap B_1)$ and
\begin{align*}
\|u\|_{C^{s+p-\frac{n}{q}}(\overline{\Omega \cap B_{1/2}})} \le c \left( \Vert u \Vert_{L^{1}_{2s}(\R^n)}  + \Vert d_{\Omega}^{s-p}f \Vert_{L^q(\Omega \cap B_1)} \right)
\end{align*}
for some constant $c > 0$, depending only on $n, s, \sigma, \lambda, \Lambda, p, q$, and $\Omega$.
\end{itemize}
\end{theorem}

The proof of \autoref{thm:Cs-eps} relies on a Morrey-type iteration argument. Given any point $x_0 \in \Omega \cap B_{1/2}$, we consider the ``frozen'' operator $L_{x_0}$ which is of the same shape as \eqref{eq-op} but has the kernel
\begin{align}\label{eq-frozen}
    K_{x_0}(x,y) = \frac{1}{2}[K(x_0 + x - y , x_0) + K(x_0 + y - x , x_0)].
\end{align}
Note that $K_{x_0}$ only depends on $x-y$ and therefore, $L_{x_0}$ is translation invariant. Then, we compare the solution $u$ with its $L_{x_0}$-harmonic replacements $v$ in $\Omega \cap B_R(x_0)$ for all scales $R \in (0,1/2)$. For a fixed $R \in (0,1/2)$, $v$ is given as the solution to
\begin{align}
\label{eq:frozen-intro}
\left\{
\begin{aligned}
L_{x_0}v&=0 &&\text{in }\Omega \cap B_R(x_0),\\
v&=u &&\text{in }\R^n \setminus (\Omega \cap B_R(x_0)).
\end{aligned}
\right.
\end{align}
Since the $C^s$ boundary regularity for $v$ is already known (see \cite{FeRo24,RoWe24}), we can transfer the regularity from $v$ to $u$ by establishing suitable comparison estimates for $[u-v]_{H^s(\R^n)}$ (see \autoref{lemma:freezing}). This way, we derive the following nonlocal Morrey-type estimate, which we state here for simplicity only in case $f = 0$ (see also \autoref{lemma:Morrey-boundary} and \eqref{eq:Morrey-boundary-help-0}):
\begin{align}
\label{eq:Morrey-type-intro}
\Phi_\sigma(u; \rho,x_0) \le c \left[\left( \frac{\rho}{R} \right)^{n-\eps} + R^{\sigma} \right] \Phi_\sigma(u; R,x_0) + c R^{n+s} \Vert u \Vert_{L^1_{2s}(\R^n)} ~~ \forall ~ 0 < \rho < R.
\end{align}
Here, $\Phi_{\sigma}(u;\rho,x_0)$ denotes the nonlocal excess functional
\begin{align}\label{eq-excess-Phi}
    \Phi_\sigma(u; \rho,x_0) := \int_{\Omega \cap B_{\rho}(x_0)} \left| \frac{u}{d_{\Omega}^s} \right| \d x + \max\{\rho,d_{\Omega}(x_0)\}^{-s} \rho^n \tail_{\sigma,B_1}(u;\rho,x_0).
\end{align}
The first summand takes into account the behavior of solutions to nonlocal equations at the boundary. The second summand is defined as
\begin{align*}
    \tail_{\sigma,B_1}(u;\rho,x_0) = \rho^{2s-\sigma} \int_{B_1 \setminus B_{\rho}(x_0)} |u(y)||y-x_0|^{-n-2s+\sigma} \d y,
\end{align*}
and it is more difficult to grasp since it reflects upon three different aspects of the problem at hand:
\begin{itemize}
    \item The position of the ball $B_{\rho}(x_0)$ relative to $\partial\Omega$ is taken into account through the prefactor $\max\{ \rho , d_{\Omega}(x_0) \}^{-s}$.
    \item The tail captures the information on $u$ at intermediate scales $B_1 \setminus B_{\rho}(x_0)$, thereby reflecting the nonlocality of the problem.
    \item The weight and scaling of the tail keep track of the $C^{\sigma}$ regularity of the kernel $K$.
\end{itemize}

A key difficulty in the proof of \eqref{eq:Morrey-type-intro} is to incorporate the tail term into the estimate, making it a purely nonlocal result. Although tail terms commonly appear in the study of nonlocal equations (see \cite{DKP14,DKP16,KaWe24}) and have also been treated before in iterative schemes (see \cite{KNS22,BDLMS25}), we believe that $\tail_{\sigma,B_1}$ is quite special in the sense that the order of the weight matches the regularity of the coefficients. This aspect is crucial to us since it allows us to profit from the regularity of $K$ also on intermediate scales. To prove interior regularity, there is no need to introduce $\tail_{\sigma,B_1}$ since solutions enjoy regularity of an order higher than $C^s$ (see \cite[Lemma 3.5]{KNS22}).

The Morrey-type estimate \eqref{eq:Morrey-type-intro} is a key component in the proof of \autoref{thm:Cs-eps}. Indeed, by combining it with interior regularity results, a standard iteration argument directly establishes the $C^{s-\eps}$ regularity of $u$.

\subsubsection{Upgrade to $C^s$ regularity}

\autoref{thm:Cs-eps} is a crucial ingredient in the proof of \autoref{thm:main-2}. However, upgrading the $C^{s-\eps}$ regularity to the optimal $C^s$ regularity is far from trivial (and this is crucially needed to get the sharp Green function estimates). In fact, it lies in the nature of perturbation arguments such as the ones we develop in this article, that through the comparison with the prototype (translation invariant) equation one always loses a (fractional) derivative of arbitrarily small order $\eps$. In other words, it is impossible to improve the exponent $n-\eps$ in \eqref{eq:Morrey-type-intro} (or any other excess decay estimate of similar type) \textit{unless $v$ is more regular than $C^s$}. However, since the \textit{$C^s$ regularity is optimal} for translation invariant problems, in fact, $v/d_{\Omega}^s$ does not even have to be continuous (see \cite{RoWe24}), new ideas are required in order to prove \autoref{thm:main-2}. This is in stark contrast to the local case.

The first key idea to establish the optimal $C^s$ regularity in \autoref{thm:inhom-Cs} is to prove an \textit{expansion of order $s+\eps$ at boundary points} of $\Omega$ with $\eps > 0$ for translation invariant equations.
Such an expansion was already proved in \cite{RoWe24} but only in terms of a one-dimensional barrier $b_{x_0}$ that solves an equation in the half-space and heavily depends on its orientation. Since this would not allow us to prove $C^s$ regularity in more general domains, we slightly improve the result in \cite{RoWe24} as follows.

For any frozen operator $L_{x_0}$ (more precisely, for any locally frozen operator $\widetilde{L}_{x_0}$ given by \eqref{eq:locally-freeze}) we construct a barrier function $\psi_{x_0} \in C^s$ that takes into account the geometry of $\Omega$ and is comparable to $d_{\Omega}^s$, by an application of the results in \cite{RoWe24}. Then, we show that solutions $v$ to \eqref{eq:frozen-intro} satisfy an expansion in terms of $\psi_{x_0}$ (see \autoref{prop:inhom-Cs}) of the following form for some $q_z \in \R$:
\begin{align*}
|v(x) - q_z \psi_{x_0}(x)| \le c |x-z|^{s+\eps} ~~ \forall z \in B_{R/2}(x_0) \cap \partial \Omega \text{ and } x \in \Omega \cap B_{R/2}(x_0).
\end{align*}

A main achievement of the current article is to show that this expansion provides enough higher order information about the boundary behavior of $v$ and that it actually implies $C^s$ regularity for $u$. This goes by a higher order nonlocal Campanato-type estimate in terms of the function $u/\psi_{x_0}$, which we state here, again, for simplicity only in case $f = 0$ (see also \autoref{lem-u-Holder-inhom}):
\begin{align}
\label{eq:Campanato-type-intro}
\Psi_{\sigma}(u;\rho,x_0) \leq c \left( \frac{\rho}{R} \right)^{n+\eps} \Psi_{\sigma}(u;R,x_0) + c R^{\sigma} \Phi_{\sigma}(u;R,x_0) + c R^{n + s } \Vert u \Vert_{L^1_{2s}(\R^n)}.
\end{align}
Here, $\Psi_{\sigma}(u;\rho,x_0)$ denotes the higher order nonlocal excess functional
\begin{align}\label{eq-excess-Psi}
\begin{split}
    \Psi_{\sigma}(u;\rho,x_0) &:= \int_{\Omega \cap B_{\rho}(x_0)} \left|\frac{u}{\psi_{x_0}} - \left( \frac{u}{\psi_{x_0}} \right)_{\Omega \cap B_{\rho}(x_0)} \right| \d x \\
    &\quad + \max \{ \rho,d_{\Omega}(x_0) \}^{-s} \rho^n \mathrm{Tail}_{\sigma,B_1}\left(u - \psi_{x_0} \left(\frac{u}{\psi_{x_0}} \right)_{\Omega \cap B_{\rho}(x_0)}; \rho, x_0 \right).
\end{split}
\end{align}
Note that the quotients $u/\psi_{x_0}$ and $v/\psi_{x_0}$ do not solve any reasonable equation, which is why the proof of \eqref{eq:Campanato-type-intro} is quite involved and differs significantly from the ones in \cite{KNS22}.

Another key difficulty is to deduce the $C^s$ regularity of $u$ from the Campanato-type estimate \eqref{eq:Campanato-type-intro}. In fact, since the functions $\psi_{x_0}$ heavily depend on $x_0$ and are themselves only in $C^s$ it is not possible to deduce any regularity of $u$ directly from \eqref{eq:Campanato-type-intro}. We solve this problem by proving in \autoref{lemma:w-regularity} that for any $x_0,y_0 \in \Omega \cap B_{1/2}$, it holds
\begin{align}
\label{eq:phi-closeness-intro}
[\psi_{x_0} - \psi_{y_0}]_{C^{s-\eps}(\overline{\Omega \cap B_{1/2}})} \le c |x_0 - y_0|^{\sigma}.
\end{align}
The proof of this result requires a considerable amount of work and goes by  extending some of the arguments in the proof of \autoref{thm:Cs-eps} to equations with nonlocal operators on the right-hand side. 

Thanks to \eqref{eq:phi-closeness-intro} we are able to replace $u/\psi_{x_0}$ in the Campanato excess decay estimate \eqref{eq:Campanato-type-intro} by $u(x)/\psi_{x}(x)$. This way, we obtain
\begin{align*}
   \left( x \mapsto \frac{u(x)}{\psi_{x}(x)} \right) \in C^{\eps},
\end{align*}
see \autoref{thm:Cs-inhom-flat}, and in particular it implies the optimal $C^s$ regularity of $u$.

\subsubsection{Higher regularity for homogeneous kernels}

Finally, let us mention that we also obtain higher order fine boundary asymptotics under additional structural assumptions on $K$. Recall that for translation invariant operators one can prove higher regularity of $u/d_{\Omega}^s$ if $K$ is homogeneous, i.e.\ $K(h) = a(h/|h|)|h|^{-n-2s}$ for some function $a : \mathbb{S}^{n-1} \to [\lambda,\Lambda]$ (see \cite{RoSe16,RoSe16b,RoSe17,AbRo20}). The following result shows that the same higher order asymptotics hold true when the frozen kernels $K_{x_0}$ are homogeneous:

\begin{theorem}\label{thm-Campanato-Holder}
Let $s \in (0,1)$, $\alpha,\sigma \in (0,s)$, and $\eps \in (0,\alpha)$. Let $\Omega \subset \R^n$ be a $C^{1,\alpha}$ domain and let $L,K,\lambda,\Lambda$ be as in \eqref{eq-op}, \eqref{eq:Kcomp}. Assume that \eqref{eq-K-cont} holds with $\mathcal{A} = B_1$ and that $K_{x_0}$ is homogeneous for any $x_0 \in \Omega \cap B_{1}$. Let $f$ be such that $d_{\Omega}^{s-p}f \in L^q(\Omega \cap B_1)$ for some $p \in (0,s]$ and $q \in (\frac{n}{p},\infty]$.
Let $u$ be a weak solution to
\begin{align*}
\left\{
\begin{aligned}
Lu&=f &&\text{in }\Omega \cap B_1, \\
u&=0 &&\text{in } B_1 \setminus \Omega.
\end{aligned}
\right.
\end{align*}
Then $u/d_\Omega^s \in C^{\min\{\alpha - \eps ,\sigma,p - \frac{n}{q}\}}_{\mathrm{loc}}(\overline{\Omega} \cap B_1)$ and
\begin{align*}
\|u/d_\Omega^s\|_{C^{\min\{\alpha - \eps ,\sigma,p - \frac{n}{q}\}}(\overline{\Omega \cap B_{1/2}})} \leq c \left( \|u\|_{L^1_{2s}(\R^n)} + \|d_{\Omega}^{s-p} f\|_{L^q(\Omega \cap B_1)} \right)
\end{align*}
for some constant $c > 0$, depending only on $n, s, \lambda, \Lambda, \alpha, \sigma, p, q, \eps$, and $\Omega$.
\end{theorem}

The following is a natural class of kernels satisfying the assumptions of \autoref{thm-Campanato-Holder}. Take
\begin{align*}
    K(x,y) = |x-y|^{-n-2s} a_x((x-y)/|x-y|) + |x-y|^{-n-2s} a_y((x-y)/|x-y|) 
\end{align*}
and $a_x(h) = a_x(-h)$ with $0 < \lambda \le a(h) \le \Lambda < \infty$ for every $h \in \mathbb{S}^{n-1}$. Note that $K$ does not possess the non-divergence form symmetry. Moreover, we believe this class of kernels to be natural since kernels of similar structure have already been considered in the literature (see \cite{BKS19,FeRo24b}).

The proof of \autoref{thm-Campanato-Holder} is less technical than the one of \autoref{thm:inhom-Cs}. In fact, in the framework of \autoref{thm-Campanato-Holder} one can establish a similar higher order Campanato-type estimate as in \eqref{eq:Campanato-type-intro} and directly deduce regularity of $u/d_{\Omega}^s$ from it since in this case the corresponding barrier functions $\phi_{x_0}$ satisfy $\phi_{x_0}/d_{\Omega}^s \in C^{\alpha}$ due to \cite{RoSe17}, see Section~\ref{sec:homCs}.

\subsection{Outline}
This article is structured as follows. In Section~\ref{sec:prelim} we introduce several function spaces, list functional inequalities, and provide auxiliary lemmas. In Section~\ref{sec:aux} we collect several well-known regularity results for weak solutions to nonlocal equations in divergence form. Comparison estimates between translation invariant equations and equations with coefficients are established in Section~\ref{sec:freezing}. Section~\ref{sec:Morrey} contains the almost optimal $C^{s-\eps}$ regularity and the proof of \autoref{thm:Cs-eps}. In Sections~\ref{sec:Cs} and \ref{sec:Hopf} we establish the higher order Campanato theory and show the Hopf lemma, thereby proving our main result \autoref{thm:main-2}. Section~\ref{sec:homCs} is devoted to the proof of \autoref{thm-Campanato-Holder}. Finally, in Section~\ref{sec:Green} we establish the Green function estimates from \autoref{thm:main-1} and \autoref{cor:gradient-bounds}.

\subsection{Acknowledgments}

Minhyun Kim was supported by the National Research Foundation of Korea (NRF) grant funded by the Korean government (MSIT) (RS-2023-00252297). Marvin Weidner was supported by the European Research Council under the Grant Agreement No. 101123223 (SSNSD), by the AEI project PID2021-125021NA-I00 (Spain), and by the Deutsche Forschungsgemeinschaft (DFG, German Research Foundation) under Germany's Excellence Strategy - EXC-2047/1 - 390685813.

\section{Preliminaries}
\label{sec:prelim}

In this section, we collect several definitions of function spaces, define weak solutions, introduce some notions, list functional inequalities, and provide auxiliary lemmas that will be used throughout the paper. Recall that we always assume  $s \in (0,1)$, $n>2s$, $\Lambda \geq \lambda >0$, and that $\Omega$ is a domain in $\R^n$.

Throughout this paper, we use the following notation:
\begin{align*}
    \Omega_\rho(x_0) = \Omega \cap B_\rho(x_0), \qquad (u)_\Omega = \dashint_{\Omega} |u| \d x = |\Omega|^{-1} \int_\Omega |u| \d x, \qquad (u)_{\rho,x_0} = (u)_{B_{\rho}(x_0)}.
\end{align*}

\subsection{Function spaces and weak solutions}
Let $\Omega' \subset \R^n$ be with $\Omega \subset \Omega'$ and define
\begin{align*}
    H^s(\Omega | \Omega')=\left\{ u \in L^2(\Omega): [u]_{H^s(\Omega | \Omega')}^2 := \int_{\Omega} \int_{\Omega'} \frac{|u(x)-u(y)|^2}{|x-y|^{n+2s}} \d y \d x < \infty \right\}
\end{align*}
to be a fractional Sobolev space equipped with the norm
\begin{align*}
    \|u\|_{H^s(\Omega | \Omega')} = \left( \|u\|_{L^2(\Omega)}^2 + [u]_{H^s(\Omega | \Omega')}^2 \right)^{1/2}.
\end{align*}
If $\Omega = \Omega'$, we write $H^s(\Omega | \Omega') = H^s(\Omega)$ and define $H^s_\Omega(\R^n)= \{ u \in H^s(\R^n): u \equiv 0 \text{ in } \R^n \setminus \Omega \}$. Given a kernel $K:\R^n \to [0, \infty]$ satisfying \eqref{eq:Kcomp}, we define for measurable functions $u, v: \R^n \to \R$
\begin{align*}
    \mathcal{E}^K(u, v) = \int_{\R^n} \int_{\R^n} (u(x)-u(y))(v(x)-v(y)) K(x, y) \d y \d x,
\end{align*}
provided that it is finite. Note that $\mathcal{E}^K(u, v)$ is finite when $u, v \in H^s(\R^n)$, and that it gives rise to an integro-differential operator $L$ given by \eqref{eq-op} via the relation
\begin{align}\label{eq-weak-formulation}
    \mathcal{E}^K(u, \varphi) = \langle Lu, \varphi \rangle \quad\text{for all }\varphi \in H^s_\Omega(\R^n).
\end{align}

In order to define weak solutions, we also recall the so-called \emph{tail space} $L^1_{2s}(\R^n)$, which is given by
\begin{align*}
    L^1_{2s}(\R^n)&=\left\{ u \in L^1_{\mathrm{loc}}(\R^n): \|u\|_{L^1_{2s}(\R^n)} := \int_{\R^n} \frac{|u(y)|}{(1+|y|)^{n+2s}} \d y < \infty \right\}.
\end{align*}
It is well known that $H^s(\R^n) \subset L^1_{2s}(\R^n)$ and that the \emph{tail}, defined by
\begin{align*}
    \tail(u; R, x_0) = R^{2s} \int_{\R^n \setminus B_R(x_0)} \frac{|u(y)|}{|y-x_0|^{n+2s}} \d y,
\end{align*}
is finite for any $x_0 \in \R^n$ and $R>0$ if $u \in L^1_{2s}(\R^n)$. The tail is introduced to capture the nonlocality of solutions to nonlocal equations with kernels $K$. 

By using the relation \eqref{eq-weak-formulation}, we define weak solutions to nonlocal equations with respect to $L$ as follows.

\begin{definition}
Let $f \in (H_\Omega(\R^n))^\ast$. We say that $u$ is a (weak) \emph{subsolution} to $Lu\leq f$ in $\Omega$ if there exists a domain $\Omega' \subset \R^n$ with $\Omega \Subset \Omega'$ such that $u \in H^s(\Omega | \Omega') \cap L^1_{2s}(\R^n)$ and
\begin{align}\label{eq-weaksol}
    \mathcal{E}^K(u, \varphi) \leq \langle f, \varphi \rangle
\end{align}
for all nonnegative functions $\varphi \in H^s_\Omega(\R^n)$. We say that $u$ is a (weak) \emph{supersolution} to $Lu\geq f$ in $\Omega$ if \eqref{eq-weaksol} holds with $u$ replaced by $-u$. Moreover, we say that $u$ is a (weak) \emph{solution} to $Lu=f$ in $\Omega$ if it is both a subsolution and a supersolution.
\end{definition}

We require weak solutions to be in $H^s(\Omega | \Omega')$ instead of $H^s_{\mathrm{loc}}(\Omega)$ in order to be able to analyze their boundary regularity. This condition is very natural, see also the discussion in \cite[p.5-6]{KiLe23}. Note also that if $u$ is a solution to $Lu=f$ in $\Omega$ with zero exterior condition $u=0$ in $\R^n \setminus \Omega$, then $u \in H^s_\Omega(\R^n)$, and so $u$ itself is an admissible test function.

As we mentioned in the introduction, we need another type of tail capturing the decay of $K-K_{x_0}$. For this purpose, we introduce the following \emph{$\sigma$-tail} for $\sigma \in (0, s)$:
\begin{align*}
    \tail_{\sigma,B_1}(u;R,x_0) = R^{2s-\sigma} \int_{B_1 \setminus B_R(x_0)} |u(y)| |y-x_0|^{-n-2s+\sigma} \d y.
\end{align*}
Note that $\tail_{\sigma,B_1}(u;R,x_0) \leq R^{-\sigma}\tail(u;R,x_0)<\infty$ if $u \in L^1_{2s}(\R^n)$ and that we have the following trivial estimate:
\begin{align}
\label{eq:sigma-tail-est}
    \tail(u;R,x_0) \le \tail_{\sigma,B_1}(u;R,x_0) + R^{2s} \tail(u;1,0).
\end{align}

\subsection{Functional inequalities}

This section is devoted to several functional inequalities that will be useful throughout the paper. We begin with the (fractional) Poincar\'e inequality.

\begin{proposition}\label{prop-Poincare}\cite[Corollary~2.1]{Pon04}
Let $x_0 \in \R^n$ and $R>0$. Then there exists $c=c(n, s)>0$ such that
\begin{align*}
\int_{B_R(x_0)} |u-(u)_{R, x_0}|^2 \d x \le c R^{2s} [u]^2_{H^s(B_R(x_0))}
\end{align*}
for any $u \in H^s(B_R(x_0))$.
\end{proposition}

The following property, which is an easy consequence of the triangle inequality, will also be used: for any $u \in H^s(\Omega)$ and $c \in \R$:
\begin{align}\label{eq-E-inf}
\dashint_{\Omega} |u-(u)_{\Omega}| \d x \le 2  {\dashint_{\Omega} |u-c| \d x}.
\end{align}

The (fractional) Poincar\'e--Wirtinger-type inequality for functions having fat zero level sets is known.

\begin{lemma}\cite[Lemma~4.7]{Coz17}
\label{lemma:Poincare-Wirtinger}
Let $x_0 \in \R^n$ and $R > 0$. Let $u \in H^s(B_R(x_0))$ be such that $|\{ u = 0\} \cap B_R(x_0)| \ge c_0 |B_R(x_0)|$ for some $c_0 > 0$. Then
\begin{align*}
\|u\|_{L^2(B_R(x_0))} \le c R^s [u]_{H^s(B_R(x_0))}
\end{align*}
for some $c =c(n, s, c_0)> 0$.
\end{lemma}

We have the following version of the (fractional) Hardy inequality, where $W^{s, p}(\R^n)$ and $W^{s, p}_\Omega(\R^n)$ are the usual fractional Sobolev spaces with integrability $p\geq 1$.

\begin{lemma}
\label{lemma:Hardy}
Let $p\geq 1$. Let $\Omega \subset \R^n$ be a domain such that $\Omega^c$ satisfies the measure density condition, i.e.\ there exists $c_0>0$ such that
\begin{align}
\label{eq:measure-density}
    \inf_{x_0 \in \partial \Omega} \inf_{r>0} \frac{|B_r(x_0) \setminus \Omega|}{|B_r(x_0)|} \geq c_0.
\end{align}
(In particular, $\Omega$ can be any bounded Lipschitz domain.)
Then
\begin{align*}
    \int_{\Omega} \left| \frac{u}{d_{\Omega}^s} \right|^p \d x \le c [u]^p_{W^{s, p}(\R^n)}
\end{align*}
for any $u \in W^{s, p}_\Omega(\R^n)$, where $c=c(n, s, p, c_0)>0$.
\end{lemma}

\begin{proof}
For any $x \in \Omega$ there exists $x_0 \in \partial \Omega$ such that $d_\Omega(x)=|x-x_0|$. We then have
\begin{align*}
    \int_{\Omega^c} \frac{\d y}{|x-y|^{n+sp}} \ge \int_{B_{d_\Omega(x)}(x_0) \setminus \Omega} \frac{\d y}{|x-y|^{n+sp}} \ge \frac{|B_{d_\Omega(x)}(x_0) \setminus \Omega|}{(2d_\Omega(x))^{n+sp}} \ge c d_{\Omega}^{-sp}(x)
\end{align*}
for some $c=c(n, s, p, c_0)>0$. Since $u \equiv 0$ in $\Omega^c$, we conclude
\begin{align*}
    \int_{\Omega} \left| \frac{u}{d_{\Omega}^s} \right|^p \d x \le c\int_{\Omega} |u(x)|^p \left( \int_{\Omega^c} \frac{\d y}{|x-y|^{n+sp}} \right) \d x = c \int_{\Omega} \int_{\Omega^c} \frac{|u(x) - u(y)|^p}{|x-y|^{n+sp}} \d y \d x \le c [u]_{W^{s,p}(\R^n)}^p,
\end{align*}
as desired.
\end{proof}

\subsection{Auxiliary lemmas}

We also have the following auxiliary estimate.

\begin{lemma}
\label{lemma:distance-integral}
Let $\Omega \subset \R^n$ be a Lipschitz domain, $x_0 \in \Omega$, $R > 0$, and $\varepsilon \in (0,1)$. Then
\begin{align*}
    \int_{\Omega_R(x_0)} d_{\Omega}^{-\eps} \d x \le c R^{n -\eps}
\end{align*}
for some $c=c(n,\eps,\Omega)>0$.
\end{lemma}

\begin{proof}
Note that in case $d_\Omega(x_0) \geq 2R$, the result is immediate since $d_{\Omega}^{-\eps} \le R^{-\eps}$ in $\Omega_R(x_0)$. Hence, we assume from now on that $d_\Omega(x_0) < 2R$.
    By applying \cite[Lemma B.2.4]{FeRo24} with $\gamma : = -\eps$ and $\beta := -n$, we obtain
    \begin{align*}
        \int_{\Omega_1(x_0)} d^{-\eps}_{\Omega} \d x = \int_{\Omega_1(x_0) \setminus B_{d_{\Omega}(x_0)/2}(x_0)} d^{-\eps}_{\Omega} \d x + \int_{B_{d_{\Omega}(x_0)/2}(x_0)} d^{-\eps}_{\Omega} \d x \le c(1 + d_{\Omega}(x_0)^{n-\eps})
    \end{align*}
    for some constant $c > 0$, depending only on $n,\eps$, and $\Omega$.
    Hence, an application of the previous estimate to $\tilde{\Omega} = R^{-1}(\Omega - x_0)$ shows that
    \begin{align*}
        \int_{\Omega_R(x_0)} d^{-\eps}_{\Omega}(x) \d x &= R^n \int_{\tilde{\Omega} \cap B_1} d^{-\eps}_{\Omega-x_0}(Rx) \d x = R^{n-\eps} \int_{\tilde{\Omega} \cap B_1} d^{-\eps}_{\tilde{\Omega}}(x) \d x \\
        &\le c R^{n-\eps} (1 + d_{\tilde{\Omega}}(0)^{n-\eps}) \le c R^{n-\eps} (1 + R^{\eps - n}d_{\Omega}(x_0)^{n-\eps}) \le c R^{n-\eps},
    \end{align*}
    where we used that $d_{\Omega - x_0}(R \cdot) = R d_{\tilde{\Omega}}$ and $d_{\Omega}(x_0) < 2R$.
\end{proof}

The following iteration lemma is crucial in the Morrey--Campanato theory.

\begin{lemma}
\label{lem-iteration}
    Let $A,B,\alpha,\beta,R_0>0$, $\varepsilon \geq 0$, and assume $\alpha>\beta$.
    Let $\Phi: [0, \infty) \to [0, \infty)$ be a function satisfying
    \begin{align}\label{eq-iteration}
        \Phi(\rho) \leq A \left[ \left(\frac{\rho}{R} \right)^{\alpha} + \varepsilon \right] \Phi(R) + BR^{\beta} \quad\text{for all }0<\rho < R \leq R_0.
    \end{align}
    Assume that for any $\tau \in (0,1)$, there exists $c_0=c_0(\tau)>0$ such that
    \begin{align}\label{eq-almost-incr}
        \Phi(\rho) \leq c_0 \Phi(\bar{\rho}) \quad\text{for all }0 <\bar{\rho} \le R_0 \text{ and } \rho \in [\tau \bar{\rho}, \bar{\rho}].
    \end{align}
    Then there exist $\varepsilon_0,c>0$ such that if $\varepsilon\leq \varepsilon_0$, then
    \begin{align*}
        \Phi(\rho) \leq c \left[ \frac{\Phi(R)}{R^{\beta}} + B \right] \rho^{\beta} \quad\text{for all }0 <\rho<R \leq R_0.
    \end{align*}
    The constants $\varepsilon_0$ and $c$ depend on $A,\alpha,\beta$, and $c_0(\tau)$ for some $\tau=\tau(A,\alpha,\beta)$.
\end{lemma}

\begin{proof}
The proof goes as in \cite[Lemma~5.13]{GiMa12}, with the only difference that we use \eqref{eq-almost-incr} instead of the assumption that $\Phi$ is non-decreasing in the last display.
\end{proof}

Throughout this paper, we will apply \autoref{lem-iteration} with $\Phi_{\sigma}(u,R):=\Phi_{\sigma}(u,R,x_0)$ and $\Psi_{\sigma}(u,R):=\Psi_{\sigma}(u,R,x_0)$, which are defined in \eqref{eq-excess-Phi} and \eqref{eq-excess-Psi}, respectively. It will require a significant amount of work to verify that $\Phi_{\sigma}(u,R)$ and $\Psi_{\sigma}(u,R)$ satisfy \eqref{eq-iteration} when $u$ is a solution to \eqref{eq:dirichlet-intro}. However, verifying \eqref{eq-almost-incr} is a straightforward consequence of the triangle inequality:

\begin{lemma}
    Let $\Omega \subset \R^n$ be a Lipschitz domain and $u, \psi \in H^s(\Omega | \Omega') \cap L^1_{2s}(\R^n)$. Assume that $0 \le \psi \le c_0 d_{\Omega}^s$ in $\Omega_{1/4}$. Let $x_0 \in \Omega_{1/8}$ and $0 < \bar{\rho} < \frac{1}{8}$. Then, for any $\tau \in (0,1)$ and $\rho \in [\tau \bar{\rho} , \bar{\rho}]$ it holds
    \begin{align*}
        \Phi_{\sigma}(u; \rho, x_0) \le c \Phi_{\sigma}(u; \bar{\rho}), \qquad \Psi_{\sigma}(u; \rho, x_0) \le c \Psi_{\sigma}(u; \bar{\rho})
    \end{align*}
    for some $c=c(n,s,c_0,\tau) > 0$.
\end{lemma}

\begin{proof}
For $\Phi_{\sigma}$, we compute
\begin{align}
\label{eq-almost-incr-Phi}
\begin{split}
\Phi_{\sigma}(u; \rho)
&\leq \Phi_{\sigma}(u; \bar{\rho}) + \max\{\rho,d_{\Omega}(x_0)\}^{-s} \rho^{n+2s-\sigma} \int_{B_{\bar{\rho}}(x_0) \setminus B_\rho(x_0)} \frac{|u(y)|}{|y-x_0|^{n+2s-\sigma}} \d y \\
&\leq \Phi_{\sigma}(u; \bar{\rho}) + \frac{1}{\max\{\rho, d_\Omega(x_0)\}^s} \int_{\Omega_{\bar{\rho}}(x_0)} |u(y)| \d y \\
&\leq \Phi_{\sigma}(u; \bar{\rho}) + c\frac{\max\{\bar{\rho}, d_\Omega(x_0)\}^s}{\max\{\rho, d_\Omega(x_0)\}^s} \int_{\Omega_{\bar{\rho}}(x_0)} \left| \frac{u}{d_\Omega^s} \right| \d y \\
&\leq c\Phi_{\sigma}(u; \bar{\rho})
\end{split}
\end{align}
for some $c=c(s,\tau)>0$.
    For $\Psi_{\sigma}$, we proceed as follows,  using \eqref{eq-E-inf}:
    \begin{align}
    \label{eq-almost-incr-Psi}
    \begin{split}
        \Psi_{\sigma}(u; \rho)
        &\leq 2\Psi_{\sigma}(u; \bar{\rho}) + \max\{\rho, d_\Omega(x_0)\}^{-s} \rho^{n+2s-\sigma} \int_{B_{\bar{\rho}}(x_0) \setminus B_\rho(x_0)} \frac{|u(y)-\psi(y)(u/\psi)_{\Omega_\rho(x_0)}|}{|y-x_0|^{n+2s-\sigma}} \d y \\
        & \leq 2\Psi_{\sigma}(u; \bar{\rho}) + \max\{\rho, d_\Omega(x_0)\}^{-s} \int_{\Omega_{\bar{\rho}}(x_0)} |u(y)-\psi(y)(u/\psi)_{\Omega_\rho(x_0)}| \d y \\
        &\leq 2\Psi_{\sigma}(u; \bar{\rho}) + c\frac{\max\{\bar{\rho}, d_\Omega(x_0)\}^s}{\max\{\rho, d_\Omega(x_0)\}^s} \int_{\Omega_{\bar{\rho}}(x_0)} \frac{|u(y)-\psi(y)(u/\psi)_{\Omega_{\bar{\rho}}(x_0)}|}{\psi(y)} \d y \\
        &\quad + \left( \int_{\Omega_{\bar{\rho}}(x_0)} \frac{\psi(y)}{\max\{\rho, d_\Omega(x_0)\}^{s}} \d y \right) \dashint_{\Omega_\rho(x_0)} \left|\frac{u}{\psi} - \left(\frac{u}{\psi} \right)_{\Omega_{\bar{\rho}}(x_0)} \right|\d x\\
        &\leq c\Psi_{\sigma}(u; \bar{\rho}),
        \end{split}
    \end{align}
    where $c=c(n, s, c_0, \tau)>0$.
\end{proof}

\section{Auxiliary results for weak solutions}
\label{sec:aux}

In this section, we provide several auxiliary results for weak solutions to $Lu=f$ in $\Omega$. The following lemma contains an important estimate for source terms $f$, satisfying $d^{s-p}_{\Omega} f \in L^q(\Omega_1)$, which will be used throughout the paper. Let us point out that such a class in particular includes source terms $f$ that are in $L^q(\Omega_1)$, and also source terms that might explode near $\partial \Omega \cap B_1$ like $d_{\Omega}^{p-s}$.

\begin{lemma}
\label{lemma:fw-integral}
Let $\Omega \subset \R^n$ be a domain such that $\Omega^c$ satisfies the measure density condition \eqref{eq:measure-density} with $c_0>0$. Let $x_0 \in \Omega$, $R>0$, and let $f$ be such that $d^{s-p}_{\Omega} f \in L^{q}(\Omega_R(x_0))$ for some $p \in (0,s]$ and $q \in (\frac{n}{s+p}, \infty]$. Then
\begin{align}\label{eq-fw-integral-2}
        \int_{\Omega_R(x_0)} |fw| \d x \le c |\{ w \not = 0 \} \cap B_R(x_0)|^{\frac{1}{2} + \frac{p}{n} - \frac{1}{q}} \Vert d_{\Omega}^{s-p} f \Vert_{L^q(\Omega_R(x_0))} [w]_{H^s(\R^n)}
\end{align}
for all $w \in H^s_{\Omega_R(x_0)}(\R^n)$, where $c=c(n,s,p,q,c_0) > 0$.
In particular, we have
\begin{align}\label{eq-fw-integral}
    \int_{\Omega_R(x_0)} |fw| \d x \le c R^{\frac{n}{2} + p - \frac{n}{q}} \Vert d_{\Omega}^{s-p} f \Vert_{L^q(\Omega_R(x_0))} [w]_{H^s(\R^n)}
\end{align}
for all $w \in H^s_{\Omega_R(x_0)}(\R^n)$, where $c=c(n,s,p,q,c_0) > 0$.
\end{lemma}

\begin{proof}
%By using H\"older's inequality and the Hardy inequality (\autoref{lemma:Hardy}), we obtain
%\begin{align*}
%\int_{\Omega_R(x_0)} fw \d x
%&\leq \|d_{\Omega}^{s-p}f\|_{L^q(\Omega_R(x_0))} \left( \int_{\Omega_R(x_0)} \left|\frac{w}{d_{\Omega_R(x_0)}^{s-p}}\right|^{q'} \d x \right)^{1/q'} \\
%&\leq c \|d_{\Omega}^{s-p}f\|_{L^q(\Omega_R(x_0))} [w]_{W^{s-p, q'}(\R^n)},
%\end{align*}
%where $q'=q/(q-1)$. Since $w=0$ outside $\Omega_R(x_0)$, we have
%\begin{align*}
%[w]_{W^{s-p, q'}(\R^n)}
%&\leq c[w]_{W^{s-p, q'}(B_{2R}(x_0))} + c \left( \int_{B_R(x_0)} \int_{B_{2R}(x_0)^c} \frac{|w(x)|^{q'}}{|x-y|^{n+(s-p)q'}} \d y \d x \right)^{1/q'} \\
%&\leq c[w]_{W^{s-p, q'}(B_{2R}(x_0))} + cR^{-s+p} \|w\|_{L^{q'}(B_R(x_0))}.
%\end{align*}
%Note that the embedding
%\begin{align}\label{eq-embedding}
%H^s(B_{2R}(x_0)) \hookrightarrow W^{s-p, q'}(B_{2R}(x_0))
%\end{align}
%follows from \cite[Lemma~4.6]{Coz17} when $q>2$, \cite[Corollary~2.3]{DNPV12} when $q=2$, and \cite[Equation~(1.301)]{Tri06} when $q<2$. Thus, the embedding \eqref{eq-embedding} and \autoref{lemma:Poincare-Wirtinger} show that
%\begin{align*}
%[w]_{W^{s-p, q'}(B_{2R}(x_0))} + R^{-s+p} \|w\|_{L^{q'}(B_R(x_0))}
%&\leq cR^{\frac{n}{2}+p-\frac{n}{q}} \left( [w]_{H^s(B_{2R}(x_0))} + R^{-s} \|w\|_{L^2(B_{2R}(x_0))} \right) \\
%&\leq cR^{\frac{n}{2}+p-\frac{n}{q}} [w]_{H^s(\R^n)},
%\end{align*}
%which proves \eqref{eq-fw-integral}.
By using H\"older's inequality and the Hardy inequality (\autoref{lemma:Hardy}), we obtain
\begin{align*}
    \int_{\Omega_R(x_0)} fw \d x
&\leq \|d_{\Omega}^{s-p}f\|_{L^q(\Omega_R(x_0))} \left( \int_{\Omega_R(x_0)} \left|\frac{w}{d_{\Omega_R(x_0)}^{s-p}}\right|^{\frac{2n}{n-2p}} \d x \right)^{\frac{n-2p}{2n}} |\{ w \not = 0 \} \cap B_R(x_0)|^{\frac{1}{2}+\frac{p}{n}-\frac{1}{q}} \\
&\leq c \|d_{\Omega}^{s-p}f\|_{L^q(\Omega_R(x_0))} [w]_{W^{s-p, \frac{2n}{n-2p}}(\R^n)}  |\{ w \not = 0 \} \cap B_R(x_0)|^{\frac{1}{2}+\frac{p}{n}-\frac{1}{q}},
\end{align*}
Here, we also used that that $q > \frac{n}{s+p} > \frac{2n}{n+2p}$, since $2s < n$. Then, since $w=0$ outside $\Omega_R(x_0)$, we have
\begin{align*}
[w]_{W^{s-p, \frac{2n}{n-2p}}(\R^n)}
&\leq c[w]_{W^{s-p, \frac{2n}{n-2p}}(B_{2R}(x_0))} + c \left( \int_{B_R(x_0)} \int_{B_{2R}(x_0)^c} \frac{|w(x)|^{\frac{2n}{n-2p}}}{|x-y|^{n+(s-p)\frac{2n}{n-2p}}} \d y \d x \right)^{\frac{n-2p}{2n}} \\
&\leq c[w]_{W^{s-p, \frac{2n}{n-2p}}(B_{2R}(x_0))} + cR^{-s+p} \|w\|_{L^{\frac{2n}{n-2p}}(B_R(x_0))}.
\end{align*}
Moreover, we have the embedding
\begin{align}\label{eq-embedding2}
    H^s(B_{2R}(x_0)) \hookrightarrow W^{s-p, \frac{2n}{n-2p}}(B_{2R}(x_0)),
\end{align}
which follows from \cite[Equation~(1.301)]{Tri06}, using that the indices of the two fractional Sobolev spaces coincide, i.e. $s - p - ( n/ \frac{2n}{n-2p} ) = s - \frac{n}{2}$. Thus, the embedding \eqref{eq-embedding2} and \autoref{lemma:Poincare-Wirtinger} show that
\begin{align*}
&[w]_{W^{s-p, \frac{n-2p}{2n}}(B_{2R}(x_0))} + R^{-s+p} \|w\|_{L^{\frac{n-2p}{2n}}(B_R(x_0))} \\
&\quad \leq c \left( [w]_{H^s(B_{2R}(x_0))} + R^{-s} \|w\|_{L^2(B_{2R}(x_0))} \right) \leq c [w]_{H^s(\R^n)},
\end{align*}
which proves \eqref{eq-fw-integral-2}. The second claim \eqref{eq-fw-integral} follows once again from the fact that $q > \frac{n}{s+p} > \frac{2n}{n+2p}$.
\end{proof}

\begin{remark}
\label{remark:f-assumptions}
The only property we need for $f$ throughout this paper is the inequality \eqref{eq-fw-integral}. Thus, one may introduce a function space
\begin{align*}
X(\Omega_1)=\left\{ f \in (H_{\Omega_1}^s(\R^n))^\ast: \|f\|_{X(\Omega_1)} < \infty \right\},
\end{align*}
\begin{align*}
    \|f\|_{X(\Omega_1)} = \sup_{x_0 \in \Omega_{1/2}, 0<R<\frac{1}{2}} \, \sup_{w \in H^s_{\Omega_R(x_0)}(\R^n)} |\{ w \not = 0 \} \cap B_R(x_0)|^{\frac{1}{2} + \frac{p}{n} - \frac{1}{q}} \frac{\int_{\Omega_R(x_0)} fw \d x}{[w]_{H^s(\R^n)}},
\end{align*}
and take $f \in X(\Omega_1)$ instead of $f$ with $d_\Omega^{s-p}f \in L^q(\Omega_1)$ in our main theorems.
\end{remark}

We have the following Caccioppoli-type inequality up to the boundary.

\begin{lemma}
\label{lemma:Cacc-bdry}
Let $\Omega$ and $f$ be given as in \autoref{lemma:fw-integral}. Let $x_0 \in \Omega$ and $R>0$. Assume that $K$ satisfies \eqref{eq:Kcomp} and let $u$ be a solution to
\begin{align}\label{eq-u-Cacc}
\left\{
\begin{aligned}
Lu&=f &&\text{in }\Omega_R(x_0), \\
u&=0 &&\text{in }B_R(x_0) \setminus \Omega.
\end{aligned}
\right.
\end{align}
Then
\begin{align*}
    [u]_{H^s(B_{R/2}(x_0))} \le c R^{-s} \|u\|_{L^2(B_R(x_0))} + c R^{\frac{n}{2}-s} \tail(u;R,x_0) + c R^{\frac{n}{2} + p - \frac{n}{q}} \Vert d^{s-p}_{\Omega} f \Vert_{L^q(\Omega_R(x_0))}
\end{align*}
for some $c=c(n, s, \lambda, \Lambda, p, q, c_0)>0$.
\end{lemma}

\begin{proof}
The proof goes as in \cite[Lemma~3.5]{KiLe23}, however since $u$ is already a solution, we can simply test the equation with $u \eta^2$. The result with $f \not= 0$ can be obtained by an application of \eqref{eq-fw-integral} together with a standard absorption argument.
\end{proof}

The local boundedness estimate of solutions $u$ to \eqref{eq-u-Cacc} up to the boundary can be obtained by following the proof of \cite[Theorem~3.1]{KiLe23} or \cite[Theorem~5]{KKP16}. We state this result in the following lemma.

\begin{lemma}
\label{lemma:locbd-bdry}
Let $\Omega, f, K, u$ be given as in \autoref{lemma:Cacc-bdry}. Then
\begin{align*}
\|u\|_{L^\infty(B_{R/2}(x_0))} \le c R^{-n} \|u\|_{L^1(B_R(x_0))} + \tail(u; R, x_0) + c R^{2s - (s-p) - \frac{n}{q}} \Vert d^{s-p}_{\Omega} f \Vert_{L^q(\Omega_R(x_0))}
\end{align*}
for some $c=c(n, s, \lambda, \Lambda, p, q, c_0)>0$.
\end{lemma}

\begin{proof}
    The proof goes by a standard nonlocal De Giorgi iteration at the boundary, following for instance \cite[Theorem~5]{KKP16} (see also \cite[Theorem~3.1]{KiLe23} for a proof via Moser iteration). First, from H\"older's inequality,  the fractional Sobolev embedding, the Caccioppoli inequality (see \cite[Lemma 4]{KKP16}), and \eqref{eq-fw-integral-2}, we obtain the following gain of integrability estimate:
    \begin{align*}
        \Vert w_{\pm,k}^2 \Vert_{L^{1}(B_{r}(x_0))} & \le |A_{\pm,k}(r+\rho,x_0)|^{\frac{2s}{n}} \Vert w_{\pm,k}^2 \Vert_{L^{\frac{n}{n-2s}}(B_{r}(x_0))} \\
        &\le c |A_{\pm,k}(r+\rho,x_0)|^{\frac{2s}{n}}  \rho^{-2s} \|w_{\pm,k}^2\|_{L^1(B_{r+\rho}(x_0))} \\
        &\quad + c |A_{\pm,k}(r+\rho,x_0)|^{\frac{2s}{n}}  \left( \frac{r+\rho}{\rho} \right)^{n} \rho^{-2s} \|w_{\pm,k}\|_{L^1(B_{r+\rho}(x_0))}\tail(w_{\pm,k};r+\rho,x_0) \\
        &\quad + c |A_{\pm,k}(r+\rho,x_0)|^{1+2 \left(\frac{s+p}{n}-\frac{1}{q} \right)} \Vert d^{s-p}_{\Omega} f \Vert_{L^q(\Omega_R(x_0))}^2,
    \end{align*}
    where we denote $w_{+,k} = (u - k_+)_+$ and $w_{-,k} = (k_- - u)_+$, where $k \in \R$, $\frac{R}{2} \le r \le r+\rho \le R$, and $\rho \in (0,r]$, and set $A_{\pm,k}(r+\rho,x_0) = B_{r+\rho}(x_0) \cap \{ w_{\pm,k} \not=0 \}$. From here, we can derive an iterative scheme, following \cite[Theorem~5]{KKP16}, where the additional summand, coming from the source term $f$ can be treated in a standard way (see for instance \cite[Chapter 11]{Sch20}), using only that $\frac{s+p}{n} - \frac{1}{q} > 0$, by assumption. This yields the desired result.
\end{proof}

We close this section by mentioning an interior regularity result for solutions to nonlocal equations with H\"older continuous coefficients in divergence form.

\begin{proposition}
\label{prop:interior-regularity}
Let $\sigma \in (0,s)$, $x_0 \in \R^n$, and $R > 0$. Assume that $K$ satisfies \eqref{eq:Kcomp}, and \eqref{eq-K-cont} with $\mathcal{A} = B_R(x_0)$.
Let $u$ be a solution to
\begin{align*}
Lu = f ~~ \text{ in } B_R=B_R(x_0),
\end{align*}
where $f \in L^q(B_R)$ for some $q \in (\frac{n}{2s}, \infty]$. Then, $u \in C^{ \min \{2s - \frac{n}{q} , 1 + \sigma - \eps \} }_{\mathrm{loc}}(B_R)$ for any $\varepsilon \in (0, \sigma)$ and
\begin{align*}
    [u]_{C^{ \min \{2s - \frac{n}{q} , 1 + \sigma - \eps \} }(\overline{B_{R/2}})} \le c R^{ - \min \{2s - \frac{n}{q} , 1 + \sigma - \eps \} } \left( R^{-n} \Vert u \Vert_{L^1(B_R)} + \tail(u;R,x_0) + R^{2s - \frac{n}{q}} \Vert f \Vert_{L^q(B_R)} \right),
\end{align*}
where $c=c(n,s,\lambda,\Lambda,\sigma,q,\eps) > 0$.
\end{proposition}

\begin{proof}
    Note that by a scaling argument, it suffices to prove the result for $x_0 = 0$ and $R = 1$. Then, the result follows from \cite{KNS22} and \cite{Now23}. See also \cite{FeRo24b}.
\end{proof}

\section{Freezing estimates}
\label{sec:freezing}

Given $x_0 \in \R^n$ and an operator $L$ with a kernel $K$ satisfying \eqref{eq:Kcomp}, we recall the definition of the frozen operator $L_{x_0}$ at  $x_0$ with kernel $K_{x_0}$ from \eqref{eq-frozen}. $K_{x_0}$ is translation invariant and satisfies \eqref{eq:Kcomp}.

A crucial tool is the comparison estimate between solutions $u$ to
\begin{align}\label{eq-u-Sect4}
\left\{
\begin{aligned}
Lu&=f &&\text{in }\Omega_1, \\
u&=0 &&\text{in } B_1 \setminus \Omega,
\end{aligned}
\right.
\end{align}
and its $L_{x_0}$-harmonic replacement $v$. More precisely, we fix $x_0 \in \Omega_{1/2}$ and $R \in (0, \frac{1}{16})$ so that $\Omega_{8R}(x_0) \subset \Omega_1$, and consider the solution $v$ (which exists by \cite[Theorem~4.3]{KiLe23}) to
\begin{align}\label{eq-v-Sect4}
\left\{
\begin{aligned}
L_{x_0}v&=0 &&\text{in }\Omega_R(x_0), \\
v&=u &&\text{in } \R^n \setminus \Omega_R(x_0).
\end{aligned}
\right.
\end{align}
Then the difference $w:=u-v$ of $u$ and $v$ satisfies
\begin{align}\label{eq-w-Sect4}
\left\{
\begin{aligned}
L_{x_0}w&=(L_{x_0}-L)u+f &&\text{in } \Omega_R(x_0), \\
w&=0 &&\text{in } \R^n \setminus \Omega_R(x_0).
\end{aligned}
\right.
\end{align}

We provide three estimates for $w$ in the rest of this section. Let us begin with the energy estimate.

\begin{lemma}
\label{lemma:freezing}
Let $\Omega \subset \R^n$ be a domain such that $\Omega^c$ satisfies the measure density condition \eqref{eq:measure-density} with $c_0>0$. Let $\sigma \in (0,s)$ and assume that $K$ satisfies \eqref{eq:Kcomp}, and \eqref{eq-K-cont} with $\mathcal{A} = B_1$.
Let $u, v, w$ be solutions of \eqref{eq-u-Sect4}, \eqref{eq-v-Sect4}, \eqref{eq-w-Sect4}, respectively, where $f$ is such that $d_{\Omega}^{s-p}f \in L^q(\Omega_1)$ for some $p \in (0,s]$ and $q\in(\frac{n}{s+p},\infty]$. Then
\begin{align*}
[w]_{H^s(\R^n)}^2
&\le c R^{2\sigma} \left( [u]_{H^s(B_{2R}(x_0))}^2 + R^{n-2s} \tail_{\sigma,B_1}(u - (u)_{B_R(x_0)};R, x_0)^2 \right) \\
&\quad + c R^{n + 2s} \left( \Vert u \Vert_{L^1_{2s}(\R^n)}^2 + R^{-2(s - p) - \frac{2n}{q}} \| d_{\Omega}^{s-p} f \|_{L^q(\Omega_R(x_0))}^2 \right)
\end{align*}
for some $c=c(n, s, \lambda, \Lambda, \sigma, p, q, c_0)>0$.
\end{lemma}

\begin{proof}
Note that by construction,
\begin{align}\label{eq-K-K0}
|K_{x_0}(x,y) - K(x,y)| \le 
    \Lambda \frac{\min\{ |x-x_0|^{\sigma} + |y-x_0|^{\sigma} , 1 \}}{|x-y|^{n+2s}}.
\end{align}
We test the equation for $w$ (choosing $w$ as a test function) and obtain
\begin{align*}
\lambda [w]_{H^s(\R^n)}^2 &\le \cE^{K_{x_0}}(w,w) = \cE^{K_{x_0} - K}(u,w) + \int_{\Omega_R(x_0)} fw \d x\\
&\le \int_{B_{2R}(x_0)}\int_{B_{2R}(x_0)} |u(x) - u(y)||w(x) - w(y)||K_{x_0}(x,y) - K(x,y)| \d y \d x \\
&\quad + 2\int_{B_{R}(x_0)}\int_{B_1 \setminus B_{2R}(x_0)} |u(x) - (u)_{R,x_0}||w(x)||K_{x_0}(x,y) - K(x,y)| \d y \d x \\
&\quad + 2\int_{B_{R}(x_0)} \int_{B_1 \setminus B_{2R}(x_0)} |(u)_{R,x_0} - u(y)||w(x)||K_{x_0}(x,y) - K(x,y)| \d y \d x \\
&\quad + 2\int_{B_{R}(x_0)}\int_{\R^n \setminus B_1} |u(x) - u(y)||w(x)||K_{x_0}(x,y) - K(x,y)| \d y \d x \\
& \quad + \int_{\Omega_R(x_0)} |f(x)||w(x)| \d x\\
&=: I_1 + I_2 + I_3 + I_4 + I_5.
\end{align*}
We estimate the five terms separately.

For $I_1$, we compute by using \eqref{eq-K-K0} and H\"older's inequality
\begin{align*}
I_1 \le c R^{\sigma} [u]_{H^s(B_{2R}(x_0))} [w]_{H^s(B_{2R}(x_0))}.
\end{align*}

For $I_2$ and $I_3$, we observe that
\begin{align}
\label{eq:K-K0-help}
    |K_{x_0}(x,y) - K(x,y)| \le c |y - x_0|^{-n-2s+\sigma} \quad \forall x \in B_R(x_0), ~ y \in B_1 \setminus B_{2R}(x_0),
\end{align}
which follows from \eqref{eq-K-K0} and $|x-x_0| \leq |y-x_0|/2 < |x-y|$, and that by \autoref{lemma:Poincare-Wirtinger} applied to $w$ (with $R:= 2R$), we get
\begin{align}
\label{eq:w-PF}
     \|w\|_{L^1(B_R(x_0))} \le c R^{\frac{n}{2}} \|w\|_{L^2(B_R(x_0))} \le c R^{\frac{n}{2} + s } [w]_{H^s(\R^n)}.
\end{align}
Hence, by using \eqref{eq:K-K0-help}, \eqref{eq:w-PF}, and the Poincar\'e inequality (\autoref{prop-Poincare}) for $u$, we have
\begin{align*}
I_2 &\leq c \int_{B_{R}(x_0)}\int_{B_1 \setminus B_{2R}(x_0)} |u(x) - (u)_{R,x_0}||w(x)| \frac{1}{|y-x_0|^{n+2s-\sigma}} \d y \d x \\
&\leq c R^{\sigma-2s} \left(\int_{B_R(x_0)} |u(x) - (u)_{R,x_0}|^2 \d x \right)^{1/2} \left(\int_{B_{2R}(x_0)} |w(x)|^2 \d x \right)^{1/2} \\
&\leq c R^{\sigma} [u]_{H^s(B_R(x_0))} [w]_{H^s(\R^n)}.
\end{align*}

For $I_3$, we obtain by \eqref{eq:K-K0-help} and \eqref{eq:w-PF}
\begin{align*}
I_3
&\le c \left(\int_{B_1 \setminus B_{R}(x_0)} |u(y) - (u)_{R,x_0}| \frac{1}{|y-x_0|^{n+2s-\sigma}} \d y \right) \left(\int_{B_R(x_0)} |w(x)| \d x \right)\\
&\le cR^{\sigma + \frac{n}{2}-s} \tail_{\sigma,B_1}(u-(u)_{R,x_0};R,x_0) [w]_{H^s(\R^n)}.
\end{align*}

For $I_4$, we use \eqref{eq-K-K0}, $|x-y| \geq |y|/2$, \eqref{eq:w-PF}, and \autoref{lemma:locbd-bdry} to obtain
\begin{align*}
    I_4 &\le c \int_{B_{R}(x_0)}\int_{\R^n \setminus B_1} |u(x) - u(y)||w(x)| \frac{1}{|y|^{n+2s}} \d y \d x \\
    &\le c \left( \|u\|_{L^\infty(B_{1/2})} + \tail(u; 1, 0) \right) \|w\|_{L^1(B_R(x_0))} \\
    &\le c R^{\frac{n}{2} + s} \left( \Vert u \Vert_{L^1_{2s}(\R^n)} + \Vert d^{s-p}_{\Omega} f \Vert_{L^q(\Omega_1)} \right) [w]_{H^s(\R^n)}.
\end{align*}

Finally, for $I_5$ we use \eqref{eq-fw-integral} and obtain
\begin{align*}
    I_5 &\le c R^{\frac{n}{2} + p - \frac{n}{q}} \Vert d_{\Omega}^{s-p} f \Vert_{L^q(\Omega_R(x_0))} [w]_{H^s(\R^n)}.
\end{align*}
Note that the measure density condition \eqref{eq:measure-density} of $\Omega^c$ is used here.
This proves the desired result after summing up all the previous estimates, and dividing by $[w]_{H^s(\R^n)}$.
\end{proof}

Let us present two applications of the previous lemma.

\begin{lemma}
\label{lemma:freezing-l2}
In the same situation as in \autoref{lemma:freezing}, it holds that
\begin{align*}
\begin{split}
\|w\|_{L^1(B_R(x_0))}
&\le c R^{\sigma} \left( \|u\|_{L^1(B_R(x_0))} + R^n \tail_{\sigma,B_1}(u; R,x_0) \right) \\
&\quad + cR^{n + 2s} \left( \Vert u \Vert_{L^1_{2s}(\R^n)} + R^{-(s-p) - \frac{n}{q}} \Vert d^{s-p}_{\Omega} f \Vert_{L^q(\Omega_R(x_0))} \right)
\end{split}
\end{align*}
for some $c=c(n, s, \lambda, \Lambda, \sigma, p, q, c_0)>0$.
\end{lemma}

\begin{proof}
By using \autoref{lemma:freezing}, \autoref{lemma:Cacc-bdry}, and \eqref{eq:sigma-tail-est}, we have
\begin{align*}
[w]_{H^s(\R^n)}^2
&\le c R^{2\sigma} \left( [u]_{H^s(B_{2R}(x_0))}^2 + R^{n-2s} (u)_{R, x_0}^2 + R^{n-2s} \tail_{\sigma,B_1}(u;R,x_0)^2 \right) \\
&\quad + c R^{n + 2s} \left(\Vert u \Vert_{L^1_{2s}(\R^n)}^2 + R^{-2(s - p) - \frac{2n}{q}} \| d_{\Omega}^{s-p} f \|_{L^q(\Omega_R(x_0))}^2 \right) \\
&\le c R^{2\sigma-2s} \left( \|u\|_{L^2(B_{4R}(x_0))}^2 + R^n \tail_{\sigma,B_1}(u; R, x_0)^2 \right) \\
&\quad + c R^{n + 2s} \left( \Vert u \Vert_{L^1_{2s}(\R^n)}^2 + R^{-2(s - p) - \frac{2n}{q}} \Vert d^{s-p}_{\Omega} f \Vert_{L^q(\Omega_R(x_0))}^2 \right).
\end{align*}
Moreover, since $\Omega_{8R}(x_0) \subset \Omega_1$, \autoref{lemma:locbd-bdry} shows that
\begin{align*}
\|u\|_{L^2(B_{4R}(x_0))}^2
&\leq c R^{-n} \left( \|u\|_{L^1(B_{8R}(x_0))} + R^n \tail(u;4R,x_0) + R^{n+2s - (s-p) - \frac{n}{q}} \Vert d_\Omega^{s-p}f \Vert_{L^q(\Omega_R(x_0))} \right)^2 \\
&\leq c R^{-n} \left( \|u\|_{L^1(B_{R}(x_0))} + R^n \tail(u;R,x_0) + R^{n+2s -(s-p) - \frac{n}{q}} \Vert d_\Omega^{s-p}f \Vert_{L^q(\Omega_R(x_0))} \right)^2.
\end{align*}
Combining these two estimates and using again \eqref{eq:sigma-tail-est} yield
\begin{align}\label{eq-w-Hs}
\begin{split}
    [w]_{H^s(\R^n)}^2
    &\le c R^{-n+2\sigma-2s} \left( \|u\|_{L^1(B_R(x_0))} + R^{n} \tail_{\sigma,B_1}(u; R,x_0) \right)^2 \\
    &\quad + cR^{n + 2s} \left( \Vert u \Vert_{L^1_{2s}(\R^n)} + R^{-(s-p) - \frac{n}{q}} \Vert d^{s-p}_{\Omega} f \Vert_{L^q(\Omega_{R}(x_0))} \right)^2.
\end{split}
\end{align}
The desired result now follows from \eqref{eq:w-PF} and \eqref{eq-w-Hs}.
\end{proof}

\begin{lemma}
\label{lemma:freezing-ds}
In the same situation as in \autoref{lemma:freezing}, it holds that
\begin{align}
\label{eq:ds-w-estimate}
\begin{split}
\int_{\Omega_R(x_0)} \left| \frac{w}{d_{\Omega}^s} \right| \d x
&\le c R^{\sigma} \left( \int_{\Omega_R(x_0)} \left|\frac{u}{d_{\Omega}^s} \right| \d x  +  \max\{R,d_{\Omega}(x_0)\}^{-s} R^{n} \tail_{\sigma,B_1}(u;R,x_0) \right)\\
&\quad + cR^{n + s} \left( \Vert u \Vert_{L^1_{2s}(\R^n)} + R^{-(s-p) - \frac{n}{q}} \Vert d_{\Omega}^{s-p} f \Vert_{L^q(\Omega_{R}(x_0))} \right)
\end{split}
\end{align}
for some $c=c(n, s, \lambda, \Lambda, \sigma, p, q, c_0)>0$.
\end{lemma}

\begin{proof}
In case $B_{2R}(x_0) \cap \Omega^c \not=\emptyset$, we observe by using H\"older's inequality, \autoref{lemma:Hardy}, and \eqref{eq-w-Hs} that
\begin{align*}
    \left( \int_{\Omega_R(x_0)} \left| \frac{w}{d_{\Omega}^s} \right| \d x \right)^2
    &\le cR^{n} \int_{\Omega_R(x_0)} \left| \frac{w}{d_{\Omega_R(x_0)}^s} \right|^2 \d x \le c R^n [w]_{H^s(\R^n)}^2 \\
    &\le c R^{2\sigma} \left( R^{-s}\|u\|_{L^1(B_R(x_0))} + R^{n-s} \tail_{\sigma,B_1}(u; R,x_0) \right)^2 \\
    &\quad + cR^{2n + 2s} \left( \Vert u \Vert_{L^1_{2s}(\R^n)} + R^{-(s-p) - \frac{n}{q}} \Vert d^{s-p}_{\Omega} f \Vert_{L^q(\Omega_{R}(x_0))} \right)^2.
\end{align*}
Since $d_\Omega \leq 3R$ in $\Omega_R(x_0)$ and $\max\{R,d_{\Omega}(x_0)\} \le 2R$, we obtain the desired estimate \eqref{eq:ds-w-estimate} in this case.

In case $B_{2R}(x_0) \subset \Omega$, we have
\begin{align}\label{eq-d-comparable}
\frac{3}{2} d_{\Omega}(x_0) \ge d_{\Omega}(x_0) + R \ge d_{\Omega} \ge d_{\Omega}(x_0) - R \ge \frac{1}{2} d_{\Omega}(x_0) ~~ \text{ in } B_{R}(x_0).
\end{align}
By using \autoref{lemma:freezing-l2}, we obtain
\begin{align*}
\int_{\Omega_R(x_0)} \left| \frac{w}{d_{\Omega}^s} \right| \d x &\le c d_{\Omega}(x_0)^{-s} \|w\|_{L^1(B_R(x_0))} \\
&\le c R^{\sigma} \left( \int_{B_R(x_0)} \left| \frac{u}{d_\Omega^s(x_0)} \right| \d x + d_{\Omega}(x_0)^{-s} R^{n} \tail_{\sigma,B_1}(u;R,x_0) \right) \\
&\quad + c d_{\Omega}(x_0)^{-s} R^{n + 2s} \left( \Vert u \Vert_{L^1_{2s}(\R^n)} + R^{-(s-p) - \frac{n}{q}} \Vert d_{\Omega}^{s-p} f \Vert_{L^q(\Omega_{R}(x_0))} \right).
\end{align*}
Therefore, \eqref{eq-d-comparable} together with $R \leq \max\{R, d_\Omega(x_0)\} = d_\Omega(x_0)$ proves \eqref{eq:ds-w-estimate} in this case.
\end{proof}

\section{\texorpdfstring{$C^{s-\eps}$}{Cs-e} boundary regularity in flat Lipschitz domains via Morrey theory}
\label{sec:Morrey}

The goal of this section is to prove the $C^{s-\eps}$ boundary regularity result for solutions to nonlocal equations in flat Lipschitz domains from \autoref{thm:Cs-eps}. To prove it, we develop a nonlocal Morrey theory at the boundary.

\subsection{Morrey estimate in the translation invariant case}

The goal of this subsection is to prove the following nonlocal Morrey-type estimate for solutions to translation invariant problems.

\begin{lemma}
\label{lemma:T1}
Let $\Omega \subset \R^n$ be a Lipschitz domain with Lipschitz constant $\delta > 0$. Let $x_0 \in \Omega$ and $R>0$. Let $L$ be a translation invariant operator with kernel $K$ satisfying \eqref{eq:Kcomp} and let $v$ be a solution to 
\begin{align*}
\left\{
\begin{aligned}
Lv&=0 &&\text{in }\Omega_R(x_0), \\
v&=0 &&\text{in } B_R(x_0) \setminus \Omega.
\end{aligned}
\right.
\end{align*}
For given $\eps \in (0,s)$, there exist $c, \delta_0 >0$, depending only on $n,s,\lambda,\Lambda,\eps$, and $\Omega$, such that if $\delta \le \delta_0$, then
\begin{align*}
    \int_{\Omega_{\rho}(x_0)} \left|\frac{v}{d_{\Omega}^s} \right| \d x
    &\le c \left( \frac{\rho}{R} \right)^{n-\eps} \left(\int_{\Omega_{R}(x_0)} \left| \frac{v}{d_{\Omega}^s} \right| \d x + \max\{R,d_{\Omega}(x_0)\}^{-s} R^{n} \tail(v; R,x_0) \right)
\end{align*}
for any $0<\rho<R$.
\end{lemma}

In order to prove \autoref{lemma:T1}, we recall the following result from \cite{RoWe24}.

\begin{proposition}
\label{lemma:T2}
Let $\Omega,x_0,R$, and $L$ be given as in \autoref{lemma:T1} and let $v$ be a solution to
\begin{align*}
\left\{
\begin{aligned}
Lv&=f &&\text{in }\Omega_R(x_0), \\
v&=0 &&\text{in } B_R(x_0) \setminus \Omega,
\end{aligned}
\right.
\end{align*}
where $f \in L^\infty(\Omega_R(x_0))$. For given $\eps \in (0,s)$, there exist $c, \delta_0 >0$, depending only on $n,s,\lambda,\Lambda,\eps$, and $\Omega$, such that if $\delta \le \delta_0$, then
\begin{align}\label{eq-T2}
    [ v ]_{C^{s-\eps}(\overline{B_{R/2}(x_0)})} \le c R^{-(s-\eps)} \left( \Vert v \Vert_{L^{\infty}(\Omega_{R}(x_0))} + \tail(v;R,x_0) + R^{2s} \|f\|_{L^\infty(\Omega_R(x_0))} \right).
\end{align}
\end{proposition}

\begin{proof}
    The result was proved for distributional solutions in \cite[Theorem~6.8]{RoWe24} (see also \cite[Proposition~2.5.4]{FeRo24} for an analogous result for homogeneous kernels $K$) in case $R = 1$. The result for general $R$ follows immediately by rescaling. Note that in particular, it holds for weak solutions by \cite[Lemma~2.2.32]{FeRo24}.
\end{proof}

\begin{remark}\label{rmk:T2}
\autoref{lemma:T2} also holds for operators that are only translation invariant in a ball, as in \eqref{eq:locally-freeze}. Indeed, if $\widetilde{L},L$ are operators with kernels $\widetilde{K}, K$, both satisfying \eqref{eq:Kcomp}, where $K$ is translation invariant and $\widetilde{K}(x, y)=K(x, y)$ for all $x, y$ with $|x-y| < 2R$, and $v$ is a solution to
\begin{align*}
\left\{
\begin{aligned}
\widetilde{L}v&=f &&\text{in }\Omega_R(x_0), \\
v&=0 &&\text{in } B_R(x_0) \setminus \Omega,
\end{aligned}
\right.
\end{align*}
then $v$ solves $Lv = (L-\widetilde{L})v+f$ in $\Omega_R(x_0)$. Note that we have for $x \in \Omega_{R}(x_0)$
\begin{align*}
    |(L - \widetilde{L}) v(x)| &\le 2 \int_{\R^n \setminus B_{2R}(x)} |v(x) - v(y)||K(x,y) - \widetilde{K}(x,y)| \d y \\
    &\le c  \int_{\R^n \setminus B_{R}(x_0)} (|v(x)| + |v(y)|) |x_0 -y|^{-n-2s} \d y \\
    &\le cR^{-2s} \Vert v \Vert_{L^{\infty}(\Omega_{R}(x_0))} + c R^{-2s} \tail(v;R,x_0).
\end{align*}
Hence, by applying \autoref{lemma:T2}, we get the exact same estimate \eqref{eq-T2} for $v$. This result will be used in Section \ref{sec:Cs} and applied to $\widetilde{K}_{x_0}$ defined in \eqref{eq:locally-freeze}.
\end{remark}

As a consequence, we are in a position to show the Morrey-type estimate at the boundary for solutions to translation invariant problems.

\begin{proof}[Proof of \autoref{lemma:T1}]
Let $\delta_0$ be the constant given in \autoref{lemma:T2} and assume $\delta \leq \delta_0$.
Note that we may assume without loss of generality that $\rho \le R/8$.
Let us first prove the desired result for balls $B_{\rho}(x_0)$ such that $B_{\rho}(x_0) \cap \Omega^c \not=\emptyset$.
Since $v$ vanishes on the boundary of $\Omega$ and since, by assumption, for any $x \in B_{\rho}(x_0)$, the projection to $\partial \Omega$ is inside $B_{2\rho}(x_0)$, we have
\begin{align*}
\int_{\Omega_{\rho}(x_0)} \left|\frac{v}{d_{\Omega}^s} \right| \d x 
&\le \int_{\Omega_{2\rho}(x_0)} d_{\Omega}^{-\eps}(x) \sup_{y \in \Omega_{2\rho}(x_0)}\frac{|v(x) - v(y)|}{|x-y|^{s-\eps}} \d x \\
&\le \left( \int_{\Omega_{2\rho}(x_0)} d_{\Omega}^{-\eps} \d x \right) [v]_{C^{s-\eps}(\overline{\Omega_{2\rho}(x_0)})} \le
c \rho^{n-\eps} [ v ]_{C^{s-\eps}(\overline{\Omega_{R/4}(x_0)})},
\end{align*}
where we applied \autoref{lemma:distance-integral} in the last estimate. Next, using the $C^{s-\eps}$-boundary estimate from \autoref{lemma:T2}, the local boundedness estimate (\autoref{lemma:locbd-bdry}), and the fact that $d_{\Omega} \le 2R$ in $\Omega_{R}(x_0)$, we deduce
\begin{align}
\label{eq:T1-help1}
\begin{split}
\int_{\Omega_{\rho}(x_0)} \left|\frac{v}{d_{\Omega}^s} \right| \d x
&\le c R^{-(s-\varepsilon)} \rho^{n-\eps} \left( \Vert v \Vert_{L^{\infty}(\Omega_{R/2}(x_0))} + \tail(v; R/2,x_0) \right)\\
&\le c R^{-s} \left( \frac{\rho}{R} \right)^{n-\eps} \left( \|v\|_{L^1(\Omega_R(x_0))} + R^n \tail(v; R, x_0) \right)\\
&\le c \left( \frac{\rho}{R} \right)^{n-\eps} \left(\int_{\Omega_{R}(x_0)} \left| \frac{v}{d_{\Omega}^s} \right| \d x + R^{n-s} \tail(v ; R, x_0) \right).
\end{split}
\end{align}
The desired result follows from \eqref{eq:T1-help1} and the fact that $\max\{R, d_\Omega(x_0)\}=R$ in case $B_\rho(x_0) \cap \Omega^c \neq \emptyset$.

Next, let us assume that $B_{4\rho}(x_0) \subset \Omega$. In that case, we have
\begin{align*}
\frac{5}{4} d_{\Omega}(x_0) \ge d_{\Omega}(x_0) + \rho \ge d_{\Omega} \ge d_{\Omega}(x_0) - \rho \ge \frac{3}{4} d_{\Omega}(x_0) ~~ \text{ in } B_{\rho}(x_0).
\end{align*}
Note that if also $B_R(x_0) \cap \Omega^c = \emptyset$, then by the local boundedness estimate (\autoref{lemma:locbd-bdry}) we obtain
\begin{align}
\label{eq:T1-help3}
\begin{split}
\int_{B_{\rho}(x_0)} \left|\frac{v}{d_{\Omega}^s} \right| \d x
&\le c d_{\Omega}(x_0)^{-s} \rho^{n} \Vert v \Vert_{L^{\infty}(B_{R/2}(x_0))} \\
&\le c d_{\Omega}(x_0)^{-s} \left( \frac{\rho}{R} \right)^n \left( \|v\|_{L^1(B_R(x_0))} + R^n \tail(v; R,x_0) \right) \\
&\leq c \left( \frac{\rho}{R} \right)^{n-\varepsilon} \left( \int_{B_{R}(x_0)} \left| \frac{v}{d_{\Omega}^s} \right| \d x + \max\{R,d_{\Omega}(x_0)\}^{-s} R^{n} \tail(v; R,x_0) \right),
\end{split}
\end{align}
where we used in the last step that $d_{\Omega} \le d_{\Omega}(x_0) + R \le 2 d_{\Omega}(x_0)$ in $B_R(x_0)$.

If $B_{4\rho}(x_0) \subset \Omega$ and $B_R(x_0) \cap \Omega^c \not=\emptyset$, then we apply \eqref{eq:T1-help1} with $\rho := \rho$ and $R := d_{\Omega}(x_0)$, and then observe that
\begin{align}
\label{eq:one-scale-up-Morrey}
\int_{\Omega_{d_{\Omega}(x_0)}(x_0)} \left|\frac{v}{d_{\Omega}^s} \right| \d x\le \int_{\Omega_{2d_{\Omega}(x_0)}(x_0)} \left|\frac{v}{d_{\Omega}^s} \right| \d x.
\end{align}
The desired result follows from \eqref{eq:T1-help1} applied with $\rho := 2 d_{\Omega}(x_0)$ and $R := R$.

Finally, note that if $B_{4\rho}(x_0) \not\subset \Omega$ and also $B_{\rho}(x_0) \cap \Omega^c = \emptyset$, then we can just apply \eqref{eq:one-scale-up-Morrey} first, and then use \eqref{eq:T1-help1}.
\end{proof}

\subsection{Regularity up to the boundary}

In this section, we prove \autoref{thm:Cs-eps}.

First, we establish the following Morrey-type estimate for solutions with respect to $L$, having $C^{\sigma}$ continuous coefficients. Recall that the excess functional $\Phi_{\sigma}(u; \rho):=\Phi_{\sigma}(u; \rho, x_0)$ is given by \eqref{eq-excess-Phi} in the next lemma.

\begin{lemma}
\label{lemma:Morrey-boundary}
Assume that we are in the same setting as in \autoref{thm:Cs-eps}. Then, there exist $c,\delta_0 > 0$, and $R_0 \in (0,\frac{1}{16})$, depending only on $n,s,\lambda,\Lambda,\sigma,p,q,\eps$, and $\Omega$, such that if $\delta \le \delta_0$, then the following holds true:
For any $x_0 \in \Omega_{1/2}$ and $0 < \rho \leq R \leq R_0$ it holds that
\begin{align*}
    \Phi_{\sigma}(u; \rho) \leq c \left( \frac{\rho}{R} \right)^{n+\min\{p-\frac{n}{q}, -\varepsilon\}} \Phi_{\sigma}(u; R) + c \rho^{n+\min\{p-\frac{n}{q}, -\varepsilon\}} \left( \Vert u \Vert_{L^1_{2s}(\R^n)} + \Vert d_{\Omega}^{s-p} f \Vert_{L^q(\Omega_1)} \right).
\end{align*}
\end{lemma}

\begin{proof}
Let $\delta_0$ be the constant given in \autoref{lemma:T1} and assume $\delta \leq \delta_0$.
We claim that for any $x_0 \in \Omega_{1/2}$ and $0 < \rho \leq R < \frac{1}{16}$ it holds
\begin{align}
\label{eq:Morrey-boundary-help-0}
\Phi_\sigma(u; \rho) \le c \left[ \left( \frac{\rho}{R} \right)^{n-\frac{\eps}{2}} + R^{\sigma} \right] \Phi_\sigma(u; R) + c R^{n+s - (s-p) -\frac{n}{q}} \left( \Vert u \Vert_{L^1_{2s}(\R^n)} + \Vert d_{\Omega}^{s-p} f \Vert_{L^q(\Omega_1)} \right).
\end{align}
The desired result follows immediately from \eqref{eq:Morrey-boundary-help-0} by an application of the standard iteration lemma in \autoref{lem-iteration} (with $\alpha:=n-\frac{\varepsilon}{2}$ and $\beta=:n+\min\{p-\frac{n}{q}, -\varepsilon\}<\alpha$, and recalling \eqref{eq-almost-incr-Phi}).

Let $v$ and $w$ be solutions to \eqref{eq-v-Sect4} and \eqref{eq-w-Sect4}, respectively. As a consequence of \autoref{lemma:T1} and \autoref{lemma:freezing-ds}, we obtain using also \eqref{eq:sigma-tail-est} and the fact that $u=v$ in $\R^n \setminus \Omega_R(x_0)$
\begin{align}
\begin{split}
\label{eq:Morrey-boundary-help1}
\int_{\Omega_{\rho}(x_0)}  \left| \frac{u}{d_{\Omega}^s} \right| \d x
&\le \int_{\Omega_{\rho}(x_0)} \left| \frac{v}{d_{\Omega}^s} \right| \d x + \int_{\Omega_{\rho}(x_0)} \left| \frac{w}{d_{\Omega}^s} \right| \d x \\
&\le c \left( \frac{\rho}{R} \right)^{n-\frac{\eps}{2}} \Phi_{\sigma}(v; R) + cR^{n+s} \tail(u; 1, 0) + \int_{\Omega_{\rho}(x_0)} \left| \frac{w}{d_{\Omega}^s} \right| \d x \\
&\le c \left( \frac{\rho}{R} \right)^{n-\frac{\eps}{2}} \Phi_\sigma(u; R) + cR^{n+s} \|u\|_{L^1_{2s}(\R^n)} + c \int_{\Omega_{R}(x_0)} \left| \frac{w}{d_{\Omega}^s} \right| \d x \\
&\le c \left[ \left( \frac{\rho}{R} \right)^{n-\frac{\eps}{2}} + R^{\sigma} \right] \Phi_\sigma(u; R) + c R^{n+s - (s-p) -\frac{n}{q}} \left( \Vert u \Vert_{L^1_{2s}(\R^n)} + \Vert d_{\Omega}^{s-p} f \Vert_{L^q(\Omega_1)} \right).
\end{split}
\end{align}

From here, it remains to show that
\begin{align}
\label{eq:tail-Morrey-u}
\begin{split}
\max&\{\rho,d_{\Omega}(x_0)\}^{-s} \rho^{n} \tail_{\sigma,B_1}(u;\rho, x_0)  \\
&\le c \left[ \left( \frac{\rho}{R} \right)^{n-\frac{\eps}{2}} + R^{\sigma} \right] \Phi_{\sigma}(u; R) + c R^{n+s - (s-p) -\frac{n}{q}} \left( \Vert u \Vert_{L^1_{2s}(\R^n)} + \Vert d_{\Omega}^{s-p} f \Vert_{L^q(\Omega_1)} \right).
\end{split}
\end{align}
To prove this, let us first consider the case $B_{\rho}(x_0) \cap \Omega^c \not= \emptyset$. Then, we take $m \in \N$ such that $2^{-m} R \leq \rho < 2^{-m+1}R$ and compute
\begin{align*}
\max&\{\rho,d_{\Omega}(x_0)\}^{-s} \rho^{n} \tail_{\sigma,B_1}(u ; \rho,x_0) \\
&\le c (2^{-m}R)^{n-s} \sum_{k = 1}^{m} (2^{-m}R)^{2s-\sigma} \int_{B_{2^{k-m}R}(x_0) \setminus B_{2^{k-m-1}R}(x_0)} \frac{|u(y)|}{|y - x_0|^{n+2s-\sigma}} \d y \\
&\quad + c (2^{-m}R)^{n-s} 2^{-m(2s-\sigma)} \tail_{\sigma,B_1}(u;R,x_0) \\
&=: I_1 + I_2.
\end{align*}
Note that for $I_2$, we can simply estimate since $\sigma < s$:
\begin{align*}
I_2 \le c \left( \frac{\rho}{R} \right)^{n+s-\sigma} R^{n-s} \tail_{\sigma,B_1}(u;R,x_0) \le c \left( \frac{\rho}{R} \right)^{n-\frac{\eps}{2}} \Phi_{\sigma}(u; R).
\end{align*}
For $I_1$, we deduce, using that $d_{\Omega} \le 2 (2^{k-m}R)$ in $B_{2^{k-m}R}(x_0)$ and applying \eqref{eq:Morrey-boundary-help1}:
\begin{align*}
I_1
&\le c (2^{-m}R)^{-s} \sum_{k = 1}^m 2^{-k(n+2s-\sigma)} \int_{B_{2^{k-m}R}(x_0)} |u| \d x \\
&\leq c \sum_{k = 1}^m 2^{-k(n+s-\sigma)} \int_{\Omega_{2^{k-m}R}(x_0)} \left| \frac{u}{d_{\Omega}^s} \right| \d x \\
&\le c \sum_{k = 1}^m 2^{-k(n+s-\sigma)} \left[ 2^{(k-m)(n-\frac{\varepsilon}{2})} + R^{\sigma} \right] \Phi_\sigma(u; R) \\
&\quad + c \sum_{k = 1}^m 2^{-k(n+s-\sigma)} R^{n+s-(s-p)-\frac{n}{q}} \left( \Vert u \Vert_{L^1_{2s}(\R^n)} + \Vert d_{\Omega}^{s-p} f \Vert_{L^q(\Omega_1)} \right) \\
&\le c \left[ \left( \frac{\rho}{R} \right)^{n-\frac{\varepsilon}{2}} + R^{\sigma} \right] \Phi_{\sigma}(u; R) + c R^{n+s - (s-p) -\frac{n}{q}} \left( \Vert u \Vert_{L^1_{2s}(\R^n)} + \Vert d_{\Omega}^{s-p} f \Vert_{L^q(\Omega_1)} \right).
\end{align*}
Here, the sums $\sum_{k=1}^m 2^{-k(s-\sigma+\frac{\varepsilon}{2})}$ and $\sum_{k = 1}^m 2^{-k(n+s-\sigma)}$ are finite since $\sigma<s$. This shows \eqref{eq:tail-Morrey-u} in case $B_{\rho}(x_0) \cap \Omega^c \not= \emptyset$.

Next, we assume that $B_R(x_0) \subset \Omega$. The proof of \eqref{eq:tail-Morrey-u} goes by the exact same arguments as before, with the only difference that we have now $\max\{2^{k-m}R, d_{\Omega}(x_0)\} = \max\{R, d_{\Omega}(x_0)\} = d_{\Omega}(x_0)$ for every $k \in \{0, 1, \dots, m\}$ and thus also $d_{\Omega} \le 2 d_{\Omega}(x_0)$ in $B_{2^{k-m}R}(x_0)$, and that we use \eqref{eq:T1-help3} instead of \eqref{eq:T1-help1}.

If $B_{4\rho}(x_0) \subset \Omega$ and $B_R(x_0) \cap \Omega^c \not=\emptyset$, we apply \eqref{eq:tail-Morrey-u} with $\rho := \rho$ and $R := d_{\Omega}(x_0)$, then observe that
\begin{align*}
d_{\Omega}(x_0)^{n-s}\tail_{\sigma,B_1}(u ; d_{\Omega}(x_0), x_0) \leq c\Phi_{\sigma}(u; 2d_\Omega(x_0)),
\end{align*}
and then use \eqref{eq:Morrey-boundary-help1} and \eqref{eq:tail-Morrey-u} with $\rho := 2 d_{\Omega}(x_0)$ and $R := R$ to deduce the desired result.

Finally, note that if $B_{4\rho}(x_0) \not\subset \Omega$ and also $B_{\rho}(x_0) \cap \Omega^c = \emptyset$, then we can just apply \eqref{eq:one-scale-up-Morrey} first, and then use  \eqref{eq:Morrey-boundary-help1} and \eqref{eq:tail-Morrey-u}.
This establishes \eqref{eq:tail-Morrey-u}.
\end{proof}

Before we prove \autoref{thm:Cs-eps}, we need the following consequence of the interior regularity for solutions to nonlocal equations with H\"older continuous coefficients (see \cite{FeRo24b}).

\begin{lemma}
\label{lemma:interior-excess-decay}
Let $\sigma \in (0,s)$ and $\eps \in (0,\sigma)$. Let $x_0 \in \R^n$ and $R > 0$. Assume that $K$ satisfies \eqref{eq:Kcomp}, and \eqref{eq-K-cont} with $\mathcal{A} = B_R(x_0)$.
Let $u$ be a solution to
\begin{align*}
Lu = f ~~ \text{ in } B_R=B_R(x_0),
\end{align*}
where $f \in L^q(B_R)$ for some $q \in (\frac{n}{2s},\infty]$. Then, for any $0 < \rho \le R/2$
\begin{align*}
&\dashint_{B_{\rho}} |u - (u)_{\rho,x_0}| \d x \\
&\leq c \left( \frac{\rho}{R} \right)^{\min \{2s - \frac{n}{q} , 1+\sigma-\eps \} } \left( \dashint_{B_{R}} |u - (u)_{R,x_0}| \d x + \tail(u - (u)_{R,x_0}; R,x_0) + R^{2s - \frac{n}{q}} \Vert f \Vert_{L^q(B_{R})} \right),
\end{align*}
where $c=c(n,s,\lambda,\Lambda,\sigma,\eps,q)>0$.
\end{lemma}

\begin{proof}
By applying the regularity estimate from \autoref{prop:interior-regularity} to $u-(u)_{R,x_0}$ in $B_{R/2}$, we obtain
\begin{align*}
&\dashint_{B_{\rho}} |u - (u)_{\rho,x_0}| \d x \\
&\le \sup_{x,y \in B_{\rho}} |u(x) - u(y)| \le (2 \rho)^{\min \{2s - \frac{n}{q} , 1+\sigma-\eps \} } [u]_{C^{\min \{2s - \frac{n}{q} , 1+\sigma-\eps \} } (\overline{B_{\rho}})} \\
&\le c \rho^{\min \{2s - \frac{n}{q} , 1+\sigma-\eps \} } [u - (u)_{R,x_0}]_{C^{\min \{2s - \frac{n}{q} , 1+\sigma-\eps \} }(\overline{B_{R/2}})} \\
&\le c \left( \frac{\rho}{R} \right)^{\min \{2s - \frac{n}{q} , 1+\sigma-\eps \} } \left( \dashint_{B_{R}} |u - (u)_{R,x_0}| \d x + \tail(u - (u)_{R,x_0}; R,x_0) + R^{2s - \frac{n}{q}} \Vert f \Vert_{L^q(B_{R})} \right).
\end{align*}
This concludes the proof, as desired.
\end{proof}

We are now in a position to provide the proof of \autoref{thm:Cs-eps}.

\begin{proof}[Proof of \autoref{thm:Cs-eps}]
We denote by $\delta_0, R_0$ the constants given in \autoref{lemma:Morrey-boundary}. We assume $\delta \leq \delta_0$ and
\begin{align*}
\|u\|_{L^1_{2s}(\R^n)} + \| d^{s-p}_{\Omega} f\|_{L^q(\Omega_1)} \leq 1
\end{align*}
and claim that for all $x_0 \in \Omega_{1/2}$ and $0<\rho< R_0$ it holds that
\begin{align}\label{eq-claim}
\int_{\Omega_{1/2} \cap B_\rho(x_0)} |u - (u)_{\Omega_{1/2}\cap B_\rho(x_0)}| \d x \leq c \rho^{n + s + \min \{p -\frac{n}{q}, - \eps\} }.
\end{align}
Since $\mathcal{L}^{1,n + s + \min \{p -\frac{n}{q}, - \eps\} }(\Omega_{1/2}) = C^{s+\min\{p-\frac{n}{q}, - \eps\}}(\overline{\Omega_{1/2}})$, the desired result in case $q \ge \frac{n}{p}$ follows from \eqref{eq-claim}. In case $q < \frac{n}{p}$, we let $\eps$ so small that $\varepsilon < -p + \frac{n}{q}$, and then the desired result also follows.

Let us now prove the claim \eqref{eq-claim}. In case $d_\Omega(x_0)/2 \leq \rho< R_0$, we have
\begin{align}\label{eq-bdry}
\int_{\Omega_{1/2} \cap B_\rho(x_0)} |u - (u)_{\Omega_{1/2}\cap B_\rho(x_0)}| \d x \leq 2\int_{\Omega_{1/2} \cap B_\rho(x_0)} |u| \d x \leq c \rho^{s} \int_{\Omega_\rho(x_0)} \left| \frac{u}{d_\Omega^s} \right| \d x.
\end{align}
Moreover, by \autoref{lemma:Morrey-boundary},  we obtain
\begin{align*}
\rho^{s} \int_{\Omega_\rho(x_0)} \left| \frac{u}{d_\Omega^s} \right| \d x \le c \rho^{n +s + \min \{p -\frac{n}{q}, - \eps\} } \left( \Phi_{\sigma}(u; R_0) + 1 \right).
\end{align*}
Note that the constant $c > 0$ depends on $R_0$. Let $\eta \in C^\infty_c(B_{7/8})$ be such that $\eta=1$ on $\Omega_{3/4}$, $0\leq \eta \leq 1$, and $|\nabla \eta|\leq c$. Then, it follows from the Hardy inequality (\autoref{lemma:Hardy}), Caccioppoli inequality (\autoref{lemma:Cacc-bdry}), and local boundedness (\autoref{lemma:locbd-bdry}) that
\begin{align*}
    \|u/d_\Omega^s\|_{L^1(\Omega_{R_0}(x_0))} \leq c\|u\eta/d_\Omega^s\|_{L^2(B_1)} \leq c[u\eta]_{H^s(\R^n)}
    \leq c[u]_{H^s(B_{15/16})} + c\|u\|_{L^2(B_{7/8})} \leq c.
\end{align*}
Since $\tail_{\sigma,B_1}(u;R_0,x_0) \le c \tail(u;R_0,x_0) \leq c \|u\|_{L^1(B_1)} + c\tail(u;1,0) \leq c$, we obtain
\begin{align}\label{eq:safe}
    \rho^{s} \int_{\Omega_\rho(x_0)} \left| \frac{u}{d_\Omega^s} \right| \d x \leq c\rho^{n +s +\min \{p -\frac{n}{q}, - \eps \} }.
\end{align}
The claim \eqref{eq-claim} follows from \eqref{eq-bdry} and \eqref{eq:safe} in this case.

Next, we consider the case $\rho < d_{\Omega}(x_0)/2 < R_0$. We apply \eqref{eq-E-inf} and the interior excess decay (see \autoref{lemma:interior-excess-decay}) with $R := d/2:=d_\Omega(x_0)/2$, and obtain
\begin{align*}
\int_{\Omega_{1/2} \cap B_\rho(x_0)} |u - (u)_{\Omega_{1/2}\cap B_\rho(x_0)}| \d x
&\leq c \int_{B_{\rho}(x_0)} |u - (u)_{\rho,x_0}| \d x \\
&\le c \left( \frac{\rho}{d} \right)^{n+ s+\min\{p-\frac{n}{q}, -\varepsilon\}} \int_{B_{d/2}(x_0)} |u - (u)_{d/2,x_0}| \d x \\
&\quad + c \left( \frac{\rho}{d} \right)^{n+ s+\min\{p-\frac{n}{q}, -\varepsilon\}}d^n\tail(u - (u)_{d/2,x_0}; d/2,x_0) \\
&\quad + c \left(\frac{\rho}{d}\right)^{n+ s+\min\{p-\frac{n}{q}, -\varepsilon\}} d^{n+2s-\frac{n}{q}} \Vert f \Vert_{L^q(B_{d/2}(x_0))},
\end{align*}
noting that $\min\{2s-\frac{n}{q}, 1+\sigma-\varepsilon\} \geq s+\min\{p-\frac{n}{q}, -\varepsilon\}$.
The estimate \eqref{eq:safe} applied with $\rho := d/2$ shows that
\begin{align*}
\int_{B_{d/2}(x_0)} |u - (u)_{d/2,x_0}| \d x \leq 2\int_{B_{d/2}(x_0)} |u| \d x \leq c d^{s} \int_{B_{d/2}(x_0)} \left| \frac{u}{d_\Omega^s} \right| \d x \leq cd^{n+s+\min\{p-\frac{n}{q}, -\varepsilon\}}.
\end{align*}
Moreover, by applying  \eqref{eq:sigma-tail-est} and \autoref{lemma:Morrey-boundary}, we deduce
\begin{align*}
    d^n \tail(u - (u)_{d/2,x_0}; d/2,x_0)
    &\leq cd^n\tail_{\sigma,B_1}(u; d/2,x_0) + cd^{n+2s} \tail(u;1,0) + cd^n(u)_{d/2,x_0} \\
    &\leq cd^{n+s+\min\{p-\frac{n}{q}, -\varepsilon\}} \left( \Phi_{\sigma}(u; R_0) + 1\right) \leq cd^{n+s+\min\{ p-\frac{n}{q}, -\varepsilon\}}.
\end{align*}
We also observe that $d_{\Omega} \geq d/2$ in $B_{d/2}(x_0)$ and hence $|f| \le cd^{p-s} (d_{\Omega}^{s-p} |f|)$ in $B_{d/2}(x_0)$. Therefore, we arrive at the claim \eqref{eq-claim} in case $\rho < d_{\Omega}(x_0)/2 < R_0$.

The result for the remaining case $d_\Omega(x_0)/2 \geq R_0$ follows immediately from the interior excess decay estimate in \autoref{lemma:interior-excess-decay}. The proof is complete.
\end{proof}

\section{Optimal boundary regularity for inhomogeneous kernels}
\label{sec:Cs}

The goal of this section is to prove the following optimal $C^s$ regularity result for general kernels $K$ satisfying \eqref{eq:Kcomp} and \eqref{eq-K-cont}.

\begin{theorem}
\label{thm:inhom-Cs}
Let $\alpha, \sigma \in (0,s)$. Let $\Omega \subset \R^n$ be a $C^{1, \alpha}$ domain with $0 \in \partial \Omega$. Assume that $K$ satisfies \eqref{eq:Kcomp}, and \eqref{eq-K-cont} with $\mathcal{A} = B_1$. 
Let $u$ be a solution to
\begin{align*}
\left\{
\begin{aligned}
Lu&=f &&\text{in }\Omega_1, \\
u&=0 &&\text{in } B_1 \setminus \Omega,
\end{aligned}
\right.
\end{align*}
where $f$ is such that $d_{\Omega}^{s-p}f \in L^q(\Omega_1)$ for some $p \in (0, s]$ and $q \in (\frac{n}{p},\infty]$. Then $u \in C^{s}_{\mathrm{loc}}(\overline{\Omega} \cap B_1)$ and
\begin{align*}
\left\Vert u \right\Vert_{C^{s}(\overline{\Omega_{1/8}})} \le c \left( \Vert u \Vert_{L^{1}_{2s}(\R^n)} + \Vert d^{s-p}_{\Omega} f \Vert_{L^q(\Omega_1)} \right),
\end{align*}
where $c=c(n,s,\lambda,\Lambda,\alpha,\sigma,p, q, \Omega) > 0$.
\end{theorem}

Note that \autoref{thm:inhom-Cs} immediately implies the first part of \autoref{thm:main-2}.

\subsection{Inhomogeneous barriers}

The goal of this subsection is to construct suitable barriers with respect to symmetric operators that are translation invariant in a ball, say $B_{1/2}$, and to establish several properties of these barriers. Such kernels will arise in \eqref{eq:locally-freeze}, where we freeze non-translation invariant kernels at a point $x_0$, but only locally. Note that this allows us to prove our main result \autoref{thm:inhom-Cs} for kernels $K$ that satisfy \eqref{eq-K-cont} only locally. Working with completely translation invariant kernels in this subsection would force us to assume \eqref{eq-K-cont} for $K$ with $\mathcal{A} = \R^n$ in \autoref{lemma:w-regularity}.

Let $L$ be a translation invariant operator with kernel $K$ satisfying \eqref{eq:Kcomp} and $\widetilde{L}$ be an operator with kernel $\widetilde{K}$ satisfying \eqref{eq:Kcomp} and $\widetilde{K}(x, y)=K(x, y)$ for all $x, y$ with $|x-y|<1/2$. In order to construct barriers, let us fix a $C^{1,\alpha}$ domain $\Omega$ with $0 \in \partial\Omega$ for some $\alpha \in (0,s)$ and take a set $D \subset \R^n$ with $\partial D \in C^{1,\alpha}$ and $\Omega_1 \subset D \subset \Omega_2$. We define the barrier function $\psi$ with respect to $\widetilde{L}$ as the solution to
\begin{align}
\label{eq:translation-invariant-solution}
\left\{
\begin{aligned}
    \widetilde{L}\psi&=0 &&\text{in } D, \\
    \psi&=g &&\text{in } D^c.
\end{aligned}
\right.
\end{align}
where $g \in C^{\infty}_c(\R^n \setminus B_3)$ is such that $0 \le g \le 1$ and $g \not\equiv 0$. The existence of $\psi$ follows by the well-posedness of the nonlocal Dirichlet problem.
Note that in particular $\psi \equiv 0$ in $B_3 \setminus D$.

The following proposition establishes several crucial properties of $\psi$. Its proof heavily relies on \cite[Theorem~6.9]{RoWe24}, which applies only to translation invariant operators, but it can be also applied to $\psi$ since $\psi$ solves
\begin{align}\label{eq-Lpsi}
    \left\{
    \begin{aligned}
        L\psi &= (L-\widetilde{L})\psi &&\text{in } D, \\
        \psi &= 0 &&\text{in } B_3.
    \end{aligned}
    \right.
\end{align}
Notice that $(L-\widetilde{L})\psi \in L^\infty(D)$ since $0\leq\psi\leq 1$ by the comparison principle and hence for any $x \in D$
\begin{align*}
\begin{split}
    |(L-\widetilde{L})\psi(x)|
    &\leq 2\int_{\R^n \setminus B_{1/2}(x)} |\psi(x)-\psi(y)| |K(x, y)-\widetilde{K}(x, y)| \d y \\
    &\leq c \int_{\R^n \setminus B_{1/2}(x)} |x-y|^{-n-2s} \d y \leq c.
\end{split}
\end{align*}

\begin{proposition}
\label{prop:inhom-Cs}
Let $\Omega,D,L,\widetilde{L},g,\psi$ be as before. Then there exists $C_1=C_1(n, s, \lambda, \Lambda, \alpha, \Omega)>0$ such that the following hold true:
\begin{itemize}
\item[(i)] $\psi \in C^s(\R^n)$ and
\begin{align*}
\|\psi\|_{C^s(\R^n)} \le C_1.
\end{align*}
\item[(ii)] It holds $\psi \le C_1 d_D^s$ in $D$ and
\begin{align*}
    C_1^{-1} d_\Omega^s \le \psi \le C_1 d_\Omega^s ~~ \text{ in } \Omega_{1/4}.
\end{align*}
\item[(iii)] Let $x_0 \in \Omega_{1/2}$ and $R > 0$ with $\Omega_R(x_0) \subset \Omega_{1/2}$, and $v$ be a solution to
\begin{align*}
\left\{
\begin{aligned}
Lv&=0 &&\text{in }\Omega_R(x_0), \\
v&=0 &&\text{in }B_R(x_0) \setminus \Omega.
\end{aligned}
\right.
\end{align*}
Then, for any $\eps \in (0,\alpha s)$, there exists $C_2=C_2(n,s,\lambda,\Lambda,\alpha,\Omega,\eps)> 0$ such that
for any $z \in B_{R/2}(x_0) \cap \partial \Omega$ there exists $q_z \in \R$ such that
\begin{align}
\label{eq:expansion-psi}
|v(x) - q_z \psi(x)| \le C_2 R^{-s-\eps} |x-z|^{s+\eps} \left(R^{-n} \Vert v \Vert_{L^1(B_R(x_0))} + \tail(v;R,x_0) \right) ~~ \forall x \in \Omega_{R/2}(x_0).
\end{align}
\end{itemize}
\end{proposition}

\begin{proof}
The property (i) and the upper bounds in (ii) are an immediate consequence of \cite[Theorem~6.9]{RoWe24}, using that $\partial D \in C^{1, \alpha}$.

Let us now prove the lower bound in (ii) and the property (iii). To do so, let us recall that by \cite[Theorem~6.9]{RoWe24}, for every $z \in B_{R/2}(x_0) \cap \partial \Omega$ there exists $q_z^{(1)}\in \R$ such that for any $x \in \Omega_{R/2}(x_0)$
\begin{align}
\label{eq:v-b-exp}
|v(x) - q^{(1)}_z b_{\nu_z}((x-z) \cdot \nu_z)| &\le c_1 |x-z|^{s+\eps} R^{-s-\eps} \left(R^{-n}\Vert v \Vert_{L^1(B_R(x_0))} + \tail(v;R,x_0) \right) ,
\end{align}
and moreover, for any $z \in B_{1/2} \cap \partial \Omega$ there exists $q_z^{(2)}\in \R$ such that for any $x \in \Omega_{1/2}$
\begin{align}
\label{eq:psi-b-exp}
|\psi(x) - q^{(2)}_z b_{\nu_z}((x-z) \cdot \nu_z)| &\le c_2 |x-z|^{s+\eps},
\end{align}
where 
\begin{align}
\label{eq:qz-upper-help}
|q^{(1)}_z| \le c_3 R^{-s} \left(R^{-n}\Vert v \Vert_{L^1(B_R(x_0))} + \tail(v;R,x_0) \right), \qquad |q^{(2)}_z| \le c_4
\end{align}
for some constants $c_1,c_2,c_3,c_4 > 0$. Here, the function $b_{\nu_z}$ denotes the half-space solution from \cite[Theorem~1.4]{RoWe24} with respect to the operator $L$ and the normal vector $\nu_z \in \mathbb{S}^{n-1}$ at $z \in \partial \Omega$.
Clearly, it must be
\begin{align*}
q^{(2)}_z = \lim_{\Omega_{1/2} \ni x \to z} \frac{\psi(x)}{b_{\nu_z}((x-z) \cdot \nu_z)} \ge 0,
\end{align*}
and \eqref{eq:psi-b-exp} in particular implies that the limit exists. We claim that there exists a constant $c_0 > 0$, depending only on $n,s,\lambda,\Lambda,\alpha,\Omega$ such that for any $z \in B_{1/2} \cap \partial \Omega$ it holds
\begin{align}
\label{eq:improved-Hopf}
q^{(2)}_z \ge c_0.
\end{align}
We will prove this property at the end of the proof. Let us first explain that it implies the lower bound in (ii), as well as (iii). 

The lower bound in (ii) is immediate upon following \cite[Step 3 in the proof of Theorem~6.10]{RoWe24}. Indeed, for any $x \in \Omega_{1/4}$, taking $z \in B_{1/2} \cap \partial \Omega$ such that $d_{\Omega}(x) = |x-z| = (x-z) \cdot \nu_z$, we find due to \eqref{eq:psi-b-exp}, \eqref{eq:improved-Hopf}, and \cite[Theorem~1.4]{RoWe24}:
\begin{align*}
\psi(x) \ge q^{(2)}_z b_{\nu_z}((x-z) \cdot \nu_z) - |\psi(x) - q^{(2)}_z b_{\nu_z}((x-z) \cdot \nu_z)| \ge c_0 c |x-z|^s - c_2 |x-z|^{s+\eps}.
\end{align*}
Thus, there exist a radius $\delta_0 > 0$ and $c_5 > 0$, depending only on $n,s,\lambda,\Lambda,\alpha,\Omega$ such that for any $x \in B_{1/4} \cap \{d_{\Omega} \le \delta_0 \}$ it holds
\begin{align*}
\psi(x) \ge c_5 d^s_{\Omega}(x).
\end{align*}
Finally, since by the weak Harnack inequality for any $x \in B_{1/4} \cap \{d_{\Omega} > \delta_0 \}$:
\begin{align*}
\psi(x) \ge c(\delta_0) \Vert \psi \Vert_{L^1_{2s}(\R^n)} \ge c(\delta_0) \Vert g \Vert_{L^1_{2s}(\R^n \setminus D)} 4^s d_{\Omega}^s(x),
\end{align*}
we conclude the proof of the lower bound in (ii).

Let us now turn to the proof of (iii). In fact, by a combination of \eqref{eq:v-b-exp} and \eqref{eq:psi-b-exp}, we deduce for any $z \in B_{R/2}(x_0) \cap \partial \Omega$, denoting $q_z := q_z^{(1)}/q_z^{(2)}$:
\begin{align*}
|v(x) - q_z \psi(x)| &\le |v(x) - q_z^{(1)}b_{\nu_z}((x-z) \cdot \nu_z)| + \frac{|q_z^{(1)}|}{q_z^{(2)}}|q_z^{(2)} b_{\nu_z}((x-z) \cdot \nu_z) - \psi(x)| \\
&\le (c_1 + c_0^{-1} c_2 c_3 )|x-z|^{s+\eps} R^{-s-\eps} \left(R^{-n}\Vert v \Vert_{L^1(B_R(x_0))} + \tail(v;R,x_0) \right),
\end{align*}
where we also used \eqref{eq:qz-upper-help} and \eqref{eq:improved-Hopf}. Moreover, by \eqref{eq:qz-upper-help}, we have 
\begin{align*}
  |q_z| \le c_0^{-1}c_3  R^{-s} \left(R^{-n}\Vert v \Vert_{L^1(B_R(x_0))} + \tail(v;R,x_0) \right),
\end{align*}
which concludes the proof of (iii).

It remains to prove the claim \eqref{eq:improved-Hopf}. To do so, let us assume that there exist sequences of translation invariant operators $L_k$ with kernels $K_k$ satisfying \eqref{eq:Kcomp} and operators $\widetilde{L}_k$ with kernels $\widetilde{K}_k$ satisfying \eqref{eq:Kcomp} and $\widetilde{K}_k(x, y)=K_k(x, y)$ for all $x, y$ with $|x-y|<1/4$, and solutions $\psi_k$ to
\begin{align*}
\left\{
\begin{aligned}
\widetilde{L}_k \psi_k &= 0 &&\text{in } D,\\
\psi_k &= g &&\text{in } D^c,
\end{aligned}
\right.
\end{align*}
such that we have as in \eqref{eq:psi-b-exp} for any $z \in B_{1/2} \cap \partial \Omega$
\begin{align}
\label{eq:psi-k-exp}
|\psi_k(x) - q^{(k)}_z b_{\nu_z}((x-z) \cdot \nu_z)| \le C |x-z|^{s+\eps} ~~ \forall x \in \Omega_{1/2},
\end{align}
and that $\inf_{z \in B_{1/2} \cap \partial \Omega} q^{(k)}_z \to 0$, as $k \to \infty$. 

Moreover, by (i), we have $\Vert \psi_k \Vert_{C^s(\R^n)} \le c$, and therefore by the Arzela--Ascoli theorem, we have $\psi_k \to \psi_{\infty}$ locally uniformly in $\R^n$. By the stability of nonlocal equations (see \cite[Proposition~2.2.36]{FeRo24}), there exist $\psi_{\infty} \in C^s_{\mathrm{loc}}(\R^n)$ and an operator $\widetilde{L}_{\infty}$ with kernel $\widetilde{K}_\infty$ satisfying \eqref{eq:Kcomp} and $\widetilde{K}_\infty(x, y)=\widetilde{K}_\infty(x-y)$ for all $x, y$ with $|x-y|<1/4$ such that
\begin{align*}
\left\{
\begin{aligned}
\widetilde{L}_\infty \psi_\infty &= 0 &&\text{in } D,\\
\psi_\infty &= g &&\text{in } D^c.
\end{aligned}
\right.
\end{align*}
Applying again \cite[Theorem~6.9]{RoWe24}, for any $z \in B_{1/2} \cap \partial \Omega$ there exists $q_z^{(\infty)} \in \R$ such that 
\begin{align}
\label{eq:psi-infty-exp}
|\psi_{\infty}(x) - q^{(\infty)}_z b_{\nu_z}((x-z) \cdot \nu_z)| \le C |x-z|^{s+\eps} ~~ \forall x \in \Omega_{1/2}.
\end{align} 
Now, let us fix any $z \in B_{1/2} \cap \partial \Omega$ and let $x \in B_{1/2}$ be such that $|x-z| = (x-z) \cdot \nu_z = d_{\Omega}(x)$ to be chosen explicitly later. Then, we have for any $k \in \N$, using that $b_{\nu_z}((x-z) \cdot \nu_z) \asymp |x-z|^s$
\begin{align*}
|q_z^{(\infty)} - q_z^{(k)}| &\le c |x-z|^{-s} \left| q_z^{(\infty)} b_{\nu_z}((x-z) \cdot \nu_z) - q_z^{(k)} b_{\nu_z}((x-z) \cdot \nu_z) \right|  \\
&\le c |x-z|^{-s} \Big( |q_z^{(\infty)} b_{\nu_z}((x-z) \cdot \nu_z) - \psi_{\infty}(x)| \\
& \qquad\qquad\qquad + |\psi_{\infty}(x) - \psi_k(x)| + |\psi_k(x) - q_z^{(k)} b_{\nu_z}((x-z) \cdot \nu_z)| \Big) \\
&\le c|x-z|^{\eps} + c \frac{\Vert \psi_{\infty} - \psi_k \Vert_{L^{\infty}(B_{1/2})}}{|x-z|^s},
\end{align*}
where we used \eqref{eq:psi-infty-exp}, \eqref{eq:psi-k-exp}.
Now, let $\delta > 0$. Choosing $x$ so that $c|x-z|^{\eps} = \delta/4$ and then taking $k$ so large that 
\begin{align*}
c \frac{\Vert \psi_{\infty} - \psi_k \Vert_{L^{\infty}(B_{1/2})}}{|x-z|^s} = c \delta^{-s/\eps} \Vert \psi_{\infty} - \psi_k \Vert_{L^{\infty}(B_{1/2})} \le \frac{\delta}{4},
\end{align*}
which is possible due to the uniform convergence $\psi_k \to \psi_{\infty}$ in $B_{1/2}$, we obtain
\begin{align*}
|q_z^{(\infty)} - q_z^{(k)}| \le \frac{\delta}{2}.
\end{align*}

Since we assumed that $\inf_{z \in B_{1/2} \cap \partial \Omega} q^{(k)}_z \to 0$, for any $\delta \in (0,1)$ there exist points $z_k \in \partial \Omega \cap B_{1/2}$ and $k_0 \in \N$ such that for any $k \ge k_0$ it holds $|q_{z_k}^{(k)}| \le \frac{\delta}{2}$. Clearly, $z_k \to z \in \partial \Omega \cap B_{1/2}$ possibly up to a subsequence, and therefore by the previous considerations, choosing $k$ large enough, it holds
\begin{align*}
|q_z^{(\infty)}| \le |q_z^{(\infty)} - q_z^{(k)}| + |q_z^{(k)}| \le \delta.
\end{align*}
Thus, since $\delta > 0$  was arbitrary, we have $\inf_{z \in B_{1/2} \cap \partial \Omega} q_z^{(\infty)} = 0$, which contradicts Step 2c in the proof of \cite[Theorem~6.10]{RoWe24}. Thus, we have established \eqref{eq:improved-Hopf}.

Finally, note that all the constants throughout the proof might also depend on $g,D$. However, $g,D$ only depend on $n,s,\Omega$, and therefore the constants in (i), (ii) only depend on those objects.
\end{proof}

We have the following consequence of the boundary $C^s$ regularity of $\psi$ combined with the interior regularity, which was established in \cite{FeRo24}.

\begin{lemma}
\label{lemma:psi-interior-reg}
Let $\psi$ be as in \eqref{eq:translation-invariant-solution}. Then, for any $\eps \in (0, s)$ it holds that
\begin{align}
\label{eq:psi-cdc}
\int_{\R^n} \frac{|\psi(x) - \psi(y)|^2}{|x-y|^{n+2s}} \d y \le c (1 + d_{D}^{-\eps}(x)) ~~ \forall x \in D,
\end{align}
and in particular $[\psi]_{H^s(D | \R^n)} \le c$, where $c=c(n,s,\lambda,\Lambda,\alpha,\Omega) > 0$.
\end{lemma}

\begin{proof}
Let us fix $x \in D$ and estimate
\begin{align*}
\int_{\R^n} \frac{|\psi(x) - \psi(y)|^2}{|x-y|^{n+2s}} \d y
&\le \int_{B_{d_D(x)/2}(x)} \frac{|\psi(x) - \psi(y)|^2}{|x-y|^{n+2s}} \d y + \int_{\R^n \setminus B_{d_D(x)/2}(x)} \frac{|\psi(x)-\psi(y)|^2}{|x-y|^{n+2s}} \d y \\
&=: I_1 + I_2.
\end{align*}
For $I_1$, we first recall that $\psi$ is a solution to \eqref{eq-Lpsi}. It thus follows from the interior regularity from \cite[Theorem~2.4.3]{FeRo24} applied in $B_{d_D(x)/2}(x)$ with $\gamma \in (0,2s)$ and \autoref{prop:inhom-Cs}(ii) that
\begin{align}
\label{eq:psi-interior-reg}
[\psi]_{C^{\gamma}(\overline{B_{d_{D}(x)/2}(x)})} \le c d_D^{-\gamma}(x) \left( \Vert \psi \Vert_{L^{\infty}(B_{d_D(x)}(x))} + \tail(\psi;d_D(x),x) + 1\right) \le c d_D^{s-\gamma}(x).
\end{align}
Note that we applied \autoref{prop:inhom-Cs}(ii) and \cite[Lemma~B.2.4]{FeRo24} to estimate
\begin{align*}
    \tail(\psi;d_D(x),x)
    \leq d_D^{2s}(x) \int_{D \setminus B_{d_D(x)}(x)} \frac{d_D^s(y)}{|y-x|^{n+2s}} \d y + d_D^{2s}(x) \int_{\R^n \setminus B_3} \frac{g(y)}{|y-x|^{n+2s}} \d y \leq c(1+d_D^s(x)).
\end{align*}
The estimate \eqref{eq:psi-interior-reg} applied with $\gamma = s + \eps$ yields
\begin{align*}
I_1 \le c d_D^{-2\eps}(x) \int_{B_{d_D(x)/2}(x)} |x-y|^{-n+2\eps} \d y \le c.
\end{align*}

For $I_2$, we apply \autoref{prop:inhom-Cs}(ii) and \cite[Lemma~B.2.4]{FeRo24} again to obtain
\begin{align*}
    I_2 \leq 2\int_{D \setminus B_{d_D(x)/2}(x)} \frac{d_D^{2s}(x) + d_D^{2s-\varepsilon}(y)}{|x-y|^{n+2s}} \d y + \int_{\R^n \setminus B_3} \frac{|g(y)|^2}{|x-y|^{n+2s}} \d y \leq c(1+d_D^{-\varepsilon}(x)).
\end{align*}
This proves \eqref{eq:psi-cdc}. The last assertion follows from \eqref{eq:psi-cdc} and \autoref{lemma:distance-integral}.
\end{proof}

\subsection{Closeness of barriers with respect to frozen kernels}

In this section, we assume that we are in the setup of \autoref{thm:inhom-Cs}, i.e.\ let $K$ be a kernel satisfying \eqref{eq:Kcomp}, and \eqref{eq-K-cont} with $\mathcal{A} = B_1$.
Given $x_0 \in \Omega_{1/2}$, we consider an operator $\widetilde{L}_{x_0}$ with kernel $\widetilde{K}_{x_0}$ frozen in $B_{1/2}$, that is,
\begin{align}
\label{eq:locally-freeze}
    \widetilde{K}_{x_0}(x, y) = K_{x_0}(x, y) \1_{B_{1/2}}(x-y) + K(x, y) \1_{B_{1/2}^c}(x-y),
\end{align}
where $K_{x_0}$ is the completely frozen kernel at $x_0$ given by \eqref{eq-frozen}. Let $D \subset \R^n$ with $\partial D \in C^{1,\alpha}$ and $\Omega_1 \subset D \subset \Omega_2$ be as before.  We define $\psi_{x_0}$ as the barrier function with respect to $\widetilde{L}_{x_0}$ as follows:
\begin{align}
\label{eq:psi-PDE}
\left\{
\begin{aligned}
\widetilde{L}_{x_0} \psi_{x_0} &= 0 &&\text{in } D,\\
\psi_{x_0} &= g &&\text{in } D^c,
\end{aligned}
\right.
\end{align}
where $g \in C^{\infty}_c(\R^n \setminus B_3)$ is such that $0 \le g \le 1$ and $g \not\equiv 0$, as in \eqref{eq:translation-invariant-solution}.

A central goal of this section is to establish the following proposition.

\begin{proposition}
\label{lemma:w-regularity}
Let $K, \Omega$ be as in \autoref{thm:inhom-Cs}. Let $D$ be as before. Then, for any $\eps \in (0,\sigma)$ and $x_0, y_0 \in \Omega_{1/2}$
\begin{align*}
[\psi_{x_0} - \psi_{y_0}]_{C^{s-\eps}(\overline{\Omega_{1/2}})} \le c |x_0 - y_0|^{\sigma},
\end{align*}
where $c > 0$ depends only on $n,s,\lambda,\Lambda,\alpha,\sigma,\eps$, and $\Omega$.
\end{proposition}

The first step is to prove an energy estimate. In order to prove the following lemma, we extend the arguments from the proof of \autoref{lemma:freezing} to equations of the form
\begin{align*}
\cE(u,\varphi) = \cE(\psi,\varphi) 
\quad\forall \varphi \in C_c^{\infty}(\Omega),
\end{align*}
where $\psi \in C^s(\R^n)$ satisfies \eqref{eq:psi-cdc}.

\begin{lemma}
\label{lemma:energy-est}
Let $K, \Omega$ be as in \autoref{thm:inhom-Cs}. Let $D$ be as before. Then for any $x_0, y_0 \in \Omega_{1/2}$
\begin{align*}
[\psi_{x_0} - \psi_{y_0}]_{H^s(\R^n)} \le c |x_0 - y_0|^{\sigma},
\end{align*}
where $c > 0$ depends only on $n,s,\lambda,\Lambda,\alpha$, and $\Omega$.
\end{lemma}

\begin{proof}
The function $u := \psi_{x_0} - \psi_{y_0}$ solves
\begin{align}\label{eq-L-y0-u}
\left\{
\begin{aligned}
\widetilde{L}_{y_0} u &= (\widetilde{L}_{y_0} - \widetilde{L}_{x_0}) \psi_{x_0} &&\text{in }D, \\
u&=0 &&\text{in } D^c.
\end{aligned}
\right.
\end{align}
We establish an energy estimate for $u$ in the spirit of \autoref{lemma:freezing}. 
By construction, and due to the regularity assumption \eqref{eq-K-cont} on $K$, we have
\begin{align}
\label{eq:K-frozen-difference}
\begin{split}
|\widetilde{K}_{x_0}(x,y) - \widetilde{K}_{y_0}(x,y)| &\le \frac{1}{2} |K(x_0 + x - y , x_0) - K(y_0 + x  - y , y_0)| \1_{B_{1/2}}(x-y) \\
&\quad + \frac{1}{2} |K(x_0 + y - x , x_0) - K(x_0 + y - x , x_0)| \1_{B_{1/2}}(x-y) \\
&\le \Lambda |x_0 - y_0|^{\sigma} |x-y|^{-n-2s} \1_{B_{1/2}}(x-y).
\end{split}
\end{align}
Testing the equation for $u$ by $u$ and using \eqref{eq:K-frozen-difference}, we deduce
\begin{align*}
\lambda [u]^2_{H^s(\R^n)} &\le \cE^{\widetilde{K}_{y_0}}(u,u) = \cE^{\widetilde{K}_{y_0} - \widetilde{K}_{x_0}}(\psi_{x_0},u) \\
&\le 2\Lambda \int_D \int_{\R^n} |\psi_{x_0}(x) - \psi_{x_0}(y)||u(x) - u(y)| \frac{|x_0-y_0|^\sigma}{|x-y|^{n+2s}} \d y \d x \\
&\le 2\Lambda |x_0 - y_0|^{\sigma} [\psi_{x_0}]_{H^s(D|\R^n)} [u]_{H^s(\R^n)}.
\end{align*}
Since $[\psi_{x_0}]_{H^s(D|\R^n)} \le c$ for some $c=c(n,s,\alpha,\lambda, \Lambda,\Omega) > 0$ by \autoref{lemma:psi-interior-reg}, the desired result follows.
\end{proof}

We are in a position to give the proof of \autoref{lemma:w-regularity}. Its proof is similar to that of \autoref{thm:Cs-eps}.

\begin{proof}[Proof of \autoref{lemma:w-regularity}]
Let us denote $u = \psi_{x_0} - \psi_{y_0}$ and recall from the proof of \autoref{lemma:energy-est} that $u$ is a solution to \eqref{eq-L-y0-u}. Given $z_0 \in \Omega_{1/2}$ and $R\in (0, 1/2)$, let $v$ be the solution to
\begin{align*}
\left\{
\begin{aligned}
\widetilde{L}_{y_0}v&=0 &&\text{in }\Omega_R(z_0), \\
v&=u &&\text{in } \R^n \setminus \Omega_R(z_0),
\end{aligned}
\right.
\end{align*}
and denote $w=u-v$. Then $w$ solves
\begin{align*}
\left\{
\begin{aligned}
\widetilde{L}_{y_0} w &= (\widetilde{L}_{y_0} - \widetilde{L}_{x_0})\psi_{x_0} &&\text{ in } \Omega_R(z_0),\\
w &= 0 &&\text{ in } \R^n \setminus \Omega_R(z_0).
\end{aligned}
\right.
\end{align*}
Note that here we used that $\Omega_R(z_0) = D \cap B_R(z_0) \subset D$. As in the proof of \autoref{lemma:energy-est}, by testing the equation for $w$ with $w$, and then using \eqref{eq:K-frozen-difference} and H\"older's inequality, we obtain
\begin{align*}
\lambda [w]_{H^s(\R^n)}^2 \le \cE^{\widetilde{K}_{y_0}}(w,w) = \cE^{\widetilde{K}_{y_0}-\widetilde{K}_{x_0}}(\psi_{x_0},w) \le 2\Lambda |x_0 - y_0|^{\sigma} [\psi_{x_0}]_{H^s(\Omega_R(z_0) | \R^n)} [w]_{H^s(\R^n)}.
\end{align*}
By \autoref{lemma:psi-interior-reg} and \autoref{lemma:distance-integral}, for any $\eps \in (0, s)$ it holds that
\begin{align*}
    [\psi_{x_0}]_{H^s(\Omega_R(z_0) | \R^n)}^2 \le c \int_{\Omega_R(z_0)} (1 + d_{D}^{-\eps}) \d x \leq \int_{\Omega_R(z_0)} (1 + d_{\Omega_R(z_0)}^{-\eps}) \d x \le c R^{n-\eps},
\end{align*}
and hence
\begin{align*}
[w]_{H^s(\R^n)}^2 \le c |x_0 - y_0|^{2\sigma} [\psi_{x_0}]_{H^s(\Omega_R(z_0) | \R^n)}^2 \le c |x_0 - y_0|^{2\sigma} R^{n-\eps}.
\end{align*}
This estimate is a counterpart of \autoref{lemma:freezing} for the equation for $u$. Clearly, by applications of H\"older's inequality, the Poincar\'e inequality (\autoref{lemma:Poincare-Wirtinger}) with $R:=2R$, and the Hardy inequality (\autoref{lemma:Hardy}), we also obtain
\begin{align}
\label{eq:w-regularity-help-1}
R^{-s}\|w\|_{L^1(\Omega_R(z_0))} + \int_{\Omega_R(z_0)} \left| \frac{w}{d_{\Omega_R(z_0)}^s} \right| \d x \leq c |x_0 - y_0|^{\sigma} R^{n-\frac{\eps}{2}}.
\end{align}

From here, we follow the lines of the proof of \autoref{thm:Cs-eps}. For this purpose, we obtain estimates corresponding to \autoref{lemma:Morrey-boundary} and \autoref{lemma:interior-excess-decay}. First, by applying \autoref{lemma:Morrey-boundary} (with $\varepsilon:=\varepsilon/2$) to $v$ instead of \autoref{lemma:T1}, and \eqref{eq:w-regularity-help-1} to $w$ instead of \autoref{lemma:freezing-ds} in the proof of \autoref{lemma:Morrey-boundary}, and then using the iteration lemma (see \autoref{lem-iteration} and \eqref{eq-almost-incr-Phi}), we deduce that there exists $R_0 \in (0, \frac{1}{16})$ such that for any $z_0 \in \Omega_{1/2}$ and $0<\rho<R\leq R_0$
\begin{align}
\begin{split}
\label{eq:w-regularity-help-2}
    \Phi_{\sigma}(u; \rho, z_0)
    &\leq c \left( \frac{\rho}{R} \right)^{n-\varepsilon} \Phi_{\sigma}(u; R, z_0) + c \rho^{n-\varepsilon} \left( \|v\|_{L^1_{2s}(\R^n)} + |x_0-y_0|^\sigma \right) \\
    &\leq c \left( \frac{\rho}{R} \right)^{n-\varepsilon} \Phi_{\sigma}(u; R, z_0) + c \rho^{n-\varepsilon} \left( \|u\|_{L^1_{2s}(\R^n)} + |x_0-y_0|^\sigma \right).
\end{split}
\end{align}
Note that we used that $\|v\|_{L^1_{2s}(\R^n)} \leq \|u\|_{L^1_{2s}(\R^n)} + \|w\|_{L^1(\Omega_R(z_0))}$ and \eqref{eq:w-regularity-help-1} in the last inequality.

Next, we apply \autoref{lemma:interior-excess-decay} to $v$ and use \eqref{eq:w-regularity-help-1} and obtain that whenever $0 < \rho < R\leq R_0$ and $B_{2R}(z_0) \subset \Omega_{1/2}$,
\begin{align}
\label{eq:w-regularity-help-3}
\begin{split}
\int_{B_{\rho}(z_0)} |u - (u)_{\rho,z_0}| \d x
&\leq \int_{B_{\rho}(z_0)} |v - (v)_{\rho,z_0}| \d x + 2\int_{B_R(z_0)} |w| \d x \\
&\le c \left( \frac{\rho}{R} \right)^{n+\min\{2s, 1+\sigma-\varepsilon\}} \left( \int_{B_{R}(z_0)} |u - (u)_{R,z_0}| \d x + R^{n}\tail(u - (u)_{R,z_0}; R,z_0) \right) \\
&\quad + c R^{n+s-\frac{\eps}{2}} |x_0 - y_0|^{\sigma}.
\end{split}
\end{align}
Moreover, \eqref{eq:w-regularity-help-3} also holds true with $\rho^{n}\tail(u - (u)_{\rho,z_0}; \rho,z_0)$ on the left-hand side. This can be checked by following the proof of \autoref{lemma:Morrey-boundary}, replacing $u/d_{\Omega}^s$ by $u - (u)_{\rho,z_0}$ and using \eqref{eq:w-regularity-help-3}. Thus, we can apply \autoref{lem-iteration} and obtain
\begin{align*}
& \int_{B_{\rho}(z_0)} |u - (u)_{\rho,z_0}| \d x + \rho^{n}\tail(u - (u)_{\rho,z_0}; \rho,z_0)  \\
& \le c \left( \frac{\rho}{R} \right)^{n+\min\{2s, 1+\sigma-\varepsilon\}} \left( \int_{B_{R}(z_0)} |u - (u)_{R,z_0}| \d x + R^{n}\tail(u - (u)_{R,z_0}; R,z_0) \right) + c \rho^{n+s-\frac{\eps}{2}} |x_0 - y_0|^{\sigma}
\end{align*}
as usual. Note that \eqref{eq-almost-incr} can be verified in the same way as for $\Phi_{\sigma}(u;\rho,x_0)$.

As a consequence, we obtain the following estimate for any $z_0 \in \Omega_{1/2}$ and $0 < \rho < R_0$ by combining \eqref{eq:w-regularity-help-2} and \eqref{eq:w-regularity-help-3} in the same way as in the proof of \autoref{thm:Cs-eps}:
\begin{align*}
\rho^{-n-s+\varepsilon} \int_{\Omega_{\rho}(z_0)} |u-(u)_{\Omega_{\rho}(z_0)}| \d x
&\le c \Phi_{\sigma}(u; R_0) + c\|u\|_{L^1_{2s}(\R^n)} + c |x_0 - y_0|^{\sigma} \\
&\le c [u]_{H^s(\R^n)} + c |x_0 - y_0|^{\sigma} \le c|x_0 - y_0|^{\sigma},
\end{align*}
where we used in the last step that $u \equiv 0$ in $\R^n \setminus D$, and the  Hardy inequality and \autoref{lemma:energy-est}.
Hence, it holds $u \in \mathcal{L}^{1,n+s-\eps}(\Omega_{1/2}) = C^{s-\eps}(\overline{\Omega_{1/2}})$ and the proof is complete.
\end{proof}

\subsection{\texorpdfstring{$C^s$}{Cs} regularity in \texorpdfstring{$C^{1,\alpha}$}{C1,alpha} domains}

In this section we establish our main result \autoref{thm:inhom-Cs}.

Having at hand the functions $\psi$ and $\psi_{x_0}$, we are ready to formulate the higher order Campanato estimate. We begin with the following estimate for translation invariant operators.

\begin{lemma}
\label{lemma:T4-inhom}
Let  $\alpha \in (0,s)$ and $\eps \in (0,\alpha s)$. Let $\Omega \subset \R^n$ be a $C^{1,\alpha }$ domain.
Let $L$ be a translation invariant operator with kernel $K$ satisfying \eqref{eq:Kcomp}. Let $x_0 \in \Omega_{1/8}$ and $0<\rho\leq R\leq \frac{1}{8}$.  Let $\psi$ be as in \autoref{prop:inhom-Cs} and $v$ be a solution to
\begin{align*}
\left\{
\begin{aligned}
Lv&=0 &&\text{in }\Omega_R(x_0), \\
v&=0 &&\text{in } B_R(x_0) \setminus \Omega.
\end{aligned}
\right.
\end{align*}
Then it holds
\begin{align*}
\begin{split}
&\int_{\Omega_{\rho}(x_0)} \left|\frac{v}{\psi} - \left( \frac{v}{\psi} \right)_{\Omega_{\rho}(x_0)} \right| \d x \\
&\leq c \left( \frac{\rho}{R} \right)^{n+\eps} \left[ \int_{\Omega_R(x_0)} \left|\frac{v}{\psi} - \left( \frac{v}{\psi} \right)_{\Omega_R(x_0)} \right| \d x + \max\{R,d_{\Omega}(x_0)\}^{-s} R^{n}  \mathrm{Tail} \left(v- \psi (v/\psi)_{\Omega_R(x_0)}; R, x_0 \right) \right],
\end{split}
\end{align*}
where $c > 0$ depends only on $n,s,\lambda,\Lambda,\alpha,\eps$, and $\Omega$.
\end{lemma}

\begin{proof}
We may assume that $\rho \leq R/4$. Let us consider the case $B_{\rho}(x_0) \cap \Omega^c \not= \emptyset$.
Let $c_0 \in \R$.
Let us denote by $q \in \R$ the factor from the expansion in  \eqref{eq:expansion-psi} corresponding to the projection of $x_0$ to $z \in \partial \Omega_{\rho}(x_0)$ for the function $v - c_0 \psi$. Note that $v-c_0 \psi$ satisfies all the assumptions from \autoref{prop:inhom-Cs}(iii). Then, we estimate by using \eqref{eq-E-inf} and \eqref{eq:expansion-psi},
\begin{align*}
\int_{\Omega_{\rho}(x_0)} \left|\frac{v}{\psi} - \left( \frac{v}{\psi} \right)_{\Omega_{\rho}(x_0)} \right| \d x &= \int_{\Omega_{\rho}(x_0)} \left|\frac{v - c_0 \psi}{\psi}(x) - \left( \frac{v - c_0 \psi}{\psi} \right)_{\Omega_{\rho}(x_0)} \right| \d x \\
& \le c \int_{\Omega_{\rho}(x_0)} \left|\frac{v - c_0 \psi}{\psi}(x) - q \right| \d x \\
&\le c \left( \int_{\Omega_{\rho}(x_0)} d_{\Omega}^{-s} \d x \right) \rho^{s+\eps} \sup_{x \in \Omega_{R/8}(x_0)} \left|\frac{(v - c_0 \psi)(x) - q \psi(x)}{|x-z|^{s+\eps}} \right|\\ 
&\le c \rho^{n+\eps} R^{-s-\eps} \left( \dashint_{\Omega_R(x_0)} |v-c_0 \psi| \d x + \mathrm{Tail}(v-c_0 \psi; R,  x_0) \right) \\
&\le c \left( \frac{\rho}{R} \right)^{n+\eps} \left( \int_{\Omega_{R}(x_0)} \left| \frac{v}{\psi}-c_0 \right| \d x + R^{n-s} \mathrm{Tail}(v-c_0 \psi; R,x_0) \right).
\end{align*}
Here, we also used that $c d^s_{\Omega} \le \psi \le cd^s_{\Omega} \le c R^s$ in $\Omega_R(x_0) \subset \Omega_{1/4}$ by \autoref{prop:inhom-Cs}(ii). Since $c_0 \in \R$ was arbitrary, we can choose $c_0 = (v/\psi)_{\Omega_R(x_0)}$, and conclude
\begin{align}
\label{eq:ti-Campanato-local-inhom}
\begin{split}
& \int_{\Omega_{\rho}(x_0)} \left|\frac{v}{\psi} - \left( \frac{v}{\psi} \right)_{\Omega_{\rho}(x_0)} \right| \d x \\
&\quad \le c \left( \frac{\rho}{R} \right)^{n+\eps} \left[ \int_{\Omega_{R}(x_0)} \left| \frac{v}{\psi} - \left( \frac{v}{\psi} \right)_{\Omega_R(x_0)} \right| \d x + R^{n-s} \mathrm{Tail} \left(v- \psi (v/\psi)_{\Omega_R(x_0)}; R, x_0 \right) \right].
\end{split}
\end{align}

Next, let us assume that $B_{R}(x_0) \subset \Omega$. By \eqref{eq-E-inf},
\begin{align*}
\int_{B_{\rho}(x_0)} \left|\frac{v}{\psi} - \left( \frac{v}{\psi} \right)_{\Omega_{\rho}(x_0)} \right| \d x
&\le c \int_{B_{\rho}(x_0)} \left|\frac{v}{\psi}(x) - \frac{v}{\psi}(x_0) \right| \d x \\
& = c \int_{B_{\rho}(x_0)} \left|\frac{v - c_0 \psi}{\psi}(x) - \frac{v - c_0 \psi}{\psi}(x_0) \right| \d x \\
&\le c \rho^{n+\varepsilon} \left[ \frac{v - c_0 \psi}{\psi} \right]_{C^{\varepsilon}(\overline{B_{R/4}(x_0)})} \\
&\leq c \rho^{n+\varepsilon} [v - c_0 \psi]_{C^{\varepsilon}(\overline{B_{R/4}(x_0)})} \Vert \psi^{-1} \Vert_{L^{\infty}(B_{R/4}(x_0))} \\
&\quad + c \rho^{n+\varepsilon} [\psi^{-1}]_{C^{\varepsilon}(\overline{B_{R/4}(x_0)})} \Vert v - c_0 \psi \Vert_{L^{\infty}(B_{R/4}(x_0))},
\end{align*}
where we denoted again $c_0 = (v/\phi)_{B_R(x_0)}$.

By the interior regularity estimate (see \autoref{lemma:T2} and \autoref{rmk:T2}) and the local boundedness (\autoref{lemma:locbd-bdry}), we obtain
\begin{align*}
    [v - c_0 \psi]_{C^{\varepsilon}(\overline{B_{R/4}(x_0)})} 
    &\leq cR^{-\varepsilon} \left( \Vert v-c_0\psi \Vert_{L^{\infty}(B_{R/2}(x_0))} + \tail(v-c_0\psi;R/2,x_0) \right) \\
    &\leq cR^{-n-\varepsilon} \left( \Vert v-c_0\psi \Vert_{L^1(B_{R}(x_0))} + R^n \tail(v-c_0\psi;R,x_0) \right) \\
    &\leq cR^{-n-\varepsilon} \left( d_\Omega^s(x_0) \int_{B_R(x_0)} \left| \frac{v}{\psi}-c_0 \right| \d x + R^n \tail(v-c_0\psi;R,x_0) \right),
\end{align*}
where we used in the last step that $R \le d_{\Omega}(x_0)$ and $\psi \leq c d_{\Omega}^s \le c d_{\Omega}^s(x_0)$ in $B_R(x_0)$. Moreover, since
\begin{align}
\label{eq:tail-phi-est}
\begin{split}
\tail(\psi;R,x_0)
&\le c R^{2s} \int_{B_{d_\Omega(x_0)/2}(x_0)} \frac{d_\Omega^s(x_0)}{|x-x_0|^{n+2s}} \d x + c R^{2s} \int_{\Omega_{1/2}(x_0) \setminus B_{d_\Omega(x_0)/2}(x_0)} \frac{1}{|x-x_0|^{n+s}} \d x \\
& \quad + c R^{2s} \int_{\R^d \setminus \Omega_{1/2}(x_0)} \frac{1}{|x-x_0|^{n+2s}} \d x \\
& \le c R^{2s} d_\Omega^{-s}(x_0) + c R^{2s} \le c d_\Omega^s(x_0),
\end{split}
\end{align}
\autoref{lemma:T2} and \autoref{rmk:T2} again show that
\begin{align*}
[\psi^{-1}]_{C^{\varepsilon}(\overline{B_{R/4}(x_0)})} &\le c \Vert \psi^{-1} \Vert_{L^{\infty}(B_{R/4}(x_0))}^2 [\psi]_{C^{\varepsilon}(\overline{B_{R/4}(x_0)})} \\
&\le c d_\Omega^{-2s}(x_0) R^{-\varepsilon} \left(\Vert \psi \Vert_{L^{\infty}(B_R(x_0))} + \tail(\psi;R,x_0) \right) \le c R^{-\varepsilon} d_\Omega^{-s}(x_0).
\end{align*}
Altogether, we have
\begin{align}\label{eq:campanato-inhom-BR}
\int_{B_{\rho}(x_0)} \left|\frac{v}{\psi} - \left( \frac{v}{\psi} \right)_{\Omega_{\rho}(x_0)} \right| \d x
&\leq c \left( \frac{\rho}{R} \right)^{n+\varepsilon} \left( \int_{B_R(x_0)} \left| \frac{v}{\psi}-c_0 \right| \d x + d_\Omega^{-s}(x_0)R^n \tail(v-c_0\psi;R,x_0) \right).
\end{align}

If $B_{4\rho}(x_0) \subset \Omega$ and $B_R(x_0) \cap \Omega^c \not=\emptyset$, we apply \eqref{eq:campanato-inhom-BR} with $\rho := \rho$ and $R := d_{\Omega}(x_0)$, use \eqref{eq-E-inf} to deduce
\begin{align}
\label{eq:one-scale-up-Campanato}
& \int_{\Omega_{d_{\Omega}(x_0)}(x_0)} \left|\frac{v}{\psi} - \left( \frac{v}{\psi} \right)_{\Omega_{d_{\Omega}(x_0)}(x_0)} \right| \d x \le c \int_{\Omega_{2d_{\Omega}(x_0)}(x_0)} \left|\frac{v}{\psi}  - \left( \frac{v}{\psi} \right)_{\Omega_{2d_{\Omega}(x_0)}(x_0)}\right| \d x,
\end{align}
and then apply \eqref{eq:ti-Campanato-local-inhom} with $\rho := 2 d_{\Omega}(x_0)$ and $R := R$ to deduce the desired result. Finally, note that if $B_{4\rho}(x_0) \not\subset \Omega$ and also $B_{\rho}(x_0) \cap \Omega^c = \emptyset$, then we can just apply \eqref{eq:one-scale-up-Campanato} first, and then use \eqref{eq:ti-Campanato-local-inhom}.
\end{proof}

As a consequence of the Campanato lemma for translation invariant kernels, we derive a Campanato iteration scheme for non-translation invariant kernels. Recall that the excess functional $\Psi_\sigma(u; \rho):=\Psi_\sigma(u; \rho, x_0)$ is defined by \eqref{eq-excess-Psi}, where $\psi_{x_0}$ is the barrier defined by \eqref{eq:psi-PDE}.

\begin{lemma}\label{lem-u-Holder-inhom}
Assume that we are in the same setting as in \autoref{thm:inhom-Cs}. Let $\varepsilon \in (0, \min\{\alpha s, s-\sigma\})$. Then for any $x_0 \in \Omega_{1/8}$ and $0<\rho\leq R \le \frac{1}{16}$ it holds
\begin{align*}
\Psi_{\sigma}(u; \rho) \leq c \left( \frac{\rho}{R} \right)^{n+\varepsilon} \Psi_{\sigma}(u; R) + cR^\sigma \Phi_{\sigma}(u; R) + c R^{n + s -(s-p)-\frac{n}{q}} \left( \Vert u \Vert_{L^1_{2s}(\R^n)} + \Vert d_{\Omega}^{s-p} f \Vert_{L^q(\Omega_{1})} \right)
\end{align*}
for some $c=c(n, s, \lambda, \Lambda, \alpha, p, q, \sigma, \Omega, \varepsilon) > 0$.
\end{lemma}

\begin{proof}
We write $\psi=\psi_{x_0}$ for simplicity.
Let $v$ and $w$ be solutions to \eqref{eq-v-Sect4} and \eqref{eq-w-Sect4}, respectively. Then, by applying \autoref{lemma:T4-inhom} and using \eqref{eq:sigma-tail-est}, we have
\begin{align*}
    \int_{\Omega_{\rho}(x_0)} \left|\frac{v}{\psi} - \left( \frac{v}{\psi} \right)_{\Omega_{\rho}(x_0)} \right| \d x
    \leq c \left( \frac{\rho}{R} \right)^{n+\varepsilon} \left( \Psi_\sigma(v;R) + R^{n+s} \tail(v-\psi(v/\psi)_{\Omega_{R}(x_0)}; 1,0) \right).
\end{align*}
By using \autoref{prop:inhom-Cs}(ii) and \cite[Lemma~B.2.4]{FeRo24}, we obtain
\begin{align*}
    \Psi_{\sigma}(v; R)
    &\leq \Psi_{\sigma}(u; R) + \max\{R, d_\Omega(x_0)\}^{-s} R^n \tail_{\sigma, B_1}(w-\psi(w/\psi)_{\Omega_R(x_0)}; R, x_0) \\
    &\leq \Psi_{\sigma}(u; R) + c \int_{\Omega_R(x_0)} \left| \frac{w}{\psi} \right| \d x.
\end{align*}
Here, we used that since $\psi(y) \le c d_{D}^s(y) \le c d_{\Omega}^s(x_0) + c |y - x_0|^s$ for $y \in B_1 \setminus B_R(x_0)$:
\begin{align*}
    &\max\{R, d_\Omega(x_0)\}^{-s} R^n \tail_{\sigma, B_1}(w-\psi(w/\psi)_{\Omega_R(x_0)}; R, x_0) \\
    &\quad\le c\left( \frac{w}{\psi} \right)_{\Omega_R(x_0)}  \max\{R , d_{\Omega}(x_0) \}^{-s} R^{n+2s-\sigma} \int_{B_1 \setminus B_R(x_0)} \frac{d_{\Omega}^s(x_0) + |y - x_0|^s}{|y-x_0|^{n+2s-\sigma}} \d y \le c \int_{\Omega_R(x_0)} \left| \frac{w}{\psi} \right| \d x.
\end{align*}
Moreover, since $\psi \le 1$ it holds $\tail(\psi;1,0) \le c$, and hence we have
\begin{align*}
    R^{n+s}\tail(v-\psi(v/\psi)_{\Omega_{R}(x_0)}; 1,0)
    &\leq R^{n+s} \tail(u-\psi(u/\psi)_{\Omega_{R}(x_0)}; 1,0) \\
    &\quad + R^{n+s}\tail(w-\psi(w/\psi)_{\Omega_{R}(x_0)}; 1,0) \\
    &\leq cR^{n+s}\|u\|_{L^1_{2s}(\R^n)} + cR^{n+s}(u/\psi)_{\Omega_{R}(x_0)} + cR^{n+s}(w/\psi)_{\Omega_{R}(x_0)} \\
    &\leq cR^{n+s}\|u\|_{L^1_{2s}(\R^n)} + cR^{\sigma} \Phi_{\sigma}(u;R) + c\int_{\Omega_R(x_0)} \left| \frac{w}{\psi} \right| \d x.
\end{align*}
As a consequence, we deduce
\begin{align*}
\int_{\Omega_{\rho}(x_0)} &\left|\frac{u}{\psi} - \left( \frac{u}{\psi} \right)_{\Omega_{\rho}(x_0)} \right| \d x \\
&\leq \int_{\Omega_{\rho}(x_0)} \left|\frac{v}{\psi} - \left( \frac{v}{\psi} \right)_{\Omega_{\rho}(x_0)} \right| \d x + \int_{\Omega_{\rho}(x_0)} \left|\frac{w}{\psi} - \left( \frac{w}{\psi} \right)_{\Omega_{\rho}(x_0)} \right| \d x \\
&\leq c \left( \frac{\rho}{R} \right)^{n+\varepsilon} \Psi_\sigma(u; R) + cR^{\sigma}\Phi_{\sigma}(u; R) + c R^{n+s} \Vert u \Vert_{L^1_{2s}(\R^n)} + c\int_{\Omega_R(x_0)} \left|\frac{w}{\psi}\right| \d x.
\end{align*}
By \autoref{lemma:freezing-ds} and the properties of $\psi$ from \autoref{prop:inhom-Cs}(ii), we obtain
\begin{align*}
\int_{\Omega_R(x_0)} \left|\frac{w}{\psi} \right| \d x
&\leq c \int_{\Omega_{R}(x_0)} \left|\frac{w}{d_\Omega^s} \right| \d x \leq cR^\sigma \Phi_{\sigma}(u; R) + cR^{n+s-(s-p)-\frac{n}{q}} \left( \Vert u \Vert_{L^1_{2s}(\R^n)} + \Vert d_{\Omega}^{s-p} f \Vert_{L^q(\Omega_{1})} \right).
\end{align*}
Thus, altogether, we have shown
\begin{align}
\label{eq:Campanato-boundary-help1}
\begin{split}
\int_{\Omega_{\rho}(x_0)} \left|\frac{u}{\psi} - \left( \frac{u}{\psi} \right)_{\Omega_{\rho}(x_0)} \right| \d x 
&\le c \left( \frac{\rho}{R} \right)^{n+\varepsilon} \Psi_\sigma(u; R) + cR^\sigma \Phi_{\sigma}(u; R) \\
&\quad + cR^{n + s-(s-p)-\frac{n}{q}} \left( \Vert u \Vert_{L^1_{2s}(\R^n)} + \Vert d_{\Omega}^{s-p} f \Vert_{L^q(\Omega_{1})} \right).
\end{split}
\end{align}

From here, it remains to show that
\begin{align}
\label{eq:tail-higher-campanato-u}
\begin{split}
\max&\{\rho,d_{\Omega}(x_0)\}^{-s} \rho^{n} \tail_{\sigma,B_1}(u - \psi (u/\psi)_{\Omega_{\rho}(x_0)};\rho,x_0) \\
&\le c \left( \frac{\rho}{R} \right)^{n+\varepsilon} \Psi_\sigma(u; R) + cR^\sigma \Phi_{\sigma}(u; R) + cR^{n + s-(s-p)-\frac{n}{q}} \left( \Vert u \Vert_{L^1_{2s}(\R^n)} + \Vert d_{\Omega}^{s-p} f \Vert_{L^q(\Omega_{1})} \right).
\end{split}
\end{align}

Let us first consider balls $B_{\rho}(x_0)$ such that $B_{\rho}(x_0) \cap \Omega^c \not=\emptyset$. We take $m \in \N$ such that $2^{-m} R < \rho \le 2^{-m+1}R$ and compute
\begin{align*}
\max&\{2^{-m}R,d_{\Omega}(x_0)\}^{-s} (2^{-m}R)^{n} \tail_{\sigma,B_1}(u- \psi (u/\psi)_{\Omega_{2^{-m}R}(x_0)}; 2^{-m}R , x_0) \\
&\le c (2^{-m}R)^{n-s} \sum_{k = 1}^m (2^{-m}R)^{2s - \sigma} \int_{B_{2^{k-m}R}(x_0) \setminus B_{2^{k-m-1}R}(x_0)} \frac{|u(y)- \psi(y) (u/\psi)_{\Omega_{2^{-m}R}(x_0)}|}{|y - x_0|^{n+2s-\sigma}} \d y \\
&\quad + c (2^{-m}R)^{n-s} 2^{-m(2s - \sigma)} \tail_{\sigma,B_1}(u - \psi (u/\psi)_{\Omega_{2^{-m}R}(x_0)};R,x_0) \\
&=: I_1 + I_2.
\end{align*}

Note that for $I_2$, we can estimate
\begin{align*}
I_2 &\le c 2^{-m(n+s-\sigma)} R^{n-s} \tail_{\sigma,B_1}(u - \psi (u/\psi)_{\Omega_{2^{-m}R}(x_0)};R,x_0) \\
&\le c \left( \frac{\rho}{R} \right)^{n+s-\sigma} \max\{R,d_{\Omega}(x_0)\}^{-s} R^n \tail_{\sigma,B_1}(u - \psi (u/\psi)_{\Omega_{2^{-m}R}(x_0)};R,x_0) \\
&\le c \left( \frac{\rho}{R} \right)^{n+s - \sigma} \max\{R,d_{\Omega}(x_0)\}^{-s} R^n \tail_{\sigma,B_1}(u - \psi (u/\psi)_{\Omega_{R}(x_0)};R,x_0) \\
&\quad+  c \left( \frac{\rho}{R} \right)^{n+s-\sigma} \max\{R,d_{\Omega}(x_0)\}^{-s} R^n |(u/\psi)_{\Omega_{2^{-m}R}(x_0)} - (u/\psi)_{\Omega_{R}(x_0)}| \tail_{\sigma,B_1}(\psi;R,x_0).
\end{align*}
Moreover, note that by \eqref{eq-E-inf} and \eqref{eq:Campanato-boundary-help1}, we have for any $k \le m$:
\begin{align}
\label{eq:average-increment}
\begin{split}
& |(u/\psi)_{\Omega_{2^{-m}R}(x_0)} - (u/\psi)_{\Omega_{2^{k-m} R}(x_0)}| \\
&\quad \le c \sum_{l = 0}^{k} \dashint_{\Omega_{2^{l-m} R}(x_0)} \left| \frac{u}{\psi}(x) - \left( \frac{u}{\psi} \right)_{\Omega_{2^{l-m}R}(x_0)} \right| \d x \\
&\quad \le c \sum_{l = 0}^{k} (2^{l-m}R)^{-n} (2^{l-m})^{n+\varepsilon} \Psi_\sigma(u;R) + c \sum_{l = 0}^{k} (2^{l-m}R)^{-n} R^{\sigma} \Phi_\sigma(u; R) \\
& \quad \quad + c \sum_{l = 0}^{k} (2^{l-m}R)^{-n} R^{n + s-(s-p)-\frac{n}{q}} \left( \Vert u \Vert_{L^1_{2s}(\R^n)} + \Vert d_{\Omega}^{s-p} f \Vert_{L^q(\Omega_{1})} \right) \\
&\quad \le c R^{-n} 2^{(k-m)\eps} \Psi_\sigma(u;R) + c \rho^{-n} R^{\sigma} \Phi_{\sigma}(u; R) \\
&\quad\quad  + c \rho^{-n}R^{n + s-(s-p)-\frac{n}{q}} \left( \Vert u \Vert_{L^1_{2s}(\R^n)} + \Vert d_{\Omega}^{s-p} f \Vert_{L^q(\Omega_{1})} \right),
\end{split}
\end{align}
where we also used that
\begin{align*}
    \sum_{l = 0}^{k} (2^{l-m}R)^{-n} 2^{(l-m)(n + \varepsilon)} \le R^{-n} \sum_{l = 0}^{k} 2^{(l-m)\varepsilon} \le c 2^{(k-m)\eps} R^{-n}, \quad \sum_{l = 0}^{k} (2^{l-m}R)^{-n} \leq  c\rho^{-n} \sum_{l = 0}^k 2^{-ln} \le c \rho^{-n}.
\end{align*}

Hence, since $\tail_{\sigma,B_1}(\psi;R,x_0) \le c \max\{R,d(x_0)\}^s$, which follows by a similar computation as in \eqref{eq:tail-phi-est}, and since $\varepsilon < s-\sigma$, we have shown
\begin{align*}
I_2
&\le c \left( \frac{\rho}{R} \right)^{n+s-\sigma} \Psi_\sigma(u;R) + c\left( \frac{\rho}{R} \right)^{s - \sigma} R^{\sigma}\Phi_\sigma(u; R) \\
&\quad + c \left( \frac{\rho}{R} \right)^{s-\sigma} R^{n + s} \left( \Vert u \Vert_{L^1_{2s}(\R^n)} + R^{-(s-p) - \frac{n}{q}} \Vert d_{\Omega}^{s-p} f \Vert_{L^q(\Omega_{1})} \right) \\
&\le c \left( \frac{\rho}{R} \right)^{n+\varepsilon} \Psi_\sigma(u; R) + cR^\sigma \Phi_{\sigma}(u; R) + cR^{n + s-(s-p)-\frac{n}{q}} \left( \Vert u \Vert_{L^1_{2s}(\R^n)} + \Vert d_{\Omega}^{s-p} f \Vert_{L^q(\Omega_{1})} \right).
\end{align*}

For $I_1$, we deduce by using that $\psi \leq cd_\Omega^s \leq c(2^{k-m}R)^s$ in $B_{2^{k-m}R}(x_0)$ (see \autoref{prop:inhom-Cs}(ii)) and \eqref{eq:average-increment}, and then applying \eqref{eq:Campanato-boundary-help1}:
\begin{align*}
I_1
&\leq c(2^{-m}R)^{-s} \sum_{k = 1}^m 2^{-k(n+2s-\sigma)} \int_{B_{2^{k-m}R}(x_0)} |u(y)- \psi(y) (u/\psi)_{\Omega_{2^{-m}R}(x_0)}| \d y \\
&\leq c \sum_{k = 1}^m 2^{-k(n+s-\sigma)} \int_{B_{2^{k-m}R}(x_0)} \left| \frac{u}{\psi}(y)- \left( \frac{u}{\psi} \right)_{\Omega_{2^{k-m}R}(x_0)} \right| \d y \\
&\quad + \sum_{k = 1}^m 2^{-k(n+s-\sigma)} (2^{k-m}R)^n |(u/\psi)_{\Omega_{2^{-m}R}(x_0)} - (u/\psi)_{\Omega_{2^{k-m} R}(x_0)}|\\
&\leq c \sum_{k = 1}^m 2^{-k(n+s-\sigma)+(k-m)(n+\varepsilon)} \Psi_{\sigma}(u; R) \\
&\quad + c \sum_{k = 1}^m 2^{-k(s-\sigma)} \left( R^\sigma\Phi_{\sigma}(u; R) + R^{n + s-(s-p)-\frac{n}{q}} \left( \Vert u \Vert_{L^1_{2s}(\R^n)} + \Vert d_{\Omega}^{s-p} f \Vert_{L^q(\Omega_{1})} \right) \right) \\
&\leq c\left( \frac{\rho}{R} \right)^{n+\varepsilon} \Psi_{\sigma}(u; R) + cR^\sigma\Phi_{\sigma}(u; R) + cR^{n + s-(s-p)-\frac{n}{q}} \left( \Vert u \Vert_{L^1_{2s}(\R^n)} + \Vert d_{\Omega}^{s-p} f \Vert_{L^q(\Omega_{1})} \right).
\end{align*}
Here, the sums $\sum_{k=1}^m 2^{-k(s-\sigma-\varepsilon)}$ and $\sum_{k=1}^m 2^{-k(s-\sigma)}$ are finite since $\varepsilon < s-\sigma$ and $\sigma < s$. This concludes the proof of \eqref{eq:tail-higher-campanato-u} in case $B_{\rho}(x_0) \cap \Omega^c \not= \emptyset$.

Next, let us assume that $B_R(x_0) \subset \Omega$.
The proof of \eqref{eq:tail-higher-campanato-u} in this case goes by the exact same arguments as in the previous case, with the only difference that we have now $\max\{2^{k-m}R;d_{\Omega}(x_0)\} = \max\{R;d_{\Omega}(x_0)\} = d_{\Omega}(x_0)$ for every $k \in \{0, 1, \dots, m\}$ and therefore also $d_{\Omega} \le 2 d_{\Omega}(x_0)$ in $B_{2^{k-m}R}(x_0)$. This concludes the proof in case $B_R(x_0) \cap \Omega^c = \emptyset$.

Finally, if $B_R(x_0) \cap \Omega^c \not=\emptyset$, we apply \eqref{eq:tail-higher-campanato-u} with $\rho := \rho$ and $R := d_{\Omega}(x_0)$, then observe that
\begin{align}
\label{eq:one-scale-up-Campanato-u}
\begin{split}
d_{\Omega}(x_0)^{n-s} \tail_{\sigma,B_1}(u - \psi(u/\psi)_{\Omega_{d_{\Omega}(x_0)}(x_0)} ; d_{\Omega}(x_0), x_0) \le c \Psi_{\sigma}(u; 2d_\Omega(x_0), x_0),
\end{split}
\end{align}
and then use \eqref{eq:Campanato-boundary-help1} with $\rho = 2 d_{\Omega}(x_0)$ and $R = R$ to deduce the desired result. Note that if $B_{4\rho}(x_0) \not\subset \Omega$ and also $B_{\rho}(x_0) \cap \Omega^c = \emptyset$, then we can just apply \eqref{eq:one-scale-up-Campanato-u} first, and then use \eqref{eq:Campanato-boundary-help1} and \eqref{eq:tail-higher-campanato-u}. Altogether, this proves \eqref{eq:tail-higher-campanato-u}, and hence the proof is complete.
\end{proof}

Now, we are in a position to prove the optimal $C^s$ regularity. The following theorem is the main result of this section and immediately implies \autoref{thm:inhom-Cs}.

\begin{theorem}
\label{thm:Cs-inhom-flat}
Assume that we are in the same setting as in \autoref{thm:inhom-Cs}. Let $\eps \in (0,\min\{\alpha s, \sigma, s-\sigma,1-s\})$. Then, there exists $R_0 \in (0,\frac{1}{16})$, depending only on $n,s,\sigma,\lambda,\Lambda,p,q,\varepsilon$, and $\Omega$, such that for any $x_0 \in \Omega_{1/8}$ and $0 < \rho \leq R_0$ it holds
\begin{align}
\label{eq-Campanato-almost-optimal-Holder-inhom}
\Psi_\sigma(u;\rho) \le c \rho^{n+\min\{\frac{\varepsilon}{2},\sigma-\varepsilon, p - \frac{n}{q}\}} \left( \|u \|_{L^1_{2s}(\R^n)} + \| d^{s-p}_{\Omega} f\|_{L^q(\Omega_1)} \right),
\end{align}
where $c = c(n, s, \lambda, \Lambda, \alpha, \sigma, p, q, \varepsilon, \Omega) > 0$. In particular, this implies
\begin{align}\label{eq-u-psi-inhom}
\left[ \frac{u}{\psi_{\cdot}} \right]_{C^{\min\{ \frac{\eps}{2},\sigma-2\varepsilon,p-\frac{n}{q}\}}(\overline{\Omega_{1/8}})} \le c \left( \|u\|_{L^1_{2s}(\R^n)} + \|d^{s-p}_{\Omega} f\|_{L^q(\Omega_1)} \right),
\end{align}
where $c=c(n, s, \lambda, \Lambda, \alpha, \sigma, p, q, \varepsilon, \Omega) > 0$.
\end{theorem}

Here, and in the rest of this subsection, we write $\psi_{\cdot}$ to denote the function $x \mapsto \psi_x(x)$, where $\psi_x$ denotes the barrier defined in \eqref{eq:psi-PDE} with respect to the operator $\widetilde{L}_x$.

\begin{proof}
Let $R_0$ be the constant given in \autoref{lemma:Morrey-boundary}, and fix $x_0 \in \Omega_{1/8}$ and $0<\rho\leq R\leq R_0$. We assume that 
\begin{align*}
\|u\|_{L^1_{2s}(\R^n)} + \|d^{s-p}_{\Omega} f\|_{L^q(\Omega_1)} \le 1.
\end{align*}
Then we have from \autoref{lem-u-Holder-inhom} and \autoref{lemma:Morrey-boundary} (applied with $\rho := R$ and $R := R_0$) that
\begin{align*}
\Psi_{\sigma}(u; \rho)
&\leq c \left( \frac{\rho}{R} \right)^{n+\varepsilon} \Psi_{\sigma}(u; R) + cR^\sigma \Phi_{\sigma}(u; R) + c R^{n + p-\frac{n}{q}} \\
&\leq c \left( \frac{\rho}{R} \right)^{n+\varepsilon} \Psi_{\sigma}(u; R) + cR^\sigma \left( \frac{R}{R_0} \right)^{n-\varepsilon} \Phi_{\sigma}(u; R_0) + cR^{n+\sigma-\varepsilon} + c R^{n + p-\frac{n}{q}}.
\end{align*}
Since $u \le c d_{\Omega}^{s-\eps}$ in $\Omega_{R_0}(x_0)$ by \autoref{thm:Cs-eps} (note that \autoref{thm:Cs-eps} is applicable here, since any $C^{1,\alpha}$ domain is locally flat Lipschitz, and we can use a covering argument), we have using also \autoref{lemma:distance-integral}:
\begin{align}
\label{eq:tail-end-est-inhom}
\begin{split}
    \Phi_{\sigma}(u; R_0)
    &\leq c \int_{\Omega_{R_0}(x_0)} d_\Omega^{-\varepsilon} \d x + \max\{R_0, d_\Omega(x_0)\}^{-s} R_0^n \tail_{\sigma, B_1}(u; R_0, x_0) \\
    &\leq cR_0^{n-\varepsilon} + R_0^{-s} \|u\|_{L^1(B_1)} \leq c.
\end{split}
\end{align}
Note that the constant $c$ here depends on $R_0$. Thus, altogether, we have
\begin{align*}
\Psi_\sigma(u;\rho) &\leq c \left( \frac{\rho}{R} \right)^{n+\varepsilon} \Psi_\sigma(u;R) + c R^{n + \sigma -\eps} + c R^{n + p - \frac{n}{q}} .
\end{align*}
It thus follows from \autoref{lem-iteration} (recalling \eqref{eq-almost-incr-Psi}) that
\begin{align}
\label{eq:Campanato-higher-help-1-inhom}
\Psi_\sigma(u; \rho) \leq c \left( \frac{\rho}{R} \right)^{n+\min\{\frac{\varepsilon}{2}, \sigma-\varepsilon,p - \frac{n}{q}\}} \Psi_\sigma(u; R) + c \rho^{n+\min\{\frac{\varepsilon}{2}, \sigma-\varepsilon,p - \frac{n}{q}\}}
\end{align}
for any $0 < \rho \leq R \leq R_0$. Note that by using the triangle inequality and \eqref{eq:tail-end-est-inhom}, we obtain $\Psi(u;R_0) \le c$ for some $c > 0$, depending on $R_0$. Hence, applying \eqref{eq:Campanato-higher-help-1-inhom} with $R = R_0$, we deduce the desired estimate \eqref{eq-Campanato-almost-optimal-Holder-inhom}.

Recall that $\psi = \psi_{x_0}$ depends on $x_0$. Thus, in order to apply Campanato's embedding, we need to replace $u/\psi_{x_0}$ by $u/\psi_{\cdot}$. To do so, we first observe that by using \eqref{eq-E-inf}
\begin{align*}
& \int_{\Omega_{1/8} \cap B_\rho(x_0)} \left| \frac{u}{\psi_x}(x) - \left( \frac{u}{\psi_{\cdot}} \right)_{\Omega_{1/8} \cap B_{\rho}(x_0)} \right| \d x \\
&\quad \le 2\int_{\Omega_{1/8} \cap B_\rho(x_0)} \left| \frac{u(x)}{\psi_{x_0}(x)}\frac{\psi_{x_0}(x)}{\psi_x(x)} - \left( \frac{u}{\psi_{x_0}} \right)_{\Omega_{\rho}(x_0)} \right| \d x \\
&\quad \le \frac{c}{2} \int_{\Omega_{\rho}(x_0)} \left| \frac{u}{\psi_{x_0}} - \left( \frac{u}{\psi_{x_0}} \right)_{\Omega_{\rho}(x_0)} \right| \left| \frac{\psi_{x_0}}{\psi_{\cdot}} + 1 \right| \d x + \frac{c}{2} \int_{\Omega_{1/8} \cap B_\rho(x_0)} \left| \frac{u}{\psi_{x_0}} + \left( \frac{u}{\psi_{x_0}} \right)_{\Omega_{\rho}(x_0)} \right| \left| \frac{\psi_{x_0}}{\psi_{\cdot}} - 1 \right| \d x  \\
&\quad =: I_1 + I_2.
\end{align*}
Here, we also used the algebraic identity $(ab-cd) = \frac{1}{2}(a-c)(b+d) + \frac{1}{2}(a+c)(b-d)$. To estimate $I_1$, we use $\psi_{x_0}(x) \asymp \psi_x(x) \asymp d_{\Omega}^s(x)$ in $\Omega_{1/4}$ due to \autoref{prop:inhom-Cs}(ii), as well as \eqref{eq-Campanato-almost-optimal-Holder-inhom}:
\begin{align*}
I_1 \le \frac{c}{2} \left \Vert \frac{\psi_{x_0}}{\psi_{\cdot}}+1 \right\Vert_{L^{\infty}(\Omega_{\rho}(x_0))} \int_{\Omega_{\rho}(x_0)} \left|\frac{u}{\psi_{x_0}} - \left( \frac{u}{\psi_{x_0}} \right)_{\Omega_{\rho}(x_0)} \right| \d x \le c \rho^{n+\min\{\frac{\eps}{2} , \sigma-\varepsilon, p - \frac{n}{q} \}}.
\end{align*}
For $I_2$, we recall that by \autoref{thm:Cs-eps} and \autoref{prop:inhom-Cs}(ii) it holds $u/\psi_{x_0} \le c u/d^s_{\Omega} \le d_{\Omega}^{-\eps}$ in $\Omega_{1/4}$ and hence we obtain by using \autoref{lemma:distance-integral} and \autoref{lemma:w-regularity}
\begin{align*}
    I_2
    &\leq c \left( \int_{\Omega_\rho(x_0)} d_\Omega^{-s-\varepsilon} \d x + \rho^{-n} \int_{\Omega_\rho(x_0)} d_\Omega^{-s} \d x\int_{\Omega_\rho(x_0)} d_\Omega^{-\varepsilon} \d x\right) \sup_{x \in \Omega_\rho(x_0)} |\psi_{x_0}(x)-\psi_x(x)| \\
    &\leq c \rho^{n} \max\{\rho, d_{\Omega}(x_0)\}^{-s-\eps} \sup_{y_0 \in \Omega_\rho(x_0)} \sup_{x \in \Omega_\rho(x_0)} |\psi_{x_0}(x)-\psi_{y_0}(x)| \\
    &\leq c \rho^{n} \max\{\rho, d_{\Omega}(x_0)\}^{-s-\eps} \max\{ \rho , d_{\Omega}(x_0) \}^{s-\varepsilon} \sup_{x_0,y_0 \in \Omega_{\rho}(x_0)} [\psi_{x_0} - \psi_{y_0}]_{C^{s-\eps}(\overline{\Omega_{1/2}})} \\
    &\leq c\rho^{n+\sigma} \max\{\rho , d_{\Omega}(x_0)\}^{-2\eps}.
\end{align*}

Thus, altogether, we have shown
\begin{align*}
\int_{\Omega_{1/8} \cap B_\rho(x_0)} \left| \frac{u}{\psi_x}(x) - \left( \frac{u}{\psi_{\cdot}} \right)_{\Omega_{1/8} \cap B_{\rho}(x_0)} \right| \d x \le c \rho^{n+\min\{\frac{\eps}{2}, \sigma-\varepsilon, p - \frac{n}{q} \}}+ c \rho^{n + \sigma - 2\eps}.
\end{align*}
Hence, by application of Campanato's embedding, we obtain \eqref{eq-u-psi-inhom}.
\end{proof}

Note that \autoref{thm:Cs-inhom-flat} in particular implies \autoref{thm:inhom-Cs}.

\begin{proof}[Proof of \autoref{thm:inhom-Cs}]
We assume again as in the proof of \autoref{thm:Cs-inhom-flat} that $u$ and $f$ are normalized. By \autoref{thm:Cs-inhom-flat} we have for some $\delta > 0$:
\begin{align*}
\left[ \frac{u}{\psi_{\cdot}} \right]_{C^{\delta}(\overline{\Omega_{1/8}})} \le c.
\end{align*}
In particular, given any $x_0 \in \Omega_{1/8}$, we fix a point $x \in \Omega_{1/8}$ with $d_{\Omega}^s(x) \ge c$ for some fixed $c > 0$ and use \autoref{prop:inhom-Cs}(ii) and the local boundedness (\autoref{lemma:locbd-bdry}), then we have
\begin{align*}
\left| \frac{u(x_0)}{\psi_{x_0}(x_0)} \right| \le \left| \frac{u(x_0)}{\psi_{x_0}(x_0)} - \frac{u(x)}{\psi_x(x)} \right| + \left|\frac{u(x)}{\psi_x(x)}\right| \le c |x-x_0|^{\delta} + c \frac{|u(x)|}{d_{\Omega}^{s}(x)} \le c.
\end{align*}
Thus, using that $\psi_{x_0}(x_0) \le c d^s_{\Omega}(x_0)$ by \autoref{prop:inhom-Cs}(ii), we deduce that $u(x_0) \le c d^s_{\Omega}(x_0)$. Since $x_0 \in \Omega_{1/8}$ was arbitrary, we have shown that $u \le c d^s_{\Omega}$ in $\Omega_{1/8}$. Hence, the desired $C^s(\overline{\Omega_{1/8}})$ regularity follows by combination with the interior regularity from \autoref{prop:interior-regularity} in the same way as in Step 2 of the proof of \cite[Proposition 2.6.4]{FeRo24}. The proof is complete.
\end{proof}

\section{Hopf lemma}
\label{sec:Hopf}

The goal of this section is to prove the following nonlocal Hopf lemma, which directly implies the second part of \autoref{thm:main-2}.

\begin{theorem}\label{thm-hopf}
Let $\alpha, \sigma \in (0,s)$. Let $\Omega \subset \R^n$ be a $C^{1,\alpha}$ domain with $0 \in \partial \Omega$. Assume that $K$ satisfies \eqref{eq:Kcomp}, and \eqref{eq-K-cont} with $\mathcal{A} = B_1$. Let $u$ be a solution to
\begin{align*}
\left\{
\begin{aligned}
Lu& = f \ge 0 &&\text{in }\Omega_1, \\
u&= 0 &&\text{in }B_1 \setminus \Omega.
\end{aligned}
\right.
\end{align*}
with $f \in L^{\infty}(\Omega_1)$ and $u \ge 0$ in $\R^n$.
Then either $u \equiv 0$ in $\Omega$ or
\begin{equation*}
u \geq c d_\Omega^s \quad \text{in } \Omega_{1/2},
\end{equation*}
where $c=c(n, s, \lambda, \Lambda, \alpha, \sigma, \Omega, u, f)>0$.
\end{theorem}

\subsection{Closeness of two solutions}

The goal of this subsection is to prove the following proposition, which establishes closeness (in a quantified way of order $C^s$) of two solution $\phi$ and $\phi_0$ with respect to $L$ and it frozen operator at zero in $\Omega_{\eps}$. It is a main ingredient in the proof of \autoref{thm-hopf}.

\begin{proposition}
\label{prop:closeness-phi}
Let $\alpha, \sigma \in (0,s)$, $\delta \in (0, \sigma)$, and $\varepsilon \in (0, 1)$. Let $\Omega \subset \R^n$ be a $C^{1,\alpha}$ domain with $0 \in \partial \Omega$.
Assume that $K$ satisfies \eqref{eq:Kcomp}, and \eqref{eq-K-cont} with $\mathcal{A} = B_1$. Let $\phi$ be a solution to  
\begin{equation*}
\left\{
\begin{aligned}
L\phi&=0 &&\text{in }\Omega_{\eps}, \\
\phi&=0 &&\text{in } B_{\eps} \setminus \Omega,
\end{aligned}
\right.
\end{equation*}
with $0 \le \phi \le 1$, and $\phi_0$ be a solution to
\begin{equation*}
\left\{
\begin{aligned}
L_0 \phi_0 & = 0 &&\text{in }\Omega_{\eps}, \\
\phi_0 & = \phi &&\text{in } \R^n \setminus \Omega_{\eps},
\end{aligned}
\right.
\end{equation*}
where $L_0$ denotes the frozen operator with respect to $L$ at $0$. Then
    \begin{align*}
    \dashint_{\Omega_{\rho}} |\phi - \phi_0| \d x \le C \eps^{-s + \sigma-\delta}\rho^s \quad\text{for any }0<\rho\leq \varepsilon,
    \end{align*}
where $c>0$ depends only on $n,s,\lambda,\Lambda,\alpha,\sigma,\delta$, and $\Omega$.
\end{proposition}

First, we need to show the following lemma:

\begin{lemma}
\label{lemma:closeness-phi-help}
Assume that we are in the setting of \autoref{prop:closeness-phi} and set $u = \phi - \phi_0$. Let $0 < R \le \varepsilon$.  Let $v$ be a solution to
\begin{equation}
\label{eq-v-Sect7}
\left\{
\begin{aligned}
L_{0} v & = 0 &&\text{in }\Omega_R, \\
v & = u &&\text{in } \R^n \setminus \Omega_R,
\end{aligned}
\right.
\end{equation}
and set $w = u-v$. Then
\begin{align*}
    [w]_{H^s(\R^n)}\le c \eps^{-s} R^{\frac{n}{2} + \sigma - \delta},
\end{align*}
where $c > 0$ depends only on $n,s,\lambda,\Lambda,\alpha,\sigma,\delta$, and $\Omega$.
\end{lemma}

\begin{proof}
Note that by assumption it holds
\begin{align*}
    \Vert \phi \Vert_{L^2(B_1)}^2 + \tail(\phi;1,0)  \le c
\end{align*}
for some constant $c > 0$, depending only on $n$ and $s$, and that $w$ solves
\begin{equation}
\label{eq-w-Sect7}
\left\{
\begin{aligned}
L_{0}w & = L_0 (\phi - \phi_0) = (L_{0}-L)\phi &&\text{in }\Omega_R, \\
w & = 0 &&\text{in } \R^n \setminus \Omega_R.
\end{aligned}
\right.
\end{equation}
By testing the equation with $w$, we get
\begin{align*}
    \lambda [w]_{H^s(\R^n)}^2 \le \cE^{K_{0}}(w,w) = \cE^{K_{0}-K}(\phi, w).
\end{align*}
We use \eqref{eq-K-K0} to estimate
    \begin{align*}
        \cE^{L_{0}-L}(\phi,w) 
        &\le cR^{\sigma} \left(\int_{\Omega_{R}} \int_{B_{2R}}  \frac{|\phi(x) - \phi(y)|^2 }{|x-y|^{n+2s}} \d y \d x \right)^{1/2} [w]_{H^s(\R^n)} \\
        &\quad + c \int_{\Omega_R} \int_{B_1 \setminus B_{2R}} |\phi(x)-\phi(y)| |w(x)| \frac{|x|^\sigma + |y|^\sigma}{|x-y|^{n+2s}} \d y \d x \\
        &\quad + c \int_{\Omega_{R}} \int_{\R^n \setminus B_1} \frac{|\phi(x)-\phi(y)| |w(x)|}{|x-y|^{n+2s}} \d y \d x \\
        &=: J_1 + J_2 + J_3.
    \end{align*}

    For $J_1$, we follow the proof of \autoref{lemma:psi-interior-reg} and use that $\phi \in C^s(\overline{\Omega_{1/2}})$ by \autoref{thm:inhom-Cs} and scaling, and obtain
    \begin{align*}
        \int_{B_{2R}}  \frac{|\phi(x) - \phi(y)|^2 }{|x-y|^{n+2s}} \d y \le c \eps^{-s} (1 + d_{\Omega}^{-\delta}(x)) ~~ \forall x \in B_{\eps},
    \end{align*}
    and hence by \autoref{lemma:distance-integral}, we have
    \begin{align*}
        J_1 \le c R^{\sigma} \left(\int_{\Omega_{R}} (1 + d_{\Omega}^{-\delta}) \d x \right)^{1/2} [w]_{H^s(\R^n)} \le c \eps^{-s} R^{\frac{n}{2} + \sigma -\frac{\delta}{2}} [w]_{H^s(\R^n)} \leq c\eps^{-s} R^{\frac{n}{2} + \sigma -\delta} [w]_{H^s(\R^n)}.
    \end{align*}

    For $J_2$, note that $|x| <|y| \leq 2|x-y|$. We then use that by \autoref{thm:inhom-Cs} it holds $\phi(x) \le c d_{\Omega}^s(x)\le c R^s$, and that $\phi(y) \leq c \eps^{-s} d^s_{\Omega}(y) \le c \eps^{-s}|y|^s$, to get
    \begin{align*}
        J_2 \le c \eps^{-s} \int_{\Omega_R} \int_{B_1 \setminus B_{2R}} |w(x)| \frac{|y|^s}{|y|^{n+2s-\sigma}} \d y \d x \le c \eps^{-s} R^{\frac{n}{2} + \sigma} [w]_{H^s(\R^n)},
    \end{align*}
    where we also used \eqref{eq:w-PF} in the last inequality.

Finally, for $J_3$, using that $0 \le \phi \le 1$, as well as \eqref{eq:w-PF}:
\begin{align*}
J_3 \le c \left( \int_{\R^n \setminus B_{1/2}} \frac{1}{|y|^{n+2s}} \d y \right) \left( \int_{\Omega_{R}} |w(x)| \d x \right) \leq cR^{\frac{n}{2}+s}[w]_{H^s(\R^n)}.
\end{align*}
Combining all the estimates and using $\sigma \in (0,s)$ finish the proof.
\end{proof}

\begin{lemma}
\label{lemma:closeness-phi-help-2}
Assume that we are in the setting of \autoref{prop:closeness-phi} and set $u = \phi - \phi_0$. Then it holds
\begin{align*}
    [u]_{H^s(\R^n)}\le c \eps^{\frac{n}{2} + \sigma - \delta -s},
\end{align*}
where $c > 0$ depends only on $n,s,\sigma,\lambda,\Lambda,\alpha,\sigma,\delta$, and $\Omega$.
\end{lemma}

\begin{proof}
Note that
\begin{align*}
\left\{
\begin{aligned}
    L_0 u &= L_0 \phi = (L_0 - L) \phi ~~ &&\text{ in } \Omega_{\eps},\\
    u &= 0 ~~ &&\text{ in } \R^n \setminus \Omega_\varepsilon.
\end{aligned}
\right.
\end{align*}
Hence, we can apply exactly the same arguments as in the proof of \autoref{lemma:closeness-phi-help} with $R := \eps$, $w := u$.
\end{proof}

We are now in a position to give the proof of \autoref{prop:closeness-phi}.

\begin{proof}[Proof of \autoref{prop:closeness-phi}]
As in \autoref{lemma:closeness-phi-help}, we denote $u = \phi - \phi_0$. Let $0 < \rho < R \le \varepsilon$. Let $v$ and $w$ be the solutions of \eqref{eq-v-Sect7} and \eqref{eq-w-Sect7}, respectively. Then, $v$ satisfies the assumptions of \autoref{lemma:T1}. We follow the proof of \eqref{eq:Morrey-boundary-help1}, but use H\"older's inequality and the Hardy inequality (\autoref{lemma:Hardy}) for $w$ and \autoref{lemma:closeness-phi-help} instead of \autoref{lemma:freezing-ds}, to establish the following Morrey-type estimate for $u$:
\begin{align*}
    \int_{\Omega_{\rho}} \left| \frac{u}{d_{\Omega}^s} \right| \d x
    &\le c \left( \frac{\rho}{R} \right)^{n-\delta/2} \Phi_{\sigma}(u; R,0) + cR^{n+s} \|u\|_{L^1_{2s}(\R^n)} + c\int_{\Omega_R} \left| \frac{w}{d_{\Omega}^s} \right| \d x \\
    &\le c \left( \frac{\rho}{R} \right)^{n-\delta/2} \Phi_{\sigma}(u; R,0) + c\eps^{-s} R^{n + \sigma - \delta},
\end{align*}
where we used that $\|u\|_{L^1_{2s}(\R^n)} \leq c$ and $R^{n+s} \leq R^{n+\sigma-\delta}$. Moreover, we can prove an estimate for $\rho^{n-s} \tail_{\sigma,B_1}(u;\rho,0)$ by following the proof of \eqref{eq:tail-Morrey-u}
\begin{align*}
    \rho^{n-s} \tail_{\sigma,B_1}(u;\rho,0) &\le c \left( \frac{\rho}{R} \right)^{n-\delta/2}  \Phi_{\sigma}(u; R, 0) + c \eps^{-s} R^{n + \sigma - \delta}.
    \end{align*}
Altogether, we obtain
\begin{align*}
    \Phi_{\sigma}(u; \rho,0) \le c \left( \frac{\rho}{R} \right)^{n-\delta/2} \Phi_{\sigma}(u; R, 0) + c \varepsilon^{\sigma-s} R^{n - \delta}.
\end{align*}
It follows from \autoref{lem-iteration} (recalling \eqref{eq-almost-incr-Phi}) that
\begin{align}\label{eq-Hopf-Morrey}
    \Phi_{\sigma}(u; \rho, 0) \le c \left( \frac{\rho}{R} \right)^{n-\delta} \Phi_{\sigma}(u; R, 0) + c \eps^{ \sigma-s} \rho^{n - \delta} \quad\text{for any }0<\rho<R\leq \varepsilon.
\end{align}
In particular, by setting $R := \eps$, it follows from \autoref{lemma:Hardy} and \autoref{lemma:closeness-phi-help-2} that
\begin{align}
\label{eq:Hopf-Morrey}
   \dashint_{\Omega_{\rho}} \left| \frac{u}{d_{\Omega}^s} \right| \d x \leq c\rho^{-\delta} \varepsilon^{-\frac{n}{2}+\delta} \left( \int_{\Omega_\varepsilon} \left| \frac{u}{d_\Omega^s} \right|^2 \d x \right)^{1/2} + \varepsilon^{\sigma-s}\rho^{-\delta} \le c \eps^{\sigma-s}\rho^{-\delta},
\end{align}
which proves the Morrey estimate.

The next step is to prove a Campanato-type estimate following the proof of \autoref{lem-u-Holder-inhom} and using again \autoref{lemma:closeness-phi-help} instead of \autoref{lemma:freezing-ds}. We obtain for any $\gamma \in (0,\alpha s)$
\begin{align*}
    \int_{\Omega_{\rho}} \left|\frac{u}{\psi} - \left( \frac{u}{\psi} \right)_{\Omega_{\rho}} \right| \d x 
    &\le c \left( \frac{\rho}{R} \right)^{n+ \gamma} \Psi_{\sigma}(u; R, 0) + cR^{\sigma}\Phi_{\sigma}(u; R, 0) + c \eps^{-s} R^{n + s} + c \int_{\Omega_R} \left| \frac{w}{d_{\Omega}^s} \right| \d x \\
    &\le  c \left( \frac{\rho}{R} \right)^{n+ \gamma} \Psi_{\sigma}(u;R,0) + cR^{\sigma}\Phi_{\sigma}(u; R, 0) + c \eps^{-s} R^{n + \sigma - \delta},
\end{align*}
where $\psi = \psi_0$ denotes the barrier function from \autoref{prop:inhom-Cs} associated to $L_0$. We can also prove an estimate for $\rho^{n-s} \mathrm{Tail}_{\sigma,B_1}(u - \psi (u/\psi)_{\Omega_{\rho}}; \rho, 0)$ by following the arguments from the proof of \autoref{lem-u-Holder-inhom}
\begin{align*}
   \rho^{n-s} \mathrm{Tail}_{\sigma,B_1}(u - \psi (u/\psi)_{\Omega_{\rho}}; \rho, 0) &\le c \left( \frac{\rho}{R} \right)^{n+\gamma} \Psi_{\sigma}(u;R,0) + cR^{\sigma}\Phi_{\sigma}(u; R, 0) + c \eps^{-s} R^{n + \sigma - \delta}.
    \end{align*}
Altogether, we obtain
\begin{align*}
    \Psi_{\sigma}(u;\rho,0) \le c \left( \frac{\rho}{R} \right)^{n+\gamma} \Psi_{\sigma}(u;R,0) + cR^{\sigma}\Phi_{\sigma}(u; R, 0) + c \eps^{-s} R^{n + \sigma - \delta}.
\end{align*}
Next, we apply \eqref{eq-Hopf-Morrey} with $\rho:=R$, $R:=\varepsilon$, and then use  \autoref{thm:inhom-Cs} to estimate
\begin{align*}
    \Phi_{\sigma}(u; R, 0) \leq c \left( \frac{R}{\varepsilon} \right)^{n-\delta} \int_{\Omega_\varepsilon} \left| \frac{u}{d_{\Omega}^s} \right| \d x + c \eps^{\sigma-s} R^{n - \delta} \leq c(\varepsilon^{\delta}+\varepsilon^{\sigma}) \eps^{-s} R^{n-\delta} \leq c \eps^{-s} R^{n-\delta},
\end{align*}
which yields
\begin{align*}
    \Psi_{\sigma}(u;\rho,0) \le c \left( \frac{\rho}{R} \right)^{n+\gamma} \Psi_{\sigma}(u;R,0) + c \eps^{-s} R^{n + \sigma - \delta}.
\end{align*}

By choosing $\frac{\gamma}{2} \le \sigma - \delta$ and applying \autoref{lem-iteration} (recall \eqref{eq-almost-incr-Psi}), we get that
\begin{align*}
    \Psi_{\sigma}(u; \rho, 0) \le c \left( \frac{\rho}{R} \right)^{n+ \frac{\gamma}{2}} \Psi_{\sigma}(u;R,0) + c \eps^{\sigma - \delta - \frac{\gamma}{2} - s} \rho^{n + \frac{\gamma}{2}} \quad\text{for any }0<\rho<R\leq \varepsilon.
\end{align*}
In particular, for $R := \eps$, we obtain the following Campanato estimate from \autoref{lemma:closeness-phi-help-2}:
\begin{align*}
   \dashint_{\Omega_{\rho}} \left|\frac{u}{\psi} - \left( \frac{u}{\psi} \right)_{\Omega_{\rho}} \right| \d x \le c \eps^{\sigma - \delta - \frac{\gamma}{2} - s} \rho^{\frac{\gamma}{2}}.
\end{align*}

By the same arguments as in the proof of \cite[(5.9)]{GiMa12}, we get that $q := \lim_{\rho \to 0} \left(\frac{u}{\psi} \right)_{\Omega_{\rho}}$ exists, and
\begin{align}
\label{eq:Hopf-Campanato-2}
    \left|q - \left( \frac{u}{\psi} \right)_{\Omega_{\rho}}  \right| \le c \eps^{\sigma - \delta - \frac{\gamma}{2} - s }\rho^{\frac{\gamma}{2}} \quad\text{for any }0< \rho \le \eps.
\end{align}
Hence, by \eqref{eq:Hopf-Morrey} and \eqref{eq:Hopf-Campanato-2} it holds
\begin{align*}
    |q| \le \left|q - \left( \frac{u}{\psi} \right)_{\Omega_\rho}  \right| + \left( \frac{u}{\psi} \right)_{\Omega_\rho} \le c \eps^{\sigma - \delta - s}.
\end{align*}
Using \eqref{eq:Hopf-Campanato-2} again, this in turn implies that for any $\rho \leq \varepsilon$
\begin{align*}
    \left( \frac{u}{\psi} \right)_{\Omega_{\rho}} \le  \left|q - \left( \frac{u}{\psi} \right)_{\Omega_{\rho}}  \right| + |q| \le  c \eps^{\sigma - \delta - \frac{\gamma}{2} - s} \rho^{\frac{\gamma}{2}} + c\eps^{\sigma - \delta} \le c \eps^{\sigma - \delta - s},
\end{align*}
which yields
\begin{align*}
    \dashint_{\Omega_{\rho}} |u| \d x \le c \rho^s [\psi]_{C^s(\overline{\Omega_{1}})} \left( \frac{u}{\psi} \right)_{\Omega_{\rho}} \le c \rho^s \eps^{-s+ \sigma - \delta} \Vert \psi \Vert_{L^{\infty}(\R^n)} \le c \rho^s \eps^{-s+ \sigma - \delta},
\end{align*}
as desired.
\end{proof}

\subsection{Proof of the Hopf lemma}

Before we prove the Hopf lemma (\autoref{thm-hopf}), we need to establish the following refined version of the Hopf lemma for translation invariant operators, determining the dependence of the constant on the domain.

\begin{lemma}\label{lem-Hopf}
Let $\Omega \subset \R^n$ be a $C^{1,\alpha}$ domain for some $\alpha \in (0, s)$ and assume that $0 \in \partial \Omega$ with outer unit normal vector $-e_n$. Let us set $B^{\eps} := B_{\eps}(4\eps e_n) \subset \Omega$ for $\eps \in (0,1)$. Let $L$ be translation invariant operator with kernel $K$ satisfying \eqref{eq:Kcomp}. Let $\bar{\phi}_{\eps}$ be the solutions to
    \begin{align*}
    \left\{
    \begin{aligned}
    L \bar{\phi}_{\eps}&=0 &&\text{in }\Omega_{\eps}, \\
    \bar{\phi}_{\eps}&=1 &&\text{in }B^{\eps}, \\
    \bar{\phi}_{\eps}&=0 &&\text{in }\R^n \setminus (\Omega_{\eps} \cup B^{\eps}).
    \end{aligned}
    \right.
    \end{align*}
    Then, it holds
    \begin{align}
    \label{eq:ti-Hopf}
        \bar{\phi}_{\eps}(m e_n)\ge c \eps^{-s} m^s ~~ \forall m \in (0,\eps/2]
    \end{align}
    for some $c > 0$ that only depends on $n,s,\lambda,\Lambda,\alpha$, the diameter and $C^{1,\alpha}$ radius of $\Omega$, but not on $\eps$.
\end{lemma}

Note that the Hopf lemma from \cite[Theorem 6.10]{RoWe24} implies \eqref{eq:ti-Hopf} with a constant $c > 0$ that might depend on $\eps$. Hence, \autoref{lem-Hopf} states that the constant $c > 0$ stays bounded as $\eps \searrow 0$. 

\begin{proof}
To prove the result, let us first consider $\phi_{\eps}(x) = \bar{\phi}_{\eps}(\eps x)$, which solves
\begin{align*}
    \left\{
    \begin{aligned}
            L_{\eps} \phi_{\eps} &= 0 &&\text{in } \eps^{-1}\Omega \cap B_1,\\
            \phi_{\eps} &= 1 &&\text{in } B_{1}(4e_n),\\
            \phi_{\eps} & = 0 &&\text{in } \R^n \setminus ([\eps^{-1} \Omega \cap B_1] \cup B_{1}(4e_n)),
    \end{aligned}
    \right.
\end{align*}
where $L_{\eps}$ is a translation invariant operator with kernel $K_{\eps}(h) = \eps^{n+2s} K(\eps h)$, which still satisfies \eqref{eq:Kcomp} with $\lambda,\Lambda$, and we used that $B_{1}(4 e_n) = \eps^{-1} B^{\eps}$. We will prove that
\begin{align}
\label{eq:Hopf-ti-help-1}
  q_{\eps} := \liminf_{m \to 0} \frac{\phi_{\eps}(m e_n)}{b_{\eps}(m)} \ge c_0
\end{align}
for some $c_0 > 0$, that does not depend on $\eps$, where $b_{\eps}$ denotes the 1D solution from \cite[Theorem 6.9]{RoWe24} with respect to $L_{\eps}$ and the normal vector $e_n$. Note that once \eqref{eq:Hopf-ti-help-1} is established, since $c_1 m^s \le b_{\eps}(m) \le c_2 m^s$ for any $m > 0$, where $c_1,c_2$ are independent of $\eps$, and moreover, by the expansion from \cite[Theorem 6.9]{RoWe24} it follows that
\begin{align*}
    \phi_{\eps}(m e_n) \ge q_{\eps} b_{\eps}(m) - |\phi_{\eps}(m e_n) - q_{\eps} b_{\eps}(m)| \ge c_0 c_1 m^s - c_3 m^{s+\gamma}
\end{align*}
for some $c_3, \gamma > 0$. Hence, for $m \le m_0$, with $m_0 > 0$ depending only on $c_0,c_1,c_3, \gamma$, which in turn only depend on $n,s,\lambda,\Lambda$, and the $C^{1,\alpha}$ radius and diameter of $\Omega$, but not on $\eps$, we have
\begin{align*}
    \phi_{\eps}(m e_n) \ge \frac{c_0 c_1}{2} m^s.
\end{align*}
Hence, by application of the interior Harnack inequality we obtain
\begin{align*}
    \phi_{\eps}(m e_n) \ge c m^s ~~ \forall m \in (0,1/2],
\end{align*} 
where $c = \min\{\frac{c_0 c_1}{2} , c_4 \inf_{ m \in (m_0,\frac{1}{2}]}   \phi_{\eps}(m e_n) \}$ for some $c_4 > 0$, depending only on $n,s,\lambda,\Lambda$. Note that by a Harnack chain argument, we can compare $\inf_{ m \in (m_0,\frac{1}{2}]}   \phi_{\eps}(m e_n)$ with $\inf_{B_1(4 e_n)} \phi_{\eps} = 1$, where the comparability constant only depends on $n,s,\lambda,\Lambda,m_0$, and the $C^{1,\alpha}$ radius of $\Omega$.
Hence, we obtain for $\bar{\phi}_{\eps}$ and $m \in (0,\frac{\eps}{2}]$:
\begin{align*}
    \bar{\phi}_{\eps}(m e_n) = \phi_{\eps}(\eps^{-1} m e_n) \ge C \eps^{-s} m^s.
\end{align*}

It remains to show \eqref{eq:Hopf-ti-help-1}. We will prove it by contradiction. Indeed, assume that $\inf_{\eps \in (0,1)} q_{\eps} = 0$. Then, observe that the sequence $\phi_{\eps}$ is uniformly bounded in $C^s_{\mathrm{loc}}(\Omega_{\eps})$, and therefore by the Arzela--Ascoli theorem, it holds $\phi_{\eps} \to \phi_0$ locally uniformly. Thus, by the stability for translation invariant nonlocal equations (see \cite[Proposition 2.2.36]{FeRo24}), it holds
\begin{align*}
\left\{
\begin{aligned}
            L_{0} \phi_{0} &= 0 &&\text{in } \{ x_n > 0 \} \cap B_1,\\
            \phi_{0} &= 1 &&\text{in } B_{1}(4e_n),\\
            \phi_{0} & = 0 &&\text{in } \R^n \setminus ([\{ x_n > 0 \} \cap B_1] \cup B_{1}(4e_n)),
\end{aligned}
\right.
\end{align*}
where $L_{\eps} \to L_0$ in the weak sense (defined in \cite[Proof of Proposition 2.2.36]{FeRo24}). Moreover, if we let $\eta \in C_c^{\infty}(B_2)$ be such that $0 \le \eta \le 1$ and $\eta \equiv 1$ in $B_1$ and set $\tilde{b}_{\eps} = b_{\eps} \eta$, then we have 
\begin{align*}
\left\{
\begin{aligned}
        L_{\eps} \tilde{b}_{\eps} &= f_{\eps} &&\text{in } \{ x_n > 0 \} \cap B_1,\\
        \tilde{b}_{\eps} &= \eta b_{\eps} &&\text{in } \{ x_n > 0 \} \setminus B_1,\\
        \tilde{b}_{\eps} &= 0 &&\text{in } \{ x_n \le 0 \},
\end{aligned}
\right.
\end{align*}
where $f_\varepsilon:=-L_\varepsilon((1-\eta)b_\varepsilon)$. For $x \in \{ x_n > 0 \} \cap B_1$ it holds that
\begin{align*}
    0 \le f_{\eps}(x) = \int_{\R^n \setminus B_1} (1-\eta)(y) b_{\eps}(y) K_{\eps}(x-y) \d y \le c
\end{align*}
for some $c > 0$, which is independent of $\eps$. Since $(b_{\eps})$ is uniformly bounded in $C^s_{\loc}(\R^n)$ (see \cite[Lemma 5.1, Theorem 1.4]{RoWe24}), by the stability for translation invariant nonlocal equations (see \cite[Proposition 2.2.36]{FeRo24}), and the weak convergence $L_{\eps} \to L_0$, we have
\begin{align*}
    b_{\eps} \to b_0, \qquad \tilde{b}_{\eps} \to \tilde{b}_0 = \eta b_0, \qquad f_{\eps} \to f_0 = \int_{\R^n \setminus B_1} (1-\eta)(y) b_{0}(y) K_{0}(\cdot -y) \d y = -L_0((1-\eta) b_0)
\end{align*}
locally uniformly, and
\begin{align*}
\left\{
\begin{aligned}
        L_{0} \tilde{b}_{0} &= f_{0} &&\text{in } \{ x_n > 0 \} \cap B_1,\\
        \tilde{b}_{\eps} &= \eta b_{0} &&\text{in } \{ x_n > 0 \} \setminus B_1,\\
        \tilde{b}_{0} &= 0 &&\text{in } \{ x_n \le 0 \}.
\end{aligned}
\right.
\end{align*}
In particular, it follows that $b_0$ is the 1D barrier in $\{ x_n > 0 \}$ with respect to $L_0$ from \cite{RoWe24} (see also \autoref{prop:inhom-Cs}), which is unique up to a positive constant.

By \cite[Theorem 6.9]{RoWe24}, it holds for some $q_0 \in \R$ and $C,\gamma > 0$, which do not depend on $\eps$ (the proof of \cite[Theorem 6.9]{RoWe24} only yields a dependence on the diameter and $C^{1,\alpha}$ radius of $\Omega$):
\begin{align*}
    |\phi_0(x) - q_{0} b_0 (x_n)| &\le C |x|^{s+\gamma} ~~ \forall x \in \{ x_n > 0 \} \cap B_{1/2}, \\
        |\phi_{\eps}(x) - q_{\eps} b_{\eps} (x_n)| &\le C |x|^{s+\gamma} ~~ \forall x \in \eps^{-1} \Omega \cap B_{1/2}.
\end{align*}
Here $q_{\eps}$ is the same quantity as in \eqref{eq:Hopf-ti-help-1}. Hence, we have for $x = x_n e_n \in \Omega_{1/2}$:
\begin{align*}
    |q_0 - q_{\eps}|
    &\le c |x|^{-s} |q_0-q_\varepsilon| |b_0(x_n)| \\
    &\le c |x|^{-s} |q_0 b_0(x_n) - \phi_0(x)| + c |x|^{-s} |\phi_{\eps}(x) - q_\varepsilon b_{\eps}(x_n)| \\
    &\quad + c |x|^{-s}  |\phi_0(x) - \phi_{\eps}(x)|+ c |x|^{-s} |q_\varepsilon b_{\eps}(x_n) - q_\varepsilon b_{0}(x_n)| \\
    &=: I_1 + I_2 + I_3 + I_4.
\end{align*}
By the expansions from the previous display, we clearly have
\begin{align*}
    I_1 + I_2 \le c |x|^{\gamma}.
\end{align*}
Moreover, by the $C^s$ convergence $\phi_{\eps} \to \phi_0$ and $b_{\eps} \to b_0$, we also have that for any fixed $|x| = \delta$, there exists $\eps$ small enough, such that $I_3 + I_4 \le \frac{\delta}{4}$. All in all, we deduce that for any $\delta > 0$, when $\eps > 0$ is small enough, then
\begin{align*}
    |q_0 - q_{\eps}| \le \frac{\delta}{2}.
\end{align*}
Finally, recalling that we assumed $\inf_{\eps \in (0,1)} q_{\eps} = 0$, we deduce that for $\eps > 0$ small enough, it must be $|q_{\eps}| \le \frac{\delta}{2}$, and hence
\begin{align*}
    |q_0| \le |q_0 - q_{\eps}| + |q_{\eps}| \le \frac{\delta}{2} + \frac{\delta}{2} = \delta.
\end{align*}
Since $\delta > 0$ was arbitrary, we deduce that $|q_0| = 0$, which contradicts Step 2b in the proof of the Hopf lemma in \cite[Theorem 6.10]{RoWe24}. The proof is complete.
\end{proof}

We are now in a position to prove \autoref{thm-hopf}.

\begin{proof}[Proof of \autoref{thm-hopf}]
Let $\varepsilon \in (0,1)$ be a constant to be determined later.
First, we will prove that for any $z \in \Omega_{1/2} \cap \{ d_{\Omega} \le \frac{\eps}{4} \}$ it holds
\begin{align}
\label{eq:Hopf-proof-boundary}
    u(z) \ge c d_{\Omega}^s(z).
\end{align}
To do so, let us fix any point $z \in \Omega_{1/2}$ with $d_\Omega(z) \leq \frac{\varepsilon}{4}$. Then, there exists $x_0 \in \partial \Omega$ such that $d_\Omega(z)=|z-x_0| \leq \frac{\varepsilon}{4}$. After a rotation,  translation, and scaling, we may assume that $x_0=0$, that the outer unit normal vector at $x_0 = 0$ is $-e_n$, that $d_\Omega(z)=|z|\leq \frac{\varepsilon}{4}$, and that $B^{\eps} :=B_{\eps}(4\eps e_n) \subset \Omega$. 

We define a barrier $\phi$ as the solution to
\begin{equation*}
\left\{
\begin{aligned}
L\phi_{\eps}&=0 &&\text{in }\Omega_{\eps}, \\
\phi_{\eps}&=1 &&\text{in }B^{\eps}, \\
\phi_{\eps}&=0 &&\text{in }\R^n \setminus (\Omega_{\eps} \cup B^{\eps}).
\end{aligned}
\right.
\end{equation*}
Then, since for $c_{\eps} :=\inf_{B^{\eps}} u$ we have
\begin{align*}
\left\{
\begin{aligned}
L(u-c_{\eps}\phi_{\eps}) &\geq 0 &&\text{in }\Omega_{\eps}, \\
u-c_{\eps} \phi_{\eps} &\geq 0 &&\text{in } \R^n \setminus \Omega_{\eps},
\end{aligned}
\right.
\end{align*}
the comparison principle shows that
\begin{equation*}
u \geq c_{\eps} \phi_{\eps} \quad\text{in }\R^n.
\end{equation*}

Next, consider the frozen operator $L_0$ with respect to $L$ at 0 and define $\bar{\phi}_{\eps}$ as the solution to
\begin{align*}
\left\{
\begin{aligned}
L_0 \bar{\phi}_{\eps} &= 0 &&\text{in }\Omega_{\eps}, \\
\bar{\phi}_{\eps} &=1 &&\text{in }B^\varepsilon, \\
\bar{\phi}_{\eps} &=0 &&\text{in }\R^n \setminus (\Omega_\varepsilon \cup B^\varepsilon).
\end{aligned}
\right.
\end{align*}
Then by \autoref{lem-Hopf} we have 
\begin{equation*}
\bar{\phi}_{\eps} \geq c_1 \eps^{-s} d_{\Omega}^s \quad\text{in }\Omega_{\eps/2} \cap \{ x = x_n e_n \}.
\end{equation*}
Consequently, by the regularity of $\partial \Omega$ and the Harnack inequality we know that for any $\rho \in (0,\frac{\eps}{4}]$ there exists a ball $B_{\kappa}(p_0) \subset \Omega_{\rho}$ with $\kappa \asymp \rho$ such that $\bar{\phi}_{\eps} \ge c_2 \eps^{-s} \rho^s$ in $B_{\kappa}(p_0)$, and therefore
\begin{align}
    \label{eq-hopf-T}
    \dashint_{\Omega_{\rho}} \bar{\phi}_{\eps} \d x \ge \frac{|B_{\kappa}(p_0)|}{|\Omega_{\rho}|} \dashint_{B_{\kappa}(p_0)} \bar{\phi}_{\eps} \d x \ge c_3 \eps^{-s} \rho^s.
\end{align}

Moreover, by \autoref{prop:closeness-phi} we get for any $\rho \le \frac{\varepsilon}{4}$, 
\begin{align}
\label{eq-phi-Cs}
    \dashint_{\Omega_{\rho}} |\phi_{\eps} - \bar{\phi}_{\eps}| \d x \le c_4 \rho^{s} \eps^{-s + \sigma - \delta}.
\end{align}

Combining \eqref{eq-hopf-T} and \eqref{eq-phi-Cs}, we obtain for any $\rho \le \frac{\varepsilon}{4}$, upon choosing $\eps$ so small that $c_4 \eps^{\sigma - \delta} \le \frac{c_3}{2}$,
\begin{align*}
    \dashint_{\Omega_{\rho}} \phi_{\eps} \d x \ge \dashint_{\Omega_{\rho}} \bar{\phi}_{\eps} \d x - \dashint_{\Omega_{\rho}} |\phi_{\eps} - \bar{\phi}_{\eps}| \d x \ge c_3 \eps^{-s} \rho^s - c_4 \rho^s \eps^{-s + \sigma - \delta} \ge \frac{c_3}{2} \eps^{-s}\rho^s. 
\end{align*}
Let us now fix $\eps > 0$ satisfying all previous restrictions, and denote $c_5 := \frac{c_3}{2} \eps^{-s}$. Altogether, we have
\begin{align*}
    c_{\eps} c_5 \rho^s \le \dashint_{\Omega_{\rho}} u \d x \qquad \forall \rho \leq \frac{\eps}{4}.
\end{align*}
We can assume without loss of generality (up to a normalization) that $\Vert u \Vert_{L^1_{2s}(\R^n)} + \Vert f \Vert_{L^{\infty}(\Omega_1)} = 1$. Then, from  \autoref{thm:inhom-Cs}, we know that $[u]_{C^s(\overline{\Omega_{1/2}})} \le c_6$, and hence for any $\eta \in (0,1)$
\begin{align*}
    |\Omega_{\rho}|^{-1} \int_{\Omega_{\rho} \cap \{ d_{\Omega} \le \eta \rho \} } u \d x \le  (\eta \rho) ^s [u]_{C^s(\overline{\Omega_{1/2}})} \le c_6 \eta^s \rho^s.
\end{align*}
Choosing $\eta \in (0,1)$ such that $2 c_6 \eta^s \le c_{\eps} c_5$, ($\eta$ does not depend on $\rho$) we obtain for any $\rho \le \frac{\eps}{4}$:
\begin{align*}
    \frac{c_{\eps} c_5}{2} \rho^s \le \dashint_{\Omega_{\rho}} u \d x  - \frac{c_{\eps} c_5}{2} \rho^s \le |\Omega_{\rho}|^{-1} \int_{\Omega_{\rho} \cap \{ d_{\Omega} > \eta \rho \} } u \d x \le c_7 \inf_{\Omega_{\rho} \cap \{ d_{\Omega} > \eta \rho \} } u,
\end{align*}
where we used the weak Harnack inequality in the last step and $c_7 > 0$ depends only on $n,s,\lambda,\Lambda,\eta$. Clearly, the weak Harnack inequality is applicable on any ball with radius $\eta\rho/2$ within $\Omega_{\rho} \cap \{ d_{\Omega} \ge \eta \rho \}$, since they are all interior balls.
Hence, choosing $\rho = d_{\Omega}(z)$, we immediately deduce \eqref{eq:Hopf-proof-boundary} for the point $z$, which was fixed before.
Note that when redoing the aforementioned argument for any other $z \in \Omega_{1/2} \cap \{ d_{\Omega} \le \eps/4 \}$, we can choose $\eps$ to be uniformly bounded from below by a constant $\eps_0$, depending only on $n,s,\lambda,\Lambda$, and $\Omega$. Therefore, the constant $c > 0$ in \eqref{eq:Hopf-proof-boundary} depends only on $n,s,\lambda,\Lambda$, the $C^{1,\alpha}$ radius of $\Omega$, and $\inf_{\{ d_{\Omega} \ge \eps_0 \} } u$ (the dependence on $u$ comes from $c_{\eps}$).
Finally, by the weak Harnack inequality we have
\begin{equation*}
u(x) \geq c(\varepsilon_0) c_0 \geq c(\varepsilon_0) c_0 \diam(\Omega)^{-s} d_\Omega^s(x)
\end{equation*}
for any $x \in \Omega_{1/2} \cap \{d_{\Omega} > \varepsilon_0 \}$. This concludes the proof.
\end{proof}

\section{Higher boundary regularity for homogeneous kernels via Campanato theory}
\label{sec:homCs}

The goal of this section is to prove the higher order boundary regularity for $u/d_{\Omega}^s$ from \autoref{thm-Campanato-Holder} under the assumption that the frozen kernels $K_{x_0}$ (see \eqref{eq-frozen}) are homogeneous for every $x_0$.

\subsection{Higher order Campanato estimate in the translation invariant case}

The first step is to deduce a Campanato-type estimate for solutions to the problem with respect to translation invariant homogeneous operators.

Let $\Omega \subset \R^n$ be a $C^{1,\alpha}$ domain for some $\alpha \in (0,s)$. Let $L$ be an operator with kernel $K$ and assume that $K$ satisfies \eqref{eq:Kcomp} and that for any $x_0 \in \Omega_1$ the kernel $K_{x_0}$, given by \eqref{eq-frozen}, of the frozen operator $L_{x_0}$ is homogeneous.

We need barrier functions as in Section~\ref{sec:Cs}, but we can work with simpler ones. We define $\phi := \phi_{x_0}$ to be the solution to
\begin{align*}
\left\{
\begin{aligned}
L_{x_0} \phi &= 0 &&\text{in } \Omega_1 ,\\
\phi &= g &&\text{in } \R^n \setminus \Omega_1,
\end{aligned}
\right.
\end{align*}
where $g \in C^{\infty}_c(\R^n \setminus B_1)$ is a function satisfying $0 \le g \le 1$ and $g \not\equiv 0$. Note in particular that $\phi=0$ in $B_1 \setminus \Omega$ and that by the comparison principle, \autoref{thm:main-2} (see also \cite[Propositions~2.6.4 and 2.6.6]{FeRo24}), and \cite[Proposition 2.7.8]{FeRo24}, we have
\begin{align}
\label{eq:phi-properties}
0\le \phi \le 1~~\text{in }\R^n, \quad c_1 d_{\Omega}^s \le \phi \le c_2 d_{\Omega}^s ~~ \text{in } \Omega_{3/4}, \quad\text{and}\quad [\phi/d^s_{\Omega}]_{C^{\alpha}(\overline{\Omega_{3/4}})} \le c_3.
\end{align}
The constants $c_1,c_2, c_3 > 0$ only depend on $n,s,\lambda,\Lambda$, and $\Omega$.

Having at hand the function $\phi$, we are ready to formulate the higher order Campanato estimate for translation invariant, homogeneous operators. This result is in complete analogy to \autoref{lemma:T4-inhom}.

\begin{lemma}
\label{lemma:T4}
Let $\Omega, L, K, \phi$ be given as above. Let $x_0 \in \Omega_{1/2}$ and $R \in (0, \frac{1}{4})$. If $v$ is a solution to
\begin{align*}
\left\{
\begin{aligned}
L_{x_0}v&=0 &&\text{in }\Omega_R(x_0), \\
v&=0 &&\text{in } B_R(x_0) \setminus \Omega,
\end{aligned}
\right.
\end{align*}
then for any $0<\rho \leq R$
\begin{align*}
    &\int_{\Omega_{\rho}(x_0)} \left|\frac{v}{\phi} - \left( \frac{v}{\phi} \right)_{\Omega_{\rho}(x_0)} \right| \d x \\ 
    &\leq c \left( \frac{\rho}{R} \right)^{n+\alpha} \left[ \int_{\Omega_{R}(x_0)} \left|\frac{v}{\phi} - \left( \frac{v}{\phi} \right)_{\Omega_{R}(x_0)} \right| \d x + \max\{R,d_{\Omega}(x_0)\}^{-s} R^n \mathrm{Tail}(v - \phi(v/\phi)_{\Omega_{R}(x_0)}; R, x_0) \right],
\end{align*}
where $c=c(n, s, \lambda, \Lambda, \alpha, \Omega) > 0$.
\end{lemma}

A key ingredient for the proof of \autoref{lemma:T4} is the following higher order boundary regularity result for translation invariant homogeneous operators from \cite[Proposition~2.7.8]{FeRo24}. 

\begin{proposition}
\label{lemma:T3}
Let $\Omega \subset \R^n$ be a $C^{1,\alpha}$ domain for some $\alpha \in (0,s)$.
Let $L$ be translation invariant, homogeneous, and assume that $K$ satisfies \eqref{eq:Kcomp}. Let $x_0 \in \Omega_{1/2}$ and $R \in (0, \frac{1}{4})$. Let $v$ be a solution to
\begin{align*}
\left\{
\begin{aligned}
Lv&=0 &&\text{in }\Omega_R(x_0), \\
v&=0 &&\text{in } B_R(x_0) \setminus \Omega.
\end{aligned}
\right.
\end{align*}
Then $v/d_\Omega^s \in C^\alpha_{\mathrm{loc}}(\Omega_{R}(x_0))$ and
\begin{align*}
R^{s+\alpha} [ v / d_\Omega^s ]_{C^{\alpha}(\overline{\Omega_{R/2}(x_0)})} + R^s \|v/d_\Omega^s\|_{L^\infty(\Omega_{R/2}(x_0))} \le c \left( \Vert v \Vert_{L^{\infty}(B_{R}(x_0))} + \tail(v; x_0, R) \right),
\end{align*}
where $c=c(n, s, \lambda, \Lambda, \alpha, \Omega) > 0$.
\end{proposition}

\begin{proof}[Proof of \autoref{lemma:T4}]
We may assume that $\rho \leq R/4$. Let $c_0 \in \R$ be arbitrary. Let us first consider balls $B_{\rho}(x_0)$ such that $B_{\rho}(x_0) \cap \Omega^c \not=\emptyset$. We observe from \eqref{eq-E-inf} that
\begin{align*}
\int_{\Omega_{\rho}(x_0)} \left|\frac{v}{\phi} - \left( \frac{v}{\phi} \right)_{\Omega_{\rho}(x_0)} \right| \d x &\le 2 \int_{\Omega_{\rho}(x_0)} \left|\frac{v}{\phi}(x) - \frac{v}{\phi}(x_0) \right| \d x \\
& = 2 \int_{\Omega_{\rho}(x_0)} \left|\frac{v - c_0 \phi}{\phi}(x) - \frac{v - c_0 \phi}{\phi}(x_0) \right| \d x \\
&\le c \rho^{n+\alpha} \left[ \frac{v - c_0 \phi}{\phi} \right]_{C^{\alpha}(\overline{\Omega_{R/4}(x_0)})} \\
&\le c \rho^{n+\alpha} \left( \left[ \frac{v-c_0 \phi}{d^s_{\Omega}} \right]_{C^{\alpha}(\overline{\Omega_{R/4}(x_0)})} + R^{-\alpha} \left\| \frac{v-c_0\phi}{d_\Omega^s} \right\|_{L^\infty(\Omega_{R/4}(x_0))} \right),
\end{align*}
where we used in the last step that $\phi/d_{\Omega}^s \in C^{\alpha}(\overline{\Omega_{R/4}(x_0)})$, which follows from \eqref{eq:phi-properties}.
Since $v - c_0 \phi$ satisfies all the assumptions from \autoref{lemma:T3} and \autoref{lemma:locbd-bdry}, we deduce
\begin{align*}
\int_{\Omega_{\rho}(x_0)} \left|\frac{v}{\phi} - \left( \frac{v}{\phi} \right)_{\Omega_{\rho}(x_0)} \right| \d x
&\le c \frac{\rho^{n+\alpha}}{R^{s+\alpha}} \left( \|v-c_0\phi\|_{L^\infty(\Omega_{R/2}(x_0))} + \mathrm{Tail}(v-c_0\phi; x_0, R/2) \right) \\
&\le c \frac{\rho^{n+\alpha}}{R^{s+\alpha}} \left( \dashint_{\Omega_R(x_0)} |v-c_0\phi| \d x + \mathrm{Tail}(v-c_0\phi; R,  x_0) \right) \\
&\le c \left( \frac{\rho}{R} \right)^{n+\alpha} \left( \int_{\Omega_{R}(x_0)} \left| \frac{v-c_0 \phi}{\phi} \right| \d x + R^{n-s} \mathrm{Tail}(v-c_0\phi; R,x_0) \right).
\end{align*}
Here, we also used that $\phi \le cd^s_{\Omega} \le c R^s$ in $\Omega_R(x_0)$, which follows from \eqref{eq:phi-properties}. Since $c_0 \in \R$ was arbitrary, we can choose $c_0 = (v/\phi)_{\Omega_R(x_0)}$. 
Hence, we conclude the desired estimate.

The case $B_{\rho}(x_0) \cap \Omega^c = \emptyset$ goes as in \autoref{lemma:T4-inhom}, using the interior regularity of $v-c_0\psi$ and $\psi$.
\end{proof}

\subsection{Higher regularity up to the boundary}

Our next goal is to establish a higher order Campanato-type estimate as in \autoref{lemma:T4} for solutions to non-translation invariant equations. We denote
\begin{align*}
\overline{\Psi}_{\sigma}(u;\rho) 
&:= \overline{\Psi}_{\sigma}(u;\rho,x_0) \\
&:= \int_{\Omega_{\rho}(x_0)} \left|\frac{u}{\phi} - \left( \frac{u}{\phi} \right)_{\Omega_{\rho}(x_0)} \right| \d x + \max\{\rho,d_{\Omega}(x_0)\}^{-s} \rho^n \mathrm{Tail}_{\sigma,B_1}\left(u - \phi\left(\frac{u}{\phi}\right)_{\Omega_{\rho}(x_0)}; \rho, x_0 \right).
\end{align*}
Recall that $\Phi_{\sigma}(u; \rho):=\Phi_{\sigma}(u; \rho, x_0)$ is given by \eqref{eq-excess-Phi}.

\begin{lemma}\label{lem-u-Holder}
Assume that we are in the same setting as in \autoref{thm-Campanato-Holder}. Then for any $x_0 \in \Omega_{1/2}$ and $0 < \rho \leq R \leq \frac{1}{16}$ it holds
\begin{align*}
\overline{\Psi}_\sigma(u;\rho)
&\leq c \left( \frac{\rho}{R} \right)^{n+\alpha} \overline{\Psi}_\sigma(u; R) + cR^{\sigma} \Phi_{\sigma}(u; R) + c R^{n + s -(s-p) - \frac{n}{q}} \left(\Vert u \Vert_{L^1_{2s}(\R^n)} + \Vert d_{\Omega}^{s-p} f \Vert_{L^q(\Omega_1)} \right),
\end{align*}
where $c=c(n, s, \lambda, \Lambda, \alpha, p, q, \sigma, \Omega) > 0$.
\end{lemma}

\begin{proof}
    The proof goes as in the proof of \autoref{lem-u-Holder-inhom}, with the only difference that we use \autoref{lemma:T4} and \eqref{eq:phi-properties} instead \autoref{lemma:T4-inhom} and \autoref{prop:inhom-Cs}(ii).
\end{proof}

We are now in a position to prove an almost optimal higher regularity result. Note that the regularity we obtain in the following result already exceeds the one from Section \ref{sec:Cs} for inhomogeneous kernels.

\begin{lemma}
\label{lemma:Campanato-almost-optimal-Holder}
Assume that we are in the same setting as in \autoref{thm-Campanato-Holder}. Let $0<\varepsilon<\min\{\alpha, \sigma\}$. Then, there exists $R_0 \in (0,\frac{1}{16})$, depending only on $n,s,\sigma,\lambda,\Lambda,p,q,\varepsilon$, and $\Omega$, such that for any $x_0 \in \Omega_{1/2}$ and $0 < \rho \leq R_0$ it holds
\begin{align}
\label{eq-Campanato-almost-optimal-Holder}
\overline{\Psi}_\sigma(u;\rho) \le c \rho^{n+\min\{\alpha-\varepsilon,\sigma-\varepsilon, p - \frac{n}{q}\}} \left( \|u\|_{L^1_{2s}(\R^n)} + \| d^{s-p}_{\Omega} f\|_{L^q(\Omega_1)} \right),
\end{align}
where $c = c(n, s, \lambda, \Lambda, \alpha, \sigma, p, q, \varepsilon, \Omega) > 0$. In particular, this implies
\begin{align}
\label{eq-Campanato-almost-optimal-Holder2}
\left[ \frac{u}{d_\Omega^s} \right]_{C^{\min\{\alpha-\varepsilon,\sigma-\varepsilon,p-\frac{n}{q}\}}(\overline{\Omega_{1/2}})} \leq c \left( \|u\|_{L^1_{2s}(\R^n)} + \|d^{s-p}_{\Omega} f\|_{L^q(\Omega_1)} \right),
\end{align}
where $c=c(n, s, \lambda, \Lambda, \alpha, \sigma, p, q, \varepsilon, \Omega) > 0$.
\end{lemma}

\begin{proof}
Let $R_0$ be as in \autoref{lemma:Morrey-boundary}, and fix $x_0 \in \Omega_{1/2}$ and $0<\rho\leq R\leq R_0$. We assume that 
\begin{align*}
\|u\|_{L^1_{2s}(\R^n)} + \|d^{s-p}_{\Omega} f\|_{L^q(\Omega_1)} \le 1.
\end{align*}
By following the proof of \autoref{thm:Cs-inhom-flat}, but using \autoref{lem-u-Holder} and \autoref{thm:main-2} instead of \autoref{lem-u-Holder-inhom} and \autoref{thm:Cs-eps}, we obtain
\begin{align*}
    \overline{\Psi}_{\sigma}(u; \rho) \leq c \left( \frac{\rho}{R} \right)^{n+\alpha} \overline{\Psi}_{\sigma}(u; R) + cR^{n+\sigma-\varepsilon} + c R^{n + p-\frac{n}{q}}.
\end{align*}
It thus follows from \autoref{lem-iteration} (observing that \eqref{eq-almost-incr-Psi} holds also for $\overline{\Psi}$) that for any $0<\rho<R\leq R_0$
\begin{align}\label{eq:Campanato-Psi-bar}
    \overline{\Psi}_{\sigma}(u; \rho) \leq c \left( \frac{\rho}{R} \right)^{n+\min\{\alpha-\varepsilon,\sigma-\varepsilon, p - \frac{n}{q}\}} \overline{\Psi}_{\sigma}(u; R) + c\rho^{n+\min\{\alpha-\varepsilon,\sigma-\varepsilon, p - \frac{n}{q}\}}.
\end{align}
 Since $\overline{\Psi}_{\sigma}(u; R_0) \leq c$, applying \eqref{eq:Campanato-Psi-bar} with $R = R_0$ shows the desired estimate \eqref{eq-Campanato-almost-optimal-Holder}.

Let us now prove \eqref{eq-Campanato-almost-optimal-Holder2} by replacing $u/\phi$ by $u/d_{\Omega}^s$. To do so, we observe by using \eqref{eq-E-inf} that
\begin{align}
\label{eq:Campanato-product}
\begin{split}
&\int_{\Omega_{1/2} \cap B_\rho(x_0)} \left| \frac{u}{d^s_{\Omega}} - \left( \frac{u}{d^s_{\Omega}} \right)_{\Omega_{1/2} \cap B_{\rho}(x_0)} \right| \d x \\
&\le 2\int_{\Omega_\rho(x_0)} \left| \frac{u}{\phi}\frac{\phi}{d^s_{\Omega}} - \left( \frac{u}{\phi} \right)_{\Omega_{\rho}(x_0)} \left( \frac{\phi}{d_{\Omega}^s} \right)_{\Omega_{\rho}(x_0)} \right| \d x \\
&\le \int_{\Omega_\rho(x_0)} \left| \frac{u}{\phi} - \left( \frac{u}{\phi} \right)_{\Omega_{\rho}(x_0)} \right| \left| \frac{\phi}{d^s_{\Omega}} + \left( \frac{\phi}{d^s_{\Omega}} \right)_{\Omega_{\rho}(x_0)} \right| \d x + \int_{\Omega_\rho(x_0)} \left| \frac{u}{\phi} + \left( \frac{u}{\phi} \right)_{\Omega_{\rho}(x_0)} \right| \left| \frac{\phi}{d^s_{\Omega}} - \left( \frac{\phi}{d^s_{\Omega}} \right)_{\Omega_{\rho}(x_0)} \right| \d x \\
& =: I_1 + I_2.
\end{split}
\end{align}
For $I_1$, we compute by making use of \eqref{eq:phi-properties} and \eqref{eq-Campanato-almost-optimal-Holder}:
\begin{align*}
I_1 \le 2\left \Vert \frac{\phi}{d^s_{\Omega}} \right\Vert_{L^{\infty}(\Omega_{\rho}(x_0))} \overline{\Psi}_{\sigma}(u;\rho) \le c\rho^{n+\min\{\alpha-\varepsilon,\sigma-\varepsilon, p - \frac{n}{q}\}}.
\end{align*}
Moreover, for $I_2$, we have by \autoref{thm:main-2} and \eqref{eq:phi-properties}
\begin{align*}
I_2 \le c\left \Vert \frac{u}{\phi} \right\Vert_{L^{\infty}(\Omega_{1/2})} \rho^{n+\alpha} \left[ \frac{\phi}{d^s_{\Omega}} \right]_{C^{\alpha}(\overline{\Omega_{\rho}(x_0)})} \le c \rho^{n+\alpha}.
\end{align*}
Altogether, this proves
\begin{align*}
\rho^{-n-\min\{\alpha-\varepsilon,\sigma-\varepsilon, p - \frac{n}{q}\}} \int_{\Omega_{1/2} \cap B_\rho(x_0)} \left| \frac{u}{d^s_{\Omega}} - \left( \frac{u}{d^s_{\Omega}} \right)_{\Omega_{1/2} \cap B_{\rho}(x_0)} \right| \d x  \le c.
\end{align*}
Therefore, $u/d^s_{\Omega} \in \mathcal{L}^{1, n+\min\{\alpha-\varepsilon,\sigma-\varepsilon,p - \frac{n}{q}\}}(\Omega_{1/2}) \simeq C^{\min\{\alpha-\varepsilon,\sigma-\varepsilon,p-\frac{n}{q}\}}(\overline{\Omega_{1/2}})$ and \eqref{eq-Campanato-almost-optimal-Holder2} holds.
\end{proof}

Let us now provide the proof of \autoref{thm-Campanato-Holder}.

\begin{proof}[Proof of \autoref{thm-Campanato-Holder}]
We assume that
\begin{align*}
\|u\|_{L^1_{2s}(\R^n)} + \|d_{\Omega}^{s-p} f\|_{L^q(\Omega_1)} \leq 1,
\end{align*}
and in the same spirit as in the proof of \autoref{lemma:Campanato-almost-optimal-Holder} we claim that
\begin{align}
\label{eq:suff-Campanato-optimal-Holder}
\rho^{-n-\min\{\alpha-\varepsilon,\sigma,p-\frac{n}{q}\}} \int_{\Omega_{1/2} \cap B_\rho(x_0)} \left| \frac{u}{\phi} - \left( \frac{u}{\phi} \right)_{\Omega_{1/2}\cap B_\rho(x_0)} \right| \d x \leq c
\end{align}
for all $x_0 \in \Omega_{1/2}$ and $0<\rho \leq R_0$, where $R_0$ is the constant given in \autoref{lemma:Campanato-almost-optimal-Holder}.

Let $0<\rho\leq R\leq R_0$. By \autoref{lem-u-Holder} we have that
\begin{align*}
\overline{\Psi}_\sigma(u; \rho) \leq c \left( \frac{\rho}{R} \right)^{n+\alpha} \overline{\Psi}_\sigma(u; R) + cR^{\sigma} \Phi_{\sigma}(u; R) + cR^{n+p-\frac{n}{q}}.
\end{align*}
Note that by \autoref{lemma:Campanato-almost-optimal-Holder}, $u/d_\Omega^s \in C^{\min\{\alpha-\varepsilon_0,\sigma-\varepsilon_0,p-\frac{n}{q}\}}(\overline{\Omega_{1/2}})$ for any $\eps_0 \in (0,\min\{\alpha, \sigma\})$, which shows in particular that $u/d^s_{\Omega}$ is locally bounded in $\overline{\Omega} \cap B_1$. Moreover, we recall from \eqref{eq-Campanato-almost-optimal-Holder} that
\begin{align*}
\max\{R,d_{\Omega}(x_0)\}^{-s} R^n \mathrm{Tail}_{\sigma,B_1}(u - \phi(u/\phi)_{\Omega_{R}(x_0)}; R, x_0) \le c R^{n+\min\{\alpha-\varepsilon_0,\sigma-\varepsilon_0,p-\frac{n}{q}\}}.
\end{align*}
Thus, we have since $q > \frac{n}{p}$
\begin{align*}
\Phi_{\sigma}(u; R)
&\le c R^n + c R^{n+\min\{\sigma-\varepsilon_0,\alpha-\varepsilon_0,p-\frac{n}{q}\}} + c(u/\phi)_{\Omega_R(x_0)} \max\{R,d_{\Omega}(x_0)\}^{-s} R^{n} \tail_{\sigma,B_1}(\phi;R,x_0) \\
&\le c \left( 1 + \max\{R,d_{\Omega}(x_0)\}^{-s}\tail_{\sigma,B_1}(\phi;R,x_0) \right) R^n \le c R^n,
\end{align*}
where we used in the last step that we have since for any $x \in \Omega_1 \setminus B_R(x_0)$ it holds $d_{\Omega}(x) \le |x-x_0| + d_{\Omega}(x_0)$ and $\phi \le c d_{\Omega}^s$ by \eqref{eq:phi-properties}, and $\sigma < s$:
\begin{align*}
\tail_{\sigma,B_1}(\phi;R,x_0)
&\le c R^{2s-\sigma} \int_{\Omega_1 \setminus B_R(x_0)} \frac{|x-x_0|^s+d_\Omega^s(x_0)}{|x-x_0|^{n+2s-\sigma}} \d x \le c R^s + c d^s(x_0) \le c \max\{R,d_{\Omega}(x_0)\}^s.
\end{align*}
Altogether, we have shown that
\begin{align*}
\overline{\Psi}_\sigma(u; \rho) \leq c \left( \frac{\rho}{R} \right)^{n+\alpha} \overline{\Psi}_\sigma(u; R) + c R^{n+\min \{\sigma,p-\frac{n}{q}\}}.
\end{align*}

Now we apply \autoref{lem-iteration} to conclude that for any $\eps \in (0,\alpha)$
\begin{align*}
\overline{\Psi}_\sigma(u; \rho) \leq c \left( \frac{\rho}{R} \right)^{n+ \min\{\alpha - \eps , \sigma , p - \frac{n}{q}\}} \overline{\Psi}_\sigma(u; R) + c \rho^{n + 2 \min\{\alpha - \eps,\sigma,p - \frac{n}{q}\}},
\end{align*}
and therefore \eqref{eq:suff-Campanato-optimal-Holder} holds true by the same reasoning as in the proof of \autoref{lemma:Campanato-almost-optimal-Holder}. From here, using \eqref{eq:Campanato-product}, \eqref{eq:phi-properties}, and \autoref{thm:main-2} again, we have
\begin{align*}
& \int_{\Omega_{1/2} \cap B_\rho(x_0)} \left| \frac{u}{d^s_{\Omega}} - \left( \frac{u}{d^s_{\Omega}} \right)_{\Omega_{1/2} \cap B_{\rho}(x_0)} \right| \d x \\
&\quad \le 2 \left \Vert \frac{\phi}{d^s_{\Omega}} \right\Vert_{L^{\infty}(\Omega_{\rho}(x_0))} \int_{\Omega_{1/2} \cap B_\rho(x_0)} \left| \frac{u}{\phi} - \left( \frac{u}{\phi} \right)_{\Omega_{1/2}\cap B_\rho(x_0)} \right| \d x + c\left \Vert \frac{u}{\phi} \right\Vert_{L^{\infty}(\Omega_{1/2})} \rho^{n+\alpha} \left[ \frac{\phi}{d^s_{\Omega}} \right]_{C^{\alpha}(\overline{\Omega_{\rho}(x_0)})} \\
&\quad\le c \rho^{n + 2 \min\{ \alpha - \eps,\sigma,p - \frac{n}{q} \} } + c\rho^{n+\alpha},
\end{align*}
which completes the proof after an application of Campanato's embedding.
\end{proof}

\section{Green function estimates}
\label{sec:Green}

In this section we prove the Green function estimates from \autoref{thm:main-1} and \autoref{cor:gradient-bounds}. Let us first recall the definition of the Green function from \cite{KKL23}:
\begin{definition}
\label{def:Green}
    Let $\Omega \subset \R^n$ be a bounded domain and $L$ be as in \eqref{eq-op}, \eqref{eq:Kcomp}. A measurable function $G: \R^n \times \R^n \to [0, \infty]$ is a \emph{Green function} associated to $\Omega$ and $L$ if $G(\cdot, y) \in L^1(\Omega)$ for each $y \in \Omega$, $G=0$ a.e.\ on $(\R^n \times \R^n) \setminus (\Omega \times \Omega)$, and
    \begin{align*}
        \int_{\R^n} G(x, y) \psi(x) \d x = \varphi(y)
    \end{align*}
    for every $y \in \Omega$, and $\psi \in L^{\infty}(\Omega)$, where $\varphi$ is the weak solution to 
\begin{align*}
\left\{
\begin{aligned}
L \varphi &= \psi &&\text{in } \Omega,\\
\varphi &= 0 &&\text{in } \R^n \setminus \Omega.
\end{aligned}
\right.
\end{align*}
\end{definition}

We need the following lemma.

\begin{lemma}
\label{lemma:Green-aux}
    Let $\Omega \subset \R^n$ be a bounded domain and $L$ be as in \eqref{eq-op}, \eqref{eq:Kcomp}. Then, the Green function $G$ associated to $\Omega$ and $L$ exists and it is unique. Moreover, the following hold true:
    \begin{itemize}
        \item $G$ satisfies the interior two-sided bounds \eqref{eq:int-upper} and \eqref{eq:int-lower}.
        \item $G$ is symmetric, i.e.\ $G(x,y) = G(y,x)$ for all $x,y \in \Omega$.
        \item $G(\cdot,y) \in L^1_{2s}(\R^n) \cap H^s(B_{3R/2}(x_0))$, and $G(\cdot,y)$ is a weak solution to $L G(\cdot,y) = 0$ in $\Omega_R(x_0)$ for any $y \in \Omega$ and any ball $B_R(x_0) \subset \R^n$ with $|x_0 - y| \ge 2R$.
    \end{itemize}
\end{lemma}

\begin{proof}
    The existence and uniqueness, as well as the first and second properties follow from \cite[Theorems 1.3, 1.4, 1.5, and 1.7]{KLL23}. The fact that $G(x,\cdot) \in L^1_{2s}(\R^n)$ follows from the upper bound \eqref{eq:int-upper}. 
    
    Let us fix $y \in \Omega$ and a ball $B_R(x_0) \subset \R^n$ with $|x_0-y| \geq 2R$. Let us recall the approximate Green function $G_{\rho}(\cdot,y) \in H^s_{\Omega}(\R^n)$ from \cite[Lemma 3.1]{KLL23} which solves $L G_{\rho}(\cdot,y) = \rho^{-n} \1_{B_{\rho}(y)}$ in $\Omega$, where we consider $\rho > 0$ so small that $B_{\rho}(y) \subset \Omega$ and $B_{\rho}(y) \cap B_{7R/4}(x_0) = \emptyset$. Then, by the Caccioppoli inequality (\autoref{lemma:Cacc-bdry}) and the local boundedness (\autoref{lemma:locbd-bdry}), we obtain
    \begin{align}
    \label{eq:approx-GF-energy}
        [G_{\rho}(\cdot,y)]_{H^s(B_{3R/2}(x_0))} \leq c R^{-\frac{n}{2}-s} \Vert G_{\rho}(\cdot,y) \Vert_{L^1(B_{4R/7}(x_0))} + c R^{\frac{n}{2}-s} \tail(G_{\rho}(\cdot,y);7R/4,x_0).
    \end{align}
    We can follow \cite[(4.11)-(4.16)]{KLL23} to deduce a bound for the right-hand side in \eqref{eq:approx-GF-energy} as follows:
    \begin{align*}
        \Vert G_{\rho}(\cdot,y) \Vert_{H^s(B_{3R/2}(x_0))} \le c R^{s - \frac{n}{2}},
    \end{align*}
    where $c > 0$ only depends on $n,s,\lambda,\Lambda$, but not on $\rho > 0$. Hence, there exists a subsequence $\rho_k \searrow 0$ such that $G_{\rho_k}(\cdot,y)$ converges weakly in $H^s(B_{3R/2}(x_0))$. Since we already know from \cite[Proof of Theorem 1.3]{KLL23} that $G_{\rho_k}(\cdot,y) \to G(\cdot,y)$ a.e.\ in $\R^n$, it must be $G(\cdot,y) \in H^s(B_{3R/2}(x_0))$ with the same bound. Moreover, we have that $L G(\cdot,y) = 0$ in $\Omega_{R}(x_0)$ in the weak sense by taking the equation for $G_{\rho}(\cdot,y)$ to the limit. 
\end{proof}

\subsection{Upper bounds}

In this section, we prove the upper bound in \autoref{thm:main-1}. The idea is to combine the interior upper Green function bounds from \eqref{eq:int-upper} with the boundary regularity from \autoref{thm:main-2} through a barrier and truncation argument. A nice feature of our technique is that it allows to establish slightly more general upper bounds of the following form for some $\beta \in (0,s]$
\begin{align}
\label{eq:Green-boundary}
G(x,y) \le c \left( \frac{d_{\Omega}(x)}{|x-y|} \wedge 1 \right)^{\beta}\left( \frac{d_{\Omega}(y)}{|x-y|} \wedge 1 \right)^{\beta} |x-y|^{-n+2s}, ~~ \forall x,y \in \Omega
\end{align}
in domains $\Omega$ that do not admit $C^s$ boundary regularity of solutions, but only $C^{\beta}$ regularity.

\begin{lemma}
\label{lemma:comp-princ}
Let $\Omega \subset \R^n$ be a bounded domain. Assume that $K$ satisfies \eqref{eq:Kcomp} and that there exists $\beta \in (0,s]$ such that for any $x_0 \in \partial \Omega$, $R > 0$, $f \in L^{\infty}(\Omega_{3R/4}(x_0))$, and any solution $u$ to 
\begin{align*}
\left\{
\begin{aligned}
L u &= f &&\text{in } \Omega_{3R/4}(x_0),\\
u &= 0 &&\text{in } B_{3R/4}(x_0) \setminus \Omega,
\end{aligned}
\right.
\end{align*}
it holds for some $C > 0$
\begin{align}
\label{eq:u-reg-assumption}
    [ u ]_{C^{\beta}(\overline{B_{R/2}(x_0)})} \le C R^{-\beta} \left( \Vert u \Vert_{L^{\infty}(\R^n)} + R^{2s} \Vert f \Vert_{L^{\infty}(\Omega_{3R/4}(x_0))} \right).
\end{align}
Then, \eqref{eq:Green-boundary} holds for some $c > 0$, depending only on $n,s,\lambda,\Lambda,\Omega,\beta$, and $C$.
\end{lemma}

\begin{proof}
We fix $x,y \in \Omega$ and distinguish between three cases.

Case 1: $|x-y| \le 8 \min \{ d_{\Omega}(x) , d_{\Omega}(y) \}$. In this case, the desired result follows from \eqref{eq:int-upper}.

Case 2: $\min\{ d_{\Omega}(x) , d_{\Omega}(y) \} < |x-y|/8$. We only explain how to get the estimate in case $d_{\Omega}(x) \le d_{\Omega}(y)$.
Let $x_0 \in \partial \Omega$ such that $|x-x_0| = d_{\Omega}(x)$ and set $R = |x-y|/4$. Then, $x \in B_{R/2}(x_0)$ and $y \in \R^n \setminus B_{3R}(x_0)$. Moreover, note that by \eqref{eq:int-upper} and \autoref{lemma:Green-aux}
\begin{align*}
\left\{
\begin{aligned}
L(G(\cdot,y)\mathbbm{1}_{B_{2R}(x_0)}) &\le c_1 R^{-n}=:f &&\text{in } \Omega \cap B_{R}(x_0) \text{ in the weak sense}, \\
G(\cdot,y)\mathbbm{1}_{B_{2R}(x_0)} &\le c_2 |\cdot-y|^{-n+2s} &&\text{in } \Omega \cap (B_{2R}(x_0) \setminus B_{R}(x_0))
\end{aligned}
\right.
\end{align*}
for some $c_1, c_2 > 0$, depending only on $n,s,\lambda,\Lambda$. Indeed, for any nonnegative function $\varphi \in H^s_{\Omega_R(x_0)}(\R^n)$,
\begin{align*}
&\mathcal{E}^K(G(\cdot, y) \mathbbm{1}_{B_{2R}(x_0)}, \varphi) \\
&\quad = - \mathcal{E}^K(G(\cdot, y) \mathbbm{1}_{B_{2R}(x_0)^c}, \varphi) \\
&\quad = 2\int_{\Omega_R(x_0)} \int_{B_{2R}(x_0)^c} G(w, y) \varphi(z) K(z, w) \d w \d z \\
&\quad \leq 2\Lambda c \int_{\Omega_R(x_0)} \left( \int_{B_R(y)} \frac{|w-y|^{2s-n}}{R^{n+2s}} \d w + \int_{B_{2R}(x_0)^c \setminus B_R(y)} \frac{R^{2s-n}}{|z-w|^{n+2s}} \d w \right) \varphi(z) \d z \\
&\quad \leq c_1 R^{-n} \int_{\Omega_R(x_0)} \varphi(z) \d z
\end{align*}
for some $c_1>0$.

Let us now consider $g \in C_c^{\infty}(\R^n)$ satisfying
\begin{align*}
    0 \le g \le c_2 R^{-n+2s} ~~\text{in } \R^n, \quad g = c_2 R^{-n+2s} ~~ \text{in } \Omega \setminus B_R(x_0), \quad\text{and}\quad g = 0~~ \text{in } B_{3R/4}(x_0) \setminus \Omega,
\end{align*}
and let $u$ be the solution to
\begin{align}\label{eq-Green-u}
\left\{
\begin{aligned}
L u &= f &&\text{in } \Omega_R(x_0),\\
u &= g &&\text{in } \R^n \setminus \Omega_R(x_0).
\end{aligned}
\right.
\end{align}
Then, we have by construction (note that $|x-z| \ge R$ for any $z \in \Omega \cap (B_{2R}(x_0) \setminus B_{R}(x_0))$)
\begin{align*}
\left\{
\begin{aligned}
L(G(\cdot,y)\mathbbm{1}_{B_{2R}(x_0)}) &\le Lu &&\text{in } \Omega \cap B_{R}(x_0),\\
G(\cdot,y)\mathbbm{1}_{B_{2R}(x_0)} &\le u &&\text{in } \Omega \cap (B_{2R}(x_0) \setminus B_{R}(x_0)),\\
G(\cdot,y)\mathbbm{1}_{B_{2R}(x_0)} = 0 &\le u &&\text{in } (\Omega \setminus B_{2R}(x_0)) \cup (\R^n \setminus \Omega).
\end{aligned}
\right.
\end{align*}
Thus, by the comparison principle, we obtain
\begin{align*}
G(\cdot,y)\mathbbm{1}_{B_{2R}(x_0)} \le u ~~ \text{ in } \R^n.
\end{align*}
In particular, at $x \in B_{R/2}(x_0)$, it holds by \eqref{eq:u-reg-assumption}
\begin{align*}
G(x,y) \le u(x) \le C \left( \frac{|x-x_0|}{R} \right)^{\beta} \left( \Vert g \Vert_{L^{\infty}(\R^n)} + R^{2s} \Vert f \Vert_{L^{\infty}(B_{3R/4}(x_0))} \right) \le c \left( \frac{d_{\Omega}(x)}{R} \right)^{\beta} R^{-n+2s},
\end{align*}
as desired, upon recalling the definition of $R$. Note that this estimate corresponds to \eqref{eq:Green-boundary} in case $|x-y|/8 < \max\{ d_{\Omega}(x) , d_{\Omega}(y) \}$ but is not sharp otherwise. It remains to treat the remaining case:

Case 3: $d_{\Omega}(x) \le d_{\Omega}(y) < |x-y|/8$ or $d_{\Omega}(y) \le d_{\Omega}(x) < |x-y|/8$. Again, we only discuss the first case. The proof goes in the same way as before, but observing that now we have by \eqref{eq:int-upper} and Case 2
\begin{align*}
\left\{
\begin{aligned}
L(G(\cdot,y)\mathbbm{1}_{B_{2R}(x_0)}) &\le c_3 R^{-n-\beta} d_{\Omega}^{\beta}(y) =:\tilde{f} &&\text{in } \Omega \cap B_{R}(x_0), \\
G(\cdot,y)\mathbbm{1}_{B_{2R}(x_0)} &\le c_4 |\cdot-y|^{-n+2s-\beta} d_{\Omega}^{\beta}(y) &&\text{in } \Omega \cap (B_{2R}(x_0) \setminus B_{R}(x_0))
\end{aligned}
\right.
\end{align*}
for some $c_3, c_4 > 0$, depending only on $n,s,\lambda,\Lambda$. We consider $\tilde{g} \in C_c^{\infty}(\R^n)$ with
\begin{align*}
        \tilde{g} = c_4 R^{-n+2s-\beta}d_{\Omega}^{\beta}(y) ~~ \text{ in } \Omega \setminus B_R(x_0), \qquad \tilde{g} = 0~~ \text{ in } B_{3R/4}(x_0) \setminus \Omega,
\end{align*}
and $0 \le \tilde{g} \le c_4 R^{-n+2s-\beta} d_{\Omega}^{\beta}(y)$ in $\R^n$, and let $\tilde{u}$ be the solution to \eqref{eq-Green-u} with $f$ and $g$ replaced by $\tilde{f}$ and $\tilde{g}$, respectively. Hence, by the comparison principle we have again 
\begin{align*}
G(\cdot,y)\mathbbm{1}_{B_{2R}(x_0)} \le \tilde{u} ~~ \text{ in } \R^n,
\end{align*}
and in particular, at $x \in B_{R/2}(x_0)$, it holds by \eqref{eq:u-reg-assumption}
\begin{align*}
G(x,y) \le \tilde{u}(x)
&\le C \left( \frac{|x-x_0|}{R} \right)^{\beta} \left( \Vert \tilde{g} \Vert_{L^{\infty}(\R^n)} + R^{2s} \Vert \tilde{f} \Vert_{L^{\infty}(B_{3R/4}(x_0))} \right) \\
&\le c \left( \frac{d_\Omega(x)}{R} \right)^\beta \left( \frac{d_\Omega(y)}{R} \right)^\beta R^{-n+2s},
\end{align*}
as desired.
\end{proof}

\begin{remark}
\label{remark:Green-function-general-domains}
Note that \eqref{eq:u-reg-assumption} in \autoref{lemma:comp-princ} holds true for different values of $\beta \in (0,s]$, depending on the regularity of $\partial \Omega$ and $K$ satisfying \eqref{eq:Kcomp}:
\begin{itemize}
\item[(i)] $\Omega$ measure density condition \eqref{eq:measure-density}: there is $\beta > 0$ such that \eqref{eq:u-reg-assumption} holds (see \cite{KKP16}).
\item[(ii)] $\Omega$ flat Lipschitz, $K$ satisfies \eqref{eq-K-cont}: \eqref{eq:u-reg-assumption} holds for any $\beta \in (0,s)$ (see \autoref{thm:Cs-eps}).
\item[(iii)] $\partial \Omega \in C^{1,\alpha}$, $K$ satisfies \eqref{eq-K-cont}: \eqref{eq:u-reg-assumption} holds with $\beta = s$ (see \autoref{thm:main-2}).
\end{itemize}
\end{remark}

\begin{proof}[Proof of the upper bound in \autoref{thm:main-1}]
This follows from \autoref{remark:Green-function-general-domains}(iii).
\end{proof}

\subsection{Lower bounds}

In this section we establish the lower bound in \autoref{thm:main-1}.

\begin{proof}[Proof of the lower bound in \autoref{thm:main-1}]
We fix $x,y \in \Omega$ and distinguish between three cases.

Case 1: $|x-y| \le \min\{ d_{\Omega}(x) , d_{\Omega}(y) \}$. In this case, the desired result follows from \eqref{eq:int-lower}. 

Case 2: $d_{\Omega}(x) < |x-y| \le d_{\Omega}(y)$ or $d_{\Omega}(y) < |x-y| \le d_{\Omega}(x)$. We will only explain how to proceed in the first subcase. We let $x_0 \in \partial \Omega$ be the projection of $x$ to $\partial \Omega$, i.e.\ $|x - x_0| = d_{\Omega}(x)$ and set $R = 2|x-y|$. We can find $z \in \Omega$ such that $R \le |y-z| \le 2R$, $R < |x_0 - z| \le 2R$, and $B_{\kappa R}(z) \subset \Omega \setminus B_R(x_0)$ for some $\kappa >0$, depending only on $\Omega$. Let $u$ be the solution to
\begin{align*}
\left\{
\begin{aligned}
        L u &= 0 &&\text{in } \Omega_{R/4}(x_0),\\
        u &= 1 &&\text{in } B_{\kappa R}(z),\\
        u &= 0 &&\text{in } \R^n \setminus (\Omega_{R/4}(x_0) \cup B_{\kappa R}(z)).
\end{aligned}
\right.
\end{align*}
By an application of the Hopf lemma (see \autoref{thm-hopf}) to $u$ and doing a rescaling argument in the exact same way as in \autoref{lem-Hopf}, we obtain
\begin{align}
\label{eq:appl-Hopf}
    u \ge c_0 R^{-s} d_{\Omega}^s ~~ \text{ in } \Omega_{R/8}(x_0)
\end{align}
for some $c_0 > 0$, depending only on $n,s,\lambda,\Lambda,\alpha$, and $\Omega$. Moreover, by \eqref{eq:int-lower} and by construction of $u$,
\begin{align*}
    G(\cdot,y) \ge c R^{-n+2s} = c R^{-n+2s} u ~~ \text{ in } B_{\kappa R}(z).
\end{align*}
Thus, since $L G(\cdot,y) = 0$ in $\Omega_{R/4}(x_0)$ by \autoref{lemma:Green-aux}, the comparison principle implies that $G(\cdot,y) \ge c R^{-n+2s} u$ in $\R^n$, and by combination with \eqref{eq:appl-Hopf},
\begin{align*}
    G(\cdot,y) \ge c R^{-n+2s} u \ge cc_0 R^{-n+s} d^s_{\Omega} ~~ \text{ in } \Omega_{R/8}(x_0).
\end{align*}
If $x \in \Omega_{R/8}(x_0)$, then this implies in particular that $G(x, y) \geq cc_0 R^{-n+s}d_\Omega^s(x)$. Otherwise, since $\partial \Omega \in C^{1,\alpha}$, we can find a ball $B_{\eta R}(x^{\ast}) \subset \Omega_{R/8}(x_0)$ with $R/16 \le d_{\Omega} \le R/8$ in $B_{\eta R}(x^{\ast})$ for some $\eta \in (0,1/16)$, depending only on $\Omega$, and \eqref{eq:appl-Hopf} implies $u \ge c_0/16^s$ in $B_{\eta R}(x^{\ast})$. Then, since $|x^{\ast} - x| \asymp R$, by a standard (weak) Harnack chain argument, we deduce that $u(x) \ge c_1$ for some $c_1 > 0$, depending only on $n,s,\lambda,\Lambda,c_0,\Omega$. Thus $G(x,y) \ge c c_1 R^{-n+2s} \ge c c_1 R^{-n+s} d_{\Omega}^s(x)$, as desired.

Case 3: $d_{\Omega}(x) \le d_{\Omega}(y) < |x-y|$ and $d_{\Omega}(y) \le d_{\Omega}(x) < |x-y|$. Again, we only explain how to proceed in the first subcase. We let $x_0,R,z,\kappa,u$ be as in Step 2, and observe that in particular, \eqref{eq:appl-Hopf} still holds for $u$. Moreover, by Case 2 (and the Harnack inequality applied to $G$), we have
\begin{align*}
    G(\cdot,y) \ge c R^{-n+s} d_{\Omega}^s(y) = c R^{-n+s} d_{\Omega}^s(y) u ~~ \text{ in } B_{\kappa R/2}(z). 
\end{align*}
Thus, as before, we conclude by the comparison principle
\begin{align*}
    G(\cdot,y) \ge c R^{-n} d_{\Omega}^s(y) d_{\Omega}^s ~~ \text{ in } \Omega_{R/8}(x_0), 
\end{align*}
which implies the desired result also for $x$ by the same arguments as in Case 2.
\end{proof}

\begin{remark}
    Note that weaker versions of the Hopf lemma might hold in less regular domains. In that case, we could get lower Green function bounds with factor $(1 \wedge \frac{d_{\Omega}(x)}{|x-y|})^{\beta}$ for some $\beta \geq s$.
\end{remark}

\subsection{Poisson kernel estimates}

Given a domain $\Omega \subset \R^n$ and $g \in L^{1}_{2s}(\R^n \setminus \Omega)$, solutions to the Dirichlet problem
\begin{align*}
\left\{
\begin{aligned}
    L u &= 0 &&\text{in } \Omega,\\
    u &= g &&\text{in } \R^n \setminus \Omega
\end{aligned}
\right.
\end{align*}
have the following representation formula
\begin{align*}
    u(x) = \int_{\R^n \setminus \Omega} g(z) P(x,z) \d z,
\end{align*}
where $P : \Omega \times (\R^n \setminus \Omega) \to [0,\infty]$ denotes the Poisson kernel associated to $L$ and $\Omega$. The following is an easy consequence of the nonlocal Gauss--Green formula (see \cite{Buc16})
\begin{align}
\label{eq:Poisson-rep}
    P(x,z) = \int_{\Omega} G(x,y) K(z,y) \d y \quad \forall x \in \Omega,~~ z \in \R^n \setminus \Omega.
\end{align}
Hence, as an immediate application of \autoref{thm:main-1}, we obtain two-sided Poisson kernel estimates. Note that this result was previously only known in $C^{1,1}$ domains (even for the fractional Laplacian, see \cite{ChSo98}) and under additional structural assumptions on the coefficients.

\begin{corollary}
    \label{cor:Poisson}
     Let $s \in (0,1)$, $\alpha,\sigma \in (0,s)$. Let $\Omega \subset \R^n$ be a $C^{1,\alpha}$ domain and let $L,K, \lambda, \Lambda$ be as in \eqref{eq-op}, \eqref{eq:Kcomp}. Assume that \eqref{eq-K-cont} holds with $\mathcal{A} \Supset \Omega$. Let $P$ be the Poisson kernel associated with $L,\Omega$. Then, it holds for any $x \in \Omega$ and $z \in \R^n \setminus \Omega$
     \begin{align*}
         c^{-1} \frac{d^s_{\Omega}(x)}{d^s_{\Omega}(z)(1 + d_{\Omega}(z))^s}  \le \frac{P(x,y)}{|x-y|^{-n}} \le c \frac{d^s_{\Omega}(x)}{d^s_{\Omega}(z)(1 + d_{\Omega}(z))^s} ,
     \end{align*}
     where $c \geq 1$ depends only on $n,s,\lambda,\Lambda,\sigma,\alpha,\Omega$.
\end{corollary}

\begin{proof}
    The proof goes by combination of \eqref{eq:Poisson-rep} and \autoref{thm:main-1} and following the computations in the proof of \cite[Theorem 1.5]{ChSo98}.
\end{proof}

\subsection{Gradient estimates}

In this section we provide the proof of the gradient estimates from \autoref{cor:gradient-bounds} by using an interior gradient estimate and combining them with \autoref{thm:main-1}.

The first main ingredient is the following higher order interior regularity estimate.

\begin{lemma}\cite[Theorem~1.9]{KNS22}
\label{lemma:int-gradient-est}
Let $s \in (\frac{1}{2},1)$, $\sigma \in (0,s)$. Assume that $K$ satisfies \eqref{eq:Kcomp}, and \eqref{eq-K-cont} with $\mathcal{A} = B_R(x_0)$.
Then, for any $x_0 \in \R^n$, $R > 0$ and any solution $u$ to
\begin{align*}
L u = 0 ~~ \text{ in } B_R(x_0),
\end{align*}
it holds
\begin{align*}
\Vert \nabla u \Vert_{L^{\infty}(B_{R/2}(x_0))} \le c R^{-1} \left( R^{-n} \Vert u \Vert_{L^{1}(B_{R}(x_0))} + \tail(u;R,x_0) \right),
\end{align*}
where $c > 0$ depends only on $n,s,\lambda,\Lambda$, and $\sigma$.
\end{lemma}

\begin{remark}
\label{remark:Green-gradient}
The only reason why we exclude the case $s \in (0,\frac{1}{2}]$ from \autoref{cor:gradient-bounds} is due to the lack of a suitable analog of \autoref{lemma:int-gradient-est} in the literature. We believe that in case $s \leq \frac{1}{2}$, one can establish \autoref{lemma:int-gradient-est} for $K$ satisfying \eqref{eq:Kcomp}, and \eqref{eq-K-cont} with $\mathcal{A} = \R^n$, and in addition that for some $\gamma>1-2s$:
\begin{align}
\label{eq:higher-reg-K}
    \sup_{x \in \R^n} \sup_{y \in B_{2\rho}(x) \setminus B_{\rho}(x)} |K(x,h+z) - K(x,z)| \le \Lambda |h|^{\gamma} \rho^{-n-2s-\gamma} \quad \forall \rho > 0 \text{ and }|h| \le \rho.
\end{align}
Proving it rigorously would be well beyond the scope of this paper. Hence, we only provide a sketch of the proof. We follow the strategy in \cite[Theorem 4.2]{FeRo24b} in the case $\beta = 2s + \gamma > 1$. Using the pointwise assumptions \eqref{eq-K-cont} and \eqref{eq:higher-reg-K} on $K$ instead of the integrated assumptions in \cite{FeRo24b}, we obtain the desired estimate for $u \in C^{2s + \gamma}(B_R(x_0)) \cap L^1_{2s}(\R^n)$ with $\tail(u;R,x_0)$ instead of the weighted $L^{\infty}$ norm. Note that the global nonlocal energy from \cite[Definition 4.1]{FeRo24b} is not required to be finite for this proof, since it is based on rewriting $L$ as a nonsymmetric operator in non-divergence form, which can be evaluated in the classical sense. Finally, we can drop the assumption $u \in C^{2s + \gamma}(B_R(x_0))$ by an approximation argument, which goes in a similar way as in \cite[Chapter 3.4.1]{FeRo24} but adapted to divergence form operators (see also \cite[Section 3]{Fer24}).
\end{remark}

A combination of the upper Green function estimate from \autoref{thm:main-1} and \autoref{lemma:int-gradient-est} implies the upper gradient Green function estimate from \autoref{cor:gradient-bounds}.

\begin{proof}[Proof of \autoref{cor:gradient-bounds}]
First, let us assume that $|x-y| \le d_{\Omega}(x)$. Then, an application of \autoref{lemma:int-gradient-est} with $u := G(\cdot,y)$, $R := |x-y|/2$, and $x_0 := x$ (this is possible by \autoref{lemma:Green-aux}) yields
\begin{align*}
|\nabla_x G(x,y)|
&\le c |x-y|^{-1} \left( |x-y|^{-n} \Vert G(\cdot,y) \Vert_{L^{1}(B_{|x-y|/2}(x))} + \tail(G(\cdot,y),|x-y|/2, x) \right) \\
&\le c |x-y|^{-1} \left( |x-y|^{-n} \Vert G(\cdot,y) \Vert_{L^{1}(B_{2|x-y|}(x))} + \tail(G(\cdot,y),2|x-y|, x) \right).
\end{align*}
Clearly, using \eqref{eq:Green-boundary} with $\beta=s$
\begin{align*}
|x-y|^{-n} \Vert G(\cdot,y) \Vert_{L^{1}(B_{2|x-y|}(x))} \le |x-y|^{-n} \Vert G(\cdot,y) \Vert_{L^{1}(B_{3|x-y|}(y))} \le c \left( 1 \wedge \frac{d_{\Omega}(y)}{|x-y|} \right)^s |x-y|^{-n+2s}.
\end{align*}
Moreover, by using \eqref{eq:Green-boundary} with $\beta=s$ again and observing that $|z-y| \ge |z-x| - |x-y| \ge |x-y|$ for $z \in \Omega \setminus B_{2|x-y|}(x)$, we have
\begin{align*}
    \tail(G(\cdot,y),2|x-y|,x)
    &\le c |x-y|^{2s} \int_{\Omega \setminus B_{2|x-y|}(x)} \left( \frac{1}{|z-y|^{n-2s}} \wedge \frac{d_\Omega^s(y)}{|z-y|^{n-s}} \right) |x-z|^{-n-2s} \d z \\
    &\leq c |x-y|^{2s} \left( \frac{1}{|x-y|^{n-2s}} \wedge \frac{d_\Omega^s(y)}{|x-y|^{n-s}} \right) \int_{\Omega \setminus B_{2|x-y|}(x)} |x-z|^{-n-2s} \d z \\
    &\leq c\left( 1 \wedge \frac{d_\Omega(y)}{|x-y|} \right)^s |x-y|^{-n+2s}.
\end{align*}
Altogether, this shows
\begin{align*}
|\nabla_x G(x,y)| \le c \left( 1 \wedge \frac{d_{\Omega}(y)}{|x-y|} \right)^s |x-y|^{-n+2s-1} ~~ \forall x,y \in \Omega ~~ \text{ with } |x-y| \le d_{\Omega}(x).
\end{align*}
Next, let us consider the case $|x-y|>d_{\Omega}(x)$. In this case, we apply \autoref{lemma:int-gradient-est} with $u := G(\cdot,y)$, $R := d_{\Omega}(x)/2$, and $x_0 := x$. This yields
\begin{align*}
|\nabla_x G(x,y)| \le c (d_{\Omega}(x))^{-1} \left( \Vert G(\cdot,y) \Vert_{L^{\infty}(B_{d_{\Omega}(x)/2}(x))} + \tail(G(\cdot,y),d_{\Omega}(x)/2,x) \right).
\end{align*}
We have by \eqref{eq:Green-boundary}:
\vspace{-0.3cm}
\begin{align*}
\Vert G(\cdot,y) \Vert_{L^{\infty}(B_{d_{\Omega}(x)/2}(x))} \le c d_{\Omega}^s(x) \left( 1 \wedge \frac{d_{\Omega}(y)}{|x-y|} \right)^s |x-y|^{-n+s}.
\end{align*}
Moreover, 
\begin{align*}
\tail(G(\cdot,y), d_{\Omega}(x)/2 , x) &\le d_{\Omega}^{2s}(x)\int_{\R^n \setminus B_{|x-y|/2}(x)} G(z,y) |x-z|^{-n-2s} \d z \\
&\quad + d_{\Omega}^{2s}(x) \int_{B_{|x-y|/2}(x) \setminus B_{d_{\Omega}}(x)/2} G(z,y) |x-z|^{-n-2s} \d z =: T_1 + T_2.
\end{align*}
For $T_1$, we follow the proof of \cite[(4.14)-(4.16)]{KKL23} but use \eqref{eq:Green-boundary} instead of \eqref{eq:int-upper} to obtain
\begin{align*}
T_1 \le cd_{\Omega}^{2s}(x) \left( 1 \wedge \frac{d_{\Omega}(y)}{|x-y|} \right)^s |x-y|^{-n} \le c d_{\Omega}^{s}(x) \left( 1 \wedge \frac{d_{\Omega}(y)}{|x-y|} \right)^s |x-y|^{-n + s}.
\end{align*}
For $T_2$, we compute, in case $d_{\Omega}(y) \ge |x-y|$, using that $|x-y| \le 2|z-y|$ and $d_{\Omega}(z) \le d_{\Omega}(x) + |x-z|$ for any $z \in B_{|x-y|/2}(x) \setminus B_{d_{\Omega}}(x)/2$:
\begin{align*}
T_2 \le c d_{\Omega}^{2s}(x) |x-y|^{-n+s} \int_{\R^n \setminus B_{d_{\Omega}(x)/2}(x)} \frac{d_{\Omega}^s(x)+|x-z|^s}{|x-z|^{n+2s}} \d z \le c d^{s}_{\Omega}(x) |x-y|^{-n+s}.
\end{align*}
A similar computation as before yields in case $d_{\Omega}(y) < |x-y|$ for $T_2$:
\begin{align*}
    T_2 \le c d_{\Omega}^{2s}(x) d_{\Omega}^s(y) |x-y|^{-n} \int_{\R^n \setminus B_{d_{\Omega}(x)/2}(x)} \frac{d_{\Omega}^s(x) + |x-z|^s}{|x-z|^{n+2s}} \d z \le c d^{s}_{\Omega}(x)d^s_{\Omega}(y) |x-y|^{-n}.
\end{align*}
Altogether, we have proved
\begin{align*}
    |\nabla_x G(x,y)| \le c d_{\Omega}^{s-1}(x) \left( 1 \wedge \frac{d_{\Omega}(y)}{|x-y|} \right)^s |x-y|^{-n+s} ~~ \forall x,y \in \Omega ~~ \text{ with } |x-y| > d_{\Omega}(x).
\end{align*}
\end{proof}

\newcommand{\etalchar}[1]{$^{#1}$}

% \bibliographystyle{alpha}
% \bibliography{literature}

\end{document}